\documentclass[11pt]{amsart} 

\RequirePackage{doi}
\usepackage{hyperref}

\usepackage{xfrac}
\newcommand{\half}{{\sfrac{1\!}{2}}}

\usepackage{amsmath, amssymb, amsfonts, amsbsy, amsthm, latexsym, stmaryrd, mathtools, enumerate
}
\usepackage[marginpar=2cm,ignoremp,margin=3cm]{geometry}
\usepackage{graphicx,float}
\usepackage[mathscr]{eucal}
\usepackage{wrapfig}
\usepackage{todonotes}
\allowdisplaybreaks

\usepackage{tikz}
\usepackage{tikz-cd}
\usetikzlibrary{calc,positioning}
\usetikzlibrary{matrix}
\usetikzlibrary{decorations.pathreplacing}

\tikzstyle{overbrace text style}=[font=\tiny, above, pos=.5, yshift=3mm]
\tikzstyle{overbrace style}=[decorate,decoration={brace,raise=2mm,amplitude=3pt}]
\tikzstyle{underbrace style}=[decorate,decoration={brace,raise=2mm,amplitude=5pt,mirror},color=gray]
\tikzstyle{underbrace text style}=[font=\tiny, below, pos=.5, yshift=-3mm]

\tikzset{
	bracket/.default=0.1cm,
	bracket/.style={
		to path={
			(\tikztostart) -- ([yshift=#1]\tikztostart) -- ([yshift=#1]\tikztotarget) \tikztonodes -- (\tikztotarget)
		}
	},
	ubracket/.default=0.1cm,
	ubracket/.style={
		to path={
			(\tikztostart) -- ([yshift=-#1]\tikztostart) -- ([yshift=-#1]\tikztotarget) \tikztonodes -- (\tikztotarget)
		}
	}
}

\def\equationautorefname~#1\null{Equation~(#1)\null}

\usepackage[mathscr]{eucal}
\usepackage{url}
\usepackage{shuffle}
\DeclareFontFamily{U}{shuffle}{}
\DeclareFontShape{U}{shuffle}{m}{n}{ <-8>shuffle7 <8->shuffle10}{}

\def\sh{\shuffle}

\usepackage{thmtools}

\declaretheorem[
style=plain,
name=Theorem,
numberwithin=section,
refname={Theorem,Theorems},
Refname={Theorem,Theorems}
]{thm}

\declaretheorem[
style=plain,
name=Lemma,
numberlike=thm,
refname={Lemma,Lemmas},
Refname={Lemma,Lemmas}
]{lem}
\declaretheorem[
style=plain,
name=Proposition,
numberlike=thm,
refname={Proposition,Propositions},
Refname={Proposition,Propositions}
]{prop}
\declaretheorem[
style=plain,
name=Corollary,
numberlike=thm,
refname={Corollary,Corollaries},
Refname={Corollary,Corollaries}
]{cor}

\declaretheorem[
style=definition,
name=Properties,
numberlike=thm,
refname={Properties,Properties},
Refname={Properties,Properties},
]{properties}

\declaretheorem[
style=plain,
name=Conjecture,
numberlike=thm,
refname={Conjecture,Conjectures},
Refname={Conjecture,Conjectures},
]{conj}
\declaretheorem[
style=definition,
name=Remark,
numberlike=thm,
refname={Remark,Remarks},
Refname={Remark,Remarks},
]{remark}

\numberwithin{equation}{section}

\def\decon{{\mathrm{decon}}}
\def\C{{\mathbb{C}}}
\def\Z{{\mathbb{Z}}}
\def\Q{{\mathbb{Q}}}
\def\P{{\mathbb{P}}}

\DeclareMathOperator{\SL}{SL}

\def\degb{{\deg_\mathcal{B}}}
\def\zm{\zeta^{\mathfrak{m}}}

\def\z{\zeta}
\DeclareMathOperator{\I}{I}

\DeclareMathOperator{\im}{im}
\def\il{{\text{I}^{\mathfrak{l}}}}
\def\imot{{\I^\mathfrak{m}}}

\def\zl{{\zeta^{\mathfrak{l}}}}
\def\ia{{\text{I}^{\mathfrak{a}}}}

\def\ibl{{\text{I}^\mathfrak{bl}}}

\def\bg{{\mathfrak{bg}}}
\def\rbg{{\mathfrak{rbg}}}
\def\gmot{{\mathfrak{g}^\mathfrak{m}}}

\DeclareMathOperator{\gr}{gr}
\def\qpoly{{\Q\langle e_0,e_1\rangle}}

\DeclareMathOperator{\Sh}{Sh}

\DeclareMathOperator{\Lie}{Lie}
\DeclareMathOperator{\id}{id}

\DeclareMathOperator{\sgn}{sgn}

\DeclareMathOperator{\per}{per}
\newcommand{\eps}{\varepsilon}
\DeclareMathOperator{\Tr}{Tr}

\def\proj{{\P^1\setminus\{0,1,\infty\}}}

\newcommand{\lmot}{\mathfrak{l}}
\newcommand{\mot}{\mathfrak{m}}
\newcommand{\ii}{\mathrm{i}}

\newcommand{\ev}{\mathrm{ev}}
\newcommand{\od}{\mathrm{od}}
\newcommand{\alt}{\mathrm{alt}}

\newcommand{\abs}[1]{\left|#1\right|}
\DeclareMathOperator{\Li}{Li}

\usepackage[normalem]{ulem}
\renewcommand{\vec}[1]{\mathbf{\uline{#1}}}

\DeclareMathOperator{\Vect}{Vec}
\DeclareMathOperator{\MT}{MT}
\DeclareMathOperator{\Spec}{Spec}
\DeclareMathOperator{\Isom}{Isom}
\DeclareMathOperator{\Gr}{Gr}

\newcommand{\diffrbg}[1]{\frac{\partial^4{#1}}{\partial x_1^4}+\frac{\partial^4{#1}}{\partial x_2^4}+\frac{\partial^4{#1}}{\partial x_3^4}-2\frac{\partial^4{#1}}{\partial x_1^2\partial x_2^2}-2\frac{\partial^4{#1}}{\partial x_2^2\partial x_3^2}-2\frac{\partial^4{#1}}{\partial x_3^2\partial x_1^2}}

\DeclareMathOperator{\D}{D}

\let\overlineO\overline
\renewcommand{\overline}[1]{\overlineO{\mathclap{\phantom{I}}#1}}

\makeatletter
\@namedef{subjclassname@2020}{%
	\textup{2020} Mathematics Subject Classification}
\makeatother

\newcommand{\shortlong}[2]{#2}

\begin{document}

\title[Evaluation of \( \zeta(2,\ldots,2,4,2,\ldots,2) \) and period polynomial relations]{Evaluation of \( \zeta(2,\ldots,2,4,2,\ldots,2) \) and \\ period polynomial relations}

\author{Steven Charlton}
\address{Max Planck Institute for Mathematics, Vivatsgasse 7, Bonn 53111, Germany}
\email{charlton@mpim-bonn.mpg.de}

\author{Adam Keilthy}
\address{Department of Mathematical Sciences - Chalmers University of Technology and University of Gothenburg
	SE-412 96 Gothenburg, Sweden}
\email{keilthy@chalmers.se}


\date{\today}

\begin{abstract}
In studying the depth filtration on multiple zeta values, difficulties quickly arise due to a disparity between it and the coradical filtration \cite{BrownDepthGraded21}. In particular, there are additional relations in the depth graded algebra coming from period polynomials of cusp forms for $\SL_2(\Z)$. In contrast, a simple combinatorial filtration, the block filtration \cite{CharltonBlock21,KeilthyBlock21} is known to agree with the coradical filtration, and so there is no similar defect in the associated graded. However, via an explicit evaluation of \(\zeta(2,\ldots,2,4,2,\ldots,2)\) as a polynomial in double zeta values, we derive these period polynomial relations as a consequence of an intrinsic symmetry of block graded multiple zeta values in block degree 2. In deriving this evaluation, we find a Galois descent of certain alternating double zeta values to classical double zeta values, which we then apply to give an evaluation of the multiple $t$ values \cite{HoffmanOdd19} \(t(2\ell,2k)\) in terms of classical double zeta values.
\end{abstract}

\maketitle

\setcounter{tocdepth}{1}
\tableofcontents
\setcounter{tocdepth}{2}

\section{Introduction}

For any tuple $(k_1,k_2,\ldots,k_r)$ of positive integers with $k_r\geq 2$, we may define a multivariable analogue of the Riemann zeta values, called a multiple zeta value (MZV) of weight $k_1+\cdots+k_r$ and depth $r$, by
\[\z(k_1,k_2,\ldots,k_r)\coloneqq\sum_{0<n_1<n_2<\cdots<n_r}\frac{1}{n_1^{k_1}n_2^{k_2}\cdots n_r^{k_r}} \,.\]
These numbers arise naturally in many areas of mathematics and mathematical physics, including in connection to associators \cite{LeMurakamiMZVAssoc,RacinetThesis}, Feynman amplitudes \cite{BroadhurstKConj96}, and as periods of mixed Tate motives \cite{BrownMTM12}. Unlike single zeta values, multiple zeta values have a rich algebraic structure, the study of which goes back to Euler. Many families of relations, such as the associator relations \cite{LeMurakamiMZVAssoc}, the double shuffle relations \cite{RacinetDoubleShuff}, and the confluence relations \cite{HIROSEConfluence}, are conjectured to exhaust all relations among MZVs. However, this is incredibly challenging and encompasses still-open questions such as the transcendence of $\z(2k+1)$. 

One approach to make this more manageable is to consider instead motivic multiple zeta values. Via their connection to mixed Tate motives, MZVs may be lifted to formal, algebraic objects, only satisfying relations coming from the geometry of $\proj$ \cite{BrownMTM12}. In this setting, much more is known: the ring $\mathcal{H}$ of motivic MZVs are known to be graded by weight, with weight graded dimensions $d_n$ given by
\[\sum_{n\geq 0}d_nx^n=\frac{1}{1-x^2-x^3} \,.\]
Motivic multiple zeta values have an explicit basis \cite{BrownMTM12}, given by the Hoffman zeta values
\[\{\zm(k_1,\ldots,k_r)\mid k_1,\ldots,k_r\in\{2,3\}\} \,.\]
However, the question of providing a complete set of relations remains an open problem.

One approach to describing all (motivic) relations among MZVs is to consider relations the associated graded algebra with respect to the depth filtration
\[\mathcal{D}_n\mathcal{H}=\langle \zm(k_1,\ldots,k_r)\mid r\leq n\rangle_\Q \,.\]
Relations in $\gr_\bullet^\mathcal{D}\mathcal{H}$ are much simpler, with the stuffle product reducing to a simple shuffle product. However, this introduces additional relations \cite{GKZ06},
\[14\zm(3,9)+75\zm(5,7)+84\zm(7,5)\equiv 0 \pmod{\text{lower depth}} \] 
and the associated Lie algebra of relations is no longer free \cite{Ihara02}
\[\{\sigma_3,\sigma_9\} - 3\{\sigma_5,\sigma_7\} \equiv 0 \pmod{\text{terms of higher depth}} \,.\]
In particular, there are a family of such quadratic relations, arising from period polynomials of cusp forms \cite{GKZ06}. Both the relations among multiple zeta values and among elements of the motivic Lie algebra are commonly referred to as the \emph{period polynomial relations}. It is conjectured that these relations determine all additional relations among depth graded multiple zeta values -- that is to say the the associated Lie algebra of relations is the quotient of a free Lie algebra by the idea generated by these quadratic relations.

\begin{conj}[Broadhurst-Kreimer, \cite{BroadhurstKConj96}]\label{conj:depthzetadim}
The generating series for the dimension of the depth graded multiple zeta values is given by
\[BK(x,y)\coloneqq\frac{1}{1-O(x)y+S(x)y^2-S(x)y^4} \,, \]
where
\[O(x)\coloneqq\frac{x^3}{1-x^2} \,, \]
and
\[S(x)=\frac{x^{12}}{(1-x^4)(1-x^6)}\]
is the generating function for the space of cusp forms of weight $n$ for the full modular group. That is to say that the number of linearly independent depth graded multiple zeta values of weight $n$ and depth $d$ is given by the coefficient of $x^ny^d$ of $BK(x,y)$, and that these dimensions are determined by the relations coming from cusp forms.
\end{conj}

However, a proof of this remains out of sight, and these additional relations make using the depth graded Lie algebra to conclude statements about ungraded MZVs challenging.

An alternative approach, first explored in \cite{KeilthyThesis20,KeilthyBlock21} and based on results in \cite{CharltonThesis,CharltonBlock21}, is to consider the so-called block filtration. This filtration provides a simple description of the coradical filtration associated to the motivic coaction in terms of a combinatorial degree function. Specifically
\[\mathcal{B}_n\mathcal{H}=\langle \zm(w)\mid \degb(w)\leq n\rangle_\Q\]
where $\degb(w)$ counts the number of subsequences $e_ie_i$ in $e_0we_1$. In \cite{KeilthyBlock21}, we see that in the associated graded algebra with respect to the block filtration, there are no additional relations, and furthermore that a complete set of relations can be given in low block degree. One might then ask how the period polynomial relations manifest in this setting.

\begin{lem}\label{lem:depthsubblock}
The depth filtration is a subfiltration of the block filtration:
\[ \mathcal{D}_n\mathcal{H}\subset \mathcal{B}_n\mathcal{H}\,.\]
\end{lem}
\begin{proof}
First note that the depth filtration is motivic:
\[\Delta\mathcal{D}_n\mathcal{H}\subset\sum_{i+j=n}\mathcal{D}_i\mathcal{A}\otimes\mathcal{D}_j\mathcal{H}\,.\]
As such, since the block filtration is equal to the coradical filtration, it suffices to show that $\mathcal{D}_1\mathcal{H}\subset\mathcal{B}_1\mathcal{H}$. This is an immediate consequence of Lemma 3.2 \cite{BrownMTM12}.
\end{proof}

As depth is a subfiltration of the block filtration, it is clear the we should be able to express double zeta values in terms of block degree 2 zeta values, and hence that all block graded relations among them, modulo products, should be determined by relations describing $\bg$, the associated Lie algebra of relations among block graded MZVs. However, \autoref{lem:depthsubblock} and \autoref{lem:blockparity} implies that, in block degree two and even weight, relations among multiple zeta values modulo terms of lower block degree are genuine relations modulo products. Thus, the period polynomial relations, modulo products, should arise as a consequence of the relations among block graded MZVs introduced in \cite{KeilthyBlock21}.

And indeed, this seems to be the case. The following is a consequence of \autoref{lem:eval:z2242prodmot}, and allows us to show that relations among double zeta values of even weight are encoded by certain explicit polynomials in $\Q[x_1,x_2,x_3]$.

\begin{cor}\label{eval:doublez2242}
Modulo products, the following holds for any $0\leq 2a\leq n$,
\begin{align*}
\sum_{i=a}^{n-a}\zl(\{2\}^i,4,\{2\}^{n-i}) & {} = 4(-1)^{n+1}\sum_{i=a}^{n-a} \zl(2i+3,2n-2i+1) \\
& {} = 4(-1)^n\zl(2a+1,2n-2a+3)\,.
\end{align*}
\begin{proof}
Letting $n\coloneqq a+b$ in \autoref{lem:eval:z2242prodmot}
, we have
\begin{align*}
\zl(\{2\}^a,4,\{2\}^b) ={}  &4(-1)^{n} \bigg[ {-}\zl(2a+2,2b+2) - \zl(2a+3,2b+1)\\
&{} +\sum_{j=1}^{2n+3} 2^{j-4-2n}\left( \! \binom{2n+3-j}{2b+1}-\binom{2n+3-j}{2a+1} \! \right)\zl(j,2n+4-j) \bigg] \,.
\end{align*}
Noting that both  $\zl(2a+2,2b+2)$ and 
\[\binom{2n+3-j}{2b+1}-\binom{2n+3-j}{2a+1}\]
are anti-symmetric in $a$ and $b$, the result immediately follows.
\end{proof}
\end{cor}

From this, it is possible to deduce that dimension of double zeta values of weight $2n+2$ modulo products is bounded above by $\lfloor\frac{n}{3}\rfloor$, which is precisely the dimension predicted by \autoref{conj:depthzetadim}. As the modulo products version of \autoref{conj:depthzetadim} is known to hold in depth two \cite{SchnepsPoisson06}, we must have that all period polynomial relations can be written in terms of the block relations, defined in \autoref{sec:blockdmzv}, and thus \autoref{prop:alldouble} holds.

\renewcommand{\shortlong}[2]{#1}
\begin{restatable*}{prop}{propalldbl}
\label{prop:alldouble}
All relations among double zeta values of weight $2n+2$ modulo products are determined by \eqref{rel1} and \eqref{rel3} via \autoref{eval:doublez2242}.
\end{restatable*}
\renewcommand{\shortlong}[2]{#2}
Indeed, using a computer one can easily write the period polynomial relations as linear combinations of relations coming from the dihedral symmetry of Section 8 of \cite{KeilthyBlock21}. A more explicit connection is given in \autoref{prop:kernel-period-isomorphism}.  \medskip

The structure of this paper is as follows. We first will briefly reminder readers of the motivic formalism, and in particular the use of the motivic coaction to deduce relations. We then describe the block filtration and review several of the results of \cite{KeilthyBlock21}. In particular, we will introduce the block dihedral symmetry and the necessary framework to discuss it.

In \autoref{sec:blockdmzv}, we then apply these results, along with a number of new evaluations to conclude that the period polynomial relations are a consequence of this block dihedral symmetry in block degree 2.  The remainder of the paper is then dedicated to the necessary technical results needed for this section. Specifically, an evaluation of $\z(\{2\}^a,4,\{2\}^b)$ in terms of double zeta values\footnote{Computer readable versions of the full evaluations from \autoref{thm:eval:zs2242} and \autoref{thm:eval:z2242} in \autoref{sec:num2242} are attached to the arXiv submission,  as plain text files in \texttt{Mathematica} syntax and in \texttt{pari/gp} syntax.}, and a computation of the dimension of the space cut out by the block dihedral symmetry. The latter is a straight forward argument from representation theory, while the former is a trek through the world of MZV relations and machinery, including: 
motivic cobracket calculations; multiple Euler sums (also called alternating MZVs) and the octagon-relation-induced dihedral symmetries thereof \cite{GlanoisThesis16,GlanoisBasis16}; 
multiple zeta star values and the stuffle-antipode \cite[Equation 2.4]{LiIdentities13},\cite[Proposition 1]{HoffmanQuasi20}; 
Zhao's generalised 2-1 theorem \cite{ZhaoIdentity16} (and the first author's block-decomposition description thereof \cite{CharltonCyclic20}); 
(alternating) multiple zeta-half values \cite{YamamotoInterp};
the explicit depth-parity theory for depth 3 alternating MZVs \cite{PanzerParity16,GoncharovGalois01}; 
and a vital (and serendipitously unearthed) generalised doubling relation \cite[Section 14.2.5]{ZhaoBook}. We divide these results between \autoref{sec:proofs} and an appendix, according to how central they are to the main results of the article.

We end the main body of the paper with a short corollary of \autoref{prop:zevbarevbar:eval} in relation to a variation of multiple zeta values, called multiple $t$ values \cite{HoffmanOdd19}
\[t(k_1,\ldots,k_d)\coloneqq\sum_{0<n_1<\cdots<n_d}\frac{1}{(2n_1-1)^{k_1}\cdots (2n_d-1)^{k_d}}\,.\]
In particular, we provide an evaluation of $t(2\ell,2k)$, when the arguments are even, via classical double zeta values, improving upon the results of \cite[Theorem 1]{MurakamiMtVs21} by giving an explicit formula for the Galois descent in this case.

\subsection*{Acknowledgements} SC was supported by Deutsche Forschungsgemeinschaft Eigene Stelle grant CH 2561/1-1, for Projektnummer 442093436, during the preparation of this work at Universit\"at Hamburg.  AK is grateful to the Max Planck Institute for Mathematics, Bonn, for the support and hospitality provided during his stay. The authors are very thankful to the referee for their insightful and considered comments.

\section{The motivic Lie algebra and block graded multiple zeta values}
An essential observation in the study of multiple zeta values is that they may be lifted to motivic periods -- algebraic objects satisfying only relations coming from geometry. Because of this, motivic multiple zeta values (mMZVs) are much simpler to study. They are known to be graded by weight and they come equipped with a coaction that encodes all motivic relations. We may consider their graded analogues with respect to a number of filtrations, or consider the associated Lie coalgebra of mMZVs modulo products, whose relations are encoded in a free Lie algebra. The theory of motivic periods is substantial \cite{BrownMotivicPeriodNotes}, so we give only an essential overview here, and refer the reader to \cite{BrownMotivicPeriodSummary} for more details.

\subsection{Motivic multiple zeta values}\label{sec:motmzv}
The formal definition of mMZVs relies on the Tannakian formalism for the category of mixed Tate motives over $\Spec\mathbb{Z}$, and is intimately related to the motivic fundamental group of $\proj$ \cite{BrownMTM12}. In brief, letting $\MT(\mathbb{Z})$ denote the category of mixed Tate motives, and denoting by 
\[\omega_B,\omega_{dR}:\MT(\mathbb{Z})\to\Vect_\Q\]
the Betti and de Rham realisation functors, the ring of motivic periods of $\MT(\mathbb{Z})$ is the ring of functions on the scheme of tensor isomorphisms
\[\mathcal{P}^{\mot}_{\MT(\mathbb{Z})}\coloneqq\mathcal{O}(\Isom_{\MT(\mathbb{Z})}^\otimes(\omega_{dR},\omega_B))\,.\]
The results of Brown \cite{BrownMTM12} tell us that there is isomorphic to $\mathcal{H}[\mathbb{L}^{-1}]$, where $\mathcal{H}$ will be the algebra of motivic multiple zeta values, and $\mathbb{L}$ is a motivic analogue of $2\pi \ii$.

However, for our purposes, the reader need only keep in mind the following properties.

\begin{properties}\label{def:Hn}
The algebra $\mathcal{H}$ of motivic multiple zeta values is the $\Q$-algebra spanned by symbols 
\[\imot(a_0;a_1,\ldots,a_n;a_{n+1})\text{ where }a_i\in\{0,1\}\,,\]
called motivic multiple zeta values or motivic iterated integrals, satisfying the following properties:
\begin{enumerate}[i)]
    \item (Equal boundaries) $\imot(a_0;a_1,\ldots,a_n;a_0)=\delta_{n,0}$,
    \item (Reversal of paths) $\imot(a_0;a_1,\ldots,a_n;a_{n+1})=(-1)^n\imot(a_{n+1};a_n,\ldots,a_1;a_0)$,
    \item (Functoriality) $\imot(a_0;a_1,\ldots,a_n;a_{n+1})=\imot(1-a_0;1-a_1,\ldots,1-a_n;1-a_{n+1})$
    \item (Shuffle product) For $1<r<n$, denote by $\Sh_{r,n-r}$ the set of permutations $\sigma$ on $n$ satisfying
    \[ \sigma(1)<\sigma(2)<\cdots<\sigma(r)\text{ and }\sigma(r+1)<\cdots<\sigma(n)\,.\]
    Then
    \[\imot(0;a_1,\ldots,a_r;1)\imot(0;a_{r+1},\ldots,a_n;1)=\sum_{\sigma\in\Sh_{n,r}}\imot(0;a_{\sigma^{-1}(1)},\ldots,a_{\sigma^{-1}(n)};1)\,.\]
    \item (Period map) There is a ring homomorphism $\per:(\mathcal{H},\cdot{})\to (\C,{}\cdot{})$, called the period map, sending a motivic iterated integral to the corresponding complex iterated integral when it converges.
\end{enumerate}

For a tuple of positive integers $(k_1,\ldots,k_d)$, and \( \ell \geq 0 \), we write \( \zm = \zm_0 \) and  \begin{equation}\label{eqn:zasint}
\zm_\ell(k_1,\ldots,k_d)\coloneqq(-1)^d\imot(0;\{0\}^\ell,1,\{0\}^{k_1-1},\ldots,1,\{0\}^{k_d-1};1)\,,
\end{equation}
where $\{0\}^n$ denotes $n$ repeated zeroes.
\end{properties}

\begin{remark}\label{rem:regularisation}
In the standard definition of motivic multiple zeta values, we have that
\[\imot(0;0;1)=\imot(0;1;1) = 0.\]
This, combined with compatibility with the shuffle product, uniquely determines the image of motivic MZVs which do not correspond to convergent iterated integrals. This convention produces what are called shuffle regularised MZVs. It is occasionally convenient to consider other regularisations, such as stuffle regularised MZVs or MZVs regularised so that $\zeta(1)$ is regularised to a non-zero constant \cite{IKZ06}.
\end{remark}

\begin{remark}
The reversal of paths property and the functoriality property give an important relation for motivic MZVs called the \emph{duality} relation:
\[\imot(0;a_1,\ldots,a_n;1)=(-1)^n\imot(0;1-a_n,\ldots,1-a_1;1)\, .\]
\end{remark}

Let $\mathcal{A}\coloneqq\mathcal{H}/(\zm(2))$ be the quotient by the ideal generated by $\zm(2)$, and denote by $\ia(a_0;a_1,\ldots,a_n;a_{n+1})$ the image of $\imot(a_0;a_1,\ldots,a_n;a_{n+1})$. The formula given below equips $\mathcal{H}$ with the structure of an $\mathcal{A}$-comodule
\[\Delta\colon\mathcal{H}\to\mathcal{A}\otimes\mathcal{H}\,.\]
Explicitly, $\Delta\imot(a_0;a_1,\ldots,a_n;a_{n+1})$ is equal to
\begin{equation}\label{eqn:coaction}
\sum_{\substack{i_0<i_1<\cdots<i_{k+1}\\i_0=0,\, i_{k+1}=n+1}}\left(\prod_{s=0}^k\ia(a_{i_s};a_{i_s+1},\ldots,a_{i_{s+1}-1};a_{i_{s+1}})\right)\otimes\imot(a_0;a_{i_1},\ldots,a_{i_k};a_{n+1})\,.\end{equation}
A linear combination $R$ of motivic multiple zeta values vanishes in $\mathcal{H}$ if and only if
\begin{enumerate}[i)]
    \item $\per(R)=0$, i.e. $R$ holds numerically
    \item $\per(R^\prime)=0$ for all transforms $R^\prime$ under the motivic coaction, i.e. the relation is motivic. A transform $R^\prime$ of $R$ under the motivic coaction is obtained by choosing a weight graded basis $\{a_i\}_{i\in I}$ of $\mathcal{A}$ and writing
    \[\Delta(R) = \sum_{i\in I} a_i\otimes R_i\,.\]
    Each $R_i$ is referred to as a transform of $R$.   
\end{enumerate}

As the coaction is quite combinatorially complicated, it is often convenient to instead consider the infinitesimal coactions $\D_r$. Define the Lie coalgebra of indecomposables 
\[\mathcal{L}\coloneqq\mathcal{A}_{>0}/\mathcal{A}_{>0}\mathcal{A}_{>0}\]
where $\mathcal{A}_{>0}$ denotes the positive weight part of $\mathcal{A}$. Denote by $I^\lmot(a_0;a_1,\ldots,a_n;a_{n+1})$ the image of $\ia(a_0;a_1,\ldots,a_n;a_{n+1})$ in $\mathcal{L}$, and similarly denote by $\zl(k_1,\ldots,k_r)$ the image of $\zm(k_1,\ldots,k_r)$ in $\mathcal{L}$. Let $\mathcal{L}_r$ be the weight $r$ component of $\mathcal{L}$. The infinitesimal coaction is then the composition
\[\mathcal{H}\to \mathcal{A}\otimes\mathcal{H}\to \mathcal{L}_r\otimes\mathcal{H}\]
of $\Delta-1\otimes \id$ with the projection to $\mathcal{L}_r$. Explicitly, $\D_r\imot(a_0;a_1,\ldots,a_n;a_{n+1})$ is given by
\begin{equation}\label{eqn:derivation}\sum_{k=0}^{n-r}I^\lmot(a_k;a_{k+1},\ldots,a_{k+r};a_{k+r+1})\otimes \imot(a_0;a_1,\ldots,a_k,a_{k+r+1},\ldots,a_n;a_{n+1})\,.\end{equation}

These infinitesimal coactions are significantly easier to compute, but still encode almost all essential information surrounding motivic MZVs.

\begin{thm}[{Brown \cite[Theorem 3.3]{BrownMTM12}}]\label{thm:coaction:identities}
Let $N>2$ and denote by $\D_{<N}=\bigoplus_{3\leq 2r+1<N}\D_{2r+1}$. Then in weight $N$, the kernel of $\D_{<N}$ is one dimensional:
\[\ker\D_{<N} \cap \mathcal{H}_N = \Q \zm(N)\,.\]
\end{thm}
Brown proves this result by considering a particular choice of isomorphism of coalgebras
\[(\mathcal{A},\Delta)\cong (\Q\langle f_3,f_5,f_7,\ldots\rangle,\Delta_{\decon})\,,\]
which he lifts to an isomorphism of comodules over these coalgebras. We may instead consider the corresponding vector spaces of indecomposables, equipped with with the structure of Lie coalgebras by defining the cobracket to be the natural cobracket coming from antisymmetrising the coproduct. We then have an isomorphism
\[(\mathcal{L},\partial)\cong(\Q\langle f_3,\ldots\rangle_{>0}/\Q\langle f_3,\ldots\rangle^{\sh 2}_{>0},\partial_{\decon})\,,\]
which we use to obtain the following standard proposition.
\begin{prop}\label{prop:coaction:modprod}
Denote by 
\[\partial_r:\mathcal{L}\to\mathcal{L}_r\otimes\mathcal{L}\] 
the $r$\textsuperscript{th} infinitesimal cobracket, given in weight $N$ by 
\[\partial_r \xi \coloneqq \pi_r\circ\partial=\D_{r}\xi-\tau\D_{N-r}\xi\]
where $\tau(a\otimes b)=b\otimes a$. Let $\partial_{<N}=\bigoplus_{3\leq 2r+1<N}\partial_{2r+1}$. Then, in weight $N$, the kernel of $\partial_{<N}$ is at most one dimensional:
\[\ker\partial_{<N}\cap \mathcal{L}_N = \Q\zl(N)\]
where we note that $\zl(2n)=0$.
\end{prop}
\begin{proof}
It is known $\mathcal{L}$ is isomorphic to $L$ the Lie coalgebra of indecomposables of $\Q\langle f_3,f_5,,\ldots\rangle$ with respect to the shuffle product, which is the cofree Lie coalgebra with cogenerators $f_3,f_5,\ldots$. Choosing an isomorphism $\phi$ such that $\phi(\zl(2n+1))=f_{2n+1}$, it suffices to show that
\[\ker\partial_{<N}\cap L_N = \Q f_N,\]
where we take $f_{2n} \coloneqq 0$.

As $L$ is graded by $f$-degree, the vanishing of $\partial_{<N}$ is equivalent to the vanishing of the full cobracket $\partial$ followed by projection onto odd weight in the first component. This composition is dual to the Lie bracket
\[L^\vee_{\mathrm{odd}}\otimes L^\vee \to L^\vee\]
where $L^\vee$ is the free Lie algebra on $f_3^\vee,f_5^\vee,\ldots$. As this map is surjective onto the $f$-degree at least 2 part of $L^\vee$, this implies that $\partial_{<N}$ is injective on the $f$-degree at least 2 part of $L$, and hence the kernel is spanned in weight $N$ by $f_N$.
\end{proof}

\begin{remark}
It is worth noting that this formalism for motivic multiple zeta values extends to more general motivic iterated integrals, in particular alternating motivic MZVs \cite{GlanoisBasis16}. We will need this extension for the results of \autoref{sec:mot2242}, and will introduce the necessary results and concepts as needed.
\end{remark}

\subsection{The motivic Lie algebra}
 Elements of $\mathcal{L}$ may be viewed as motivic multiple zeta values, modulo products. By considering the weight-graded dual, we obtain a Lie algebra $\gmot$, called the motivic Lie algebra. From the theory of mixed Tate motives and Tannakian categories, this Lie algebra is equal to a subspace of $\qpoly$, equipped with the Ihara bracket $\{\cdot,\cdot\}$ \cite{BrownMTM12,iharalin}. Via the pairing
 \[\langle I^\lmot(0;a_1,\ldots,a_m;1), e_{i_1}\ldots e_{i_n}\rangle = \delta_{m,n}\delta_{a_1,i_1}\ldots \delta_{a_m,i_m}\,,\]
 elements of $\gmot$ may be viewed as encoding relations among elements of $\mathcal{L}$. For example, in weight $5$, $\gmot$ is spanned by
 \[\sigma_5 = e_1e_0^4 + \frac{9}{2}e_1e_0^2e_1e_0+\cdots\]
 from which we can conclude that
 \[\zm(3,2) \equiv \frac{9}{2}\zm(5) \pmod{\text{products}} \,.\]

 As such, describing relations among motivic MZVs (up to products) is equivalent to describing the elements of $\gmot$. In particular, to describe all such relations, it would suffice to describe explicit generators for $\gmot$. It is known \cite{delmixedtate} that the motivic Lie algebra is isomorphic to a free Lie algebra
 \[\gmot\cong \Lie[\sigma_3,\sigma_5,\ldots]\,,\]
 with generators $\sigma_{2k+1}$ in every odd weight greater than 1. However, this isomorphism is non-canonical, and there does not exist an explicit canonical choice of representatives of $\sigma_{2k+1}$. However, we can somewhat remedy this by considering graded relations among motivic MZVs for a certain filtration. 
 
\subsection{The block filtration}\label{sec:blockdecomp}
 In \cite{BrownLetter16}, Brown proposed a new filtration on the space of convergent motivic MZVs, based on the work of the first author \cite{CharltonBlock21}. This was expanded to a filtration on all motivic MZVs by the second author in \cite{KeilthyBlock21, KeilthyThesis20}, in which the associated graded algebra - and relations therein - is investigated. In this section, we provide a brief summary of this filtration and relations in the associated graded algebra.
 
 Call a word in two letters $\{a,b\}$ alternating if it contains no subword of the form $aa$ or $bb$. As noted in \cite{CharltonBlock21}, every word in $\{a,b\}$ then has a unique factorisation into alternating words of maximal length. This allows us to uniquely determine a word $w$ by its first letter and the lengths of the alternating blocks in this factorisation $(x;\ell_1,\ldots,\ell_n)$, $x\in\{a,b\}$. We call this sequence the block decomposition of the word $w$.
 
 We define the block degree $\degb(w)$ of a word $w$ to be the number of instances of subwords of the form $aa$ or $bb$ in $w$. This allows us to define the block filtration on the vector space $\Q\langle a,b\rangle$ by defining the $n$\textsuperscript{th} part
 \[
    \mathcal{B}_n\Q\langle a,b\rangle \coloneqq\langle w\mid \degb(w)\leq n\rangle_\Q
 \]
 to be the vector subspace spanned by words of degree at most $n$.
 
 As motivic iterated integrals, and hence motivic MZVs, may be viewed as a quotient of $\qpoly$ via the map
 \[e_{i_0}e_{i_1}\ldots e_{i_{n+1}}\mapsto \imot(i_0;i_1,\ldots,i_n;i_{n+1})\]
 the space of motivic MZVs inherits the block filtration. We may also define the block degree of an iterated integral, by taking the block degree of the associated word. This turns out to be a very natural filtration to consider, satisfying a number of nice properties, the proofs of which we shall either sketch here, or may be found in \cite{KeilthyBlock21, KeilthyThesis20}.
 
 \begin{prop}
 The block filtration is equal to the coradical filtration induced by the motivic coaction. Furthermore, when restricted to the Hoffman basis, it is the level filtration of Brown \cite{BrownMTM12}. Hence, any MZV of block degree $N$ may be written as a linear combination of Hoffman MZVs with at most $N$ threes.
 \end{prop}

\begin{lem}[\cite{CharltonBlock21}]\label{lem:blockparity}
All MZVs of block degree $b$ and weight $N$, with $b$ and $N$ of opposite parity, vanish.
\end{lem}
\begin{proof}
If $\imot(i_0;i_1,\ldots,i_n;i_{N+1})$ has block degree $b$, then the final letter of $e_{i_0}\ldots e_{i_{N+1}}$ must be equal to the final letter of the alternating word of length $N+2-b$, beginning with $e_{i_0}$. In particular, we must have that $e_{i_{N+2-b}}=e_{i_0}$ if $N+2-b$ is odd, i.e. $N$ and $b$ are off opposite parity. Hence $\imot(i_0;i_1,\ldots,i_n;i_{n+1})=0$, as it has the same start and end points of the integral.
\end{proof}

Analogously to depth graded MZVs \cite{BrownDepthGraded21}, we may consider the associated graded algebra 
\[\gr^\mathcal{B}\mathcal{H}\coloneqq\bigoplus_{n\geq 0} \mathcal{B}_n\mathcal{H}/\mathcal{B}_{n-1}\mathcal{H}\]
and consider relations among block graded multiple zeta values. Much like relations among motivic MZVs, modulo products, in the motivic Lie algebra $\gmot$, relations among block graded MZVs, modulo products, are encoded in the block Lie algebra
\[\bg\coloneqq\bigoplus_{n\geq 0}\mathcal{B}^n\gmot/\mathcal{B}^{n+1}\gmot\]
where the filtration on $\gmot$ is induced by the filtration
\[\mathcal{B}^n\qpoly \coloneqq\langle w\mid \degb(w)\geq n\rangle_\Q\]
via the embedding $\gmot\hookrightarrow e_0\qpoly e_1$. We denote by $\bg_n$ the block degree $n$ part, via the embedding into $\mathcal{B}^n\qpoly/\mathcal{B}^{n+1}\qpoly$.

\begin{prop}
As Lie algebras $\gmot\cong\bg$. In particular, they are (non-canonically) isomorphic to $ \Lie[\sigma_3,\sigma_5,\ldots]$.
\end{prop}

\begin{prop}
Let $\{\sigma_{2k+1}\}_{k\geq 1}$ be a choice of generators for $\gmot$. Then the block degree $1$ piece of the image of $\{\sigma_{2k+1}\}_{k\geq 1}$ in $\qpoly$ is independent of the choice of generators. In particular, we can make a canonical identification between the image of $\bg$ in $\qpoly$ and the free Lie algebra $\Lie[\sigma_3,\sigma_5,\ldots]$.
\end{prop}

\begin{remark}\label{rem:depthgradedfail}
It is in these two results that we see a stark contrast to the case of depth graded multiple zeta values \cite{BrownDepthGraded21}. Analogously to the above, one can consider the associated graded Lie algebra for the depth filtration, induced by the $e_1$-degree of words in $\qpoly$. As for the block graded case, the image of the generators $\{\sigma_{2k+1}\}_{k\geq 1}$ in depth $1$ is canonical. However, unlike the block graded case, the depth graded Lie algebra $\mathfrak{dg}$ is not free, having quadratic relations and extra generators in depth 4. These quadratic relations are algebraically well understood \cite{BrownDepthGraded21,SchnepsPoisson06}, and give a somewhat mysterious connection to modular forms. Indeed, the quadratic relations are exactly encoded in the period polynomials of cusp forms. This is a relationship that we can discuss in a new light using the approaches of this paper.
\end{remark}

As the image of $\bg$ in $\qpoly$ lies in $e_0\qpoly e_1$, the block decomposition gives an injection of vector spaces
\[\bg \to \bigoplus_{n\geq 1}\Q[x_1,\ldots,x_n]\]
obtained by sending a word $w=e_0\ldots$ to $x_1^{\ell_1}\ldots x_n^{\ell_n}$, where $(e_0;\ell_1,\ldots,\ell_n)$ is the block decomposition of $w$. The image of $\bg_n$ under this map lies in $(x_1-x_{n+1})x_1\ldots x_{n+1}\Q[x_1,\ldots,x_{n+1}]$. We denote by $\rbg_n$ the image of $\bg_n$ divided by $(x_1-x_{n+1})x_1\ldots x_{n+1}$, and let $\mathfrak{rbg}\coloneqq\bigoplus_{n\geq 1}\rbg_n$. Thus, we have reduced the problem of describing relations among block graded MZVs modulo products to describing $\rbg$ as a subspace of $\bigoplus_{n\geq 2}\Q[x_1,\ldots,x_n]$.

In \cite{KeilthyBlock21}, a number of relations are found. In particular, elements of $\rbg$ satisfy a functional equation coming from shuffle regularisation; a differential equation; and have a dihedral symmetry.

\begin{prop}
If $f(x_1,\ldots,x_n)\in\rbg$, then
\[ f(x_1,\ldots,x_n)=f(x_n,\ldots,x_1)=(-1)^{n+1} f(x_2,\ldots,x_n,x_1)\,.\]
\end{prop}

It turns out these three relations, along with \autoref{lem:blockparity}, describe most relations among block graded MZVs \cite{KeilthyBlock21}. In block degree 1, the shuffle regularisation, the differential equation, and this dihedral symmetry, along with \autoref{lem:blockparity}, describe all relations among block graded motivic multiple zeta values.

\begin{prop}
The vector space $\rbg_1$ is the subspace of $\Q[x_1,x_2]$ given by polynomials $f(x_1,x_2)$ such that
\begin{alignat*}{3}
      f(0,x)   & {} = 2f(x,-x)\,,  \\
    f(-x_1,-x_2) & {}= f(x_1,x_2)\,,  && \quad\quad \text{
    \smash{{$\displaystyle\frac{\partial^2 f}{\partial x_1^2}=\frac{\partial^2 f}{\partial x_2^2} \,.$}}
    }\\
      f(x_1,x_2) & {} = f(x_2,x_1)\,, 
     \\[-1ex]
    \end{alignat*}
\end{prop}

Unfortunately, even in block degree 2, this is insufficient, leaving degrees of freedom linear in weight. While a remedy in block degree 2 was given in Proposition 2.8.7 of \cite{KeilthyThesis20}, it turns out that the failure of space cut out by the dihedral symmetry, differential equation, and shuffle regularisation to encode all relations in block degree 2 has an interesting connection to double zeta values, and gives an alternative source of the relations between double zeta values coming from period polynomials. This connection is explored in the next section.

\section{Block graded relations among double zeta values}\label{sec:blockdmzv}
As noted above, relations among block graded motivic multiple zeta values, modulo products, are determined by the coefficients of elements of $\bg$. However, these relations are also genuine relations among motivic multiple zeta values mod products for motivic MZVs of block degree at most 2. Observe that
\[\mathcal{B}_0\mathcal{L}=\langle \z(2n)\mid n\geq 1\rangle_\Q =\{0\}\pmod{\text{products}}\]
and so, modulo products, $\mathcal{B}_1\mathcal{L}/\mathcal{B}_0\mathcal{L}=\mathcal{B}_1\mathcal{L}$. Similarly, as MZVs of block degree 1 are necessarily of odd weight, and MZVs of block degree 2 are necessarily of even weight, block graded relations among motivic MZVs of block degree 2 are genuine relations, modulo products. As \autoref{eval:doublez2242} defines an explicit representation of double zeta values in terms of MZVs of block degree 2, relations among double zeta values are determined, modulo products, by the coefficients of elements of $\bg$. In this section, we will make this precise, and show that all relations among double zeta values are determined by the below relations.

Explicitly, following Remark 9.3 of \cite{KeilthyBlock21}, the weight $2n+2$, block degree 2 piece of $\bg$ can be identified with a subspace of $V_n\subset \Q[x_1,x_2,x_3]$, where $V_n$ is spanned by polynomials satisfying the following relations.
\begin{thm}[\cite{KeilthyBlock21}]\label{thm:block-relations}
  Define $V_n\subset\Q[x_1,x_2,x_3]$ to be the space spanned by polynomials satisfying the block relations: 
\begin{align}
\tag{Relation 0}\label{rel0} & f(\lambda x_1,\lambda x_2,\lambda x_3) = \lambda^{2n}f(x_1,x_2,x_3)\text{ for all }\lambda\in \Q \,,\\[1ex]
\tag{Relation 1}\label{rel1} & f(x_1,x_2,x_3)=f(x_2,x_3,x_1)=-f(x_3,x_2,x_1)\,,\\
\tag{Relation 2}\label{rel2} &\frac{1}{2}\big(f(0,y,z)-f(0,y,-z)\big)=f(-y,y,z)-f(y,-z,z)\,,\\
\tag{Relation 3}\label{rel3} & \diffrbg{f}=0 \,.
\end{align}
    Then $\rbg_{2,2n+2}$, weight $2n+2$ component of $\rbg_2$, is a subspace of $V_n$.
\end{thm}
As mentioned previously, this inclusion is strict. Additional relations are necessary in order to completely describe $\rbg_{2,2n+2}$ as a subspace of $\Q[x_1,x_2,x_3]$. A choice of such relations is given in Proposition 2.8.7 of \cite{KeilthyThesis20}.
For any $f(x_1,x_2,x_3)=\sum_{i+j+k=2n}\alpha_{i,j,k}x_1^ix_2^jx_3^k$, define \[f_e(x_1,x_2,x_3)\coloneqq \sum_{i+j+k=n}\alpha_{2i,2j,2k}x_1^{2i}x_2^{2j}x_3^{2k}\,.\] One may easily check that if $f(x_1,x_2,x_3)\in V_n$, then both $f_e(x_1,x_2,x_3)$ and $f(x_1,x_2,x_3)-f_e(x_1,x_2,x_3)$ are elements of $V_n$. As such,
 the linear map $f(x_1,x_2,x_3)\mapsto f_e(x_1,x_2,x_3)$ defines a projection $V_n\to V_n$.

\begin{restatable}[\cite{KeilthyThesis20}]{prop}{blockdimcount}\label{prop:blockdimensioncount}
Let $V_n$ be as above, and let $P_e:V_n\to V_n$ denote the projection $f(x_1,x_2,x_3)\mapsto f_e(x_1,x_2,x_3)$. Then
\begin{align*}
\dim \im P_e &\leq\Big\lfloor\frac{n}{3}\Big\rfloor\, ,\\
\dim \ker P_e &=\Big\lfloor\frac{n-1}{2}\Big\rfloor = \dim \gmot_{2,2n+2}\, ,
\end{align*}
where $\gmot_{2,2n+2}$ denotes the vector space spanned by $\{\sigma_k,\sigma_\ell\}$ with $k+\ell=2n+2$.
\end{restatable}
We delay the proof until the following section. Denote by 
 \[\Phi_n(x_1,\ldots,x_{n+1}):=\sum_{\ell_1,\ell_2,\ldots,\ell_{n+1}} \ibl(\ell_1,\ldots,\ell_{n+1}) \otimes x_1^{\ell_1}\ldots x_{n+1}^{\ell_{n+1}} \in \gr_n^\mathcal{B}\mathcal{L}\otimes\Q[x_1,\ldots,x_{n+1}]\]
 the generating series of block degree $n$ motivic MZVs modulo products. As a consequence of Lemma 7.4 \cite{KeilthyBlock21}, 
 \[\Phi_n(x_1,\ldots,x_{n+1}) = x_1\ldots x_{n+1}(x_1-x_{n+1})\phi_n(x_1,\ldots,x_{n+1})\]
 for some $\phi_n\in \gr_n^\mathcal{B}\mathcal{L}\otimes\Q[x_1,\ldots,x_{n+1}]$. This series is then given by
 \[\phi_n(x_1,\ldots,x_{n+1}) = \sum_{k_1,\ldots,k_{n+1}\geq 0} F(k_1,\ldots,k_{n+1})\otimes x_1^{k_1}\ldots x_{n+1}^{k_{n+1}}\]
 where
 \[F(k_1,\ldots,k_{n+1}) = -\sum_{i+j=k_1}\ibl(i+1,k_2+1,\ldots,k_n+1,k_{n+1}+j+2).\]
 Conversely, we have that
 \[\ibl(\ell_1,\ldots,\ell_{n+1}) = F(\ell_1-2,\ell_2-1,\ldots,\ell_n-1,\ell_{n+1}-1) - F(\ell_1-1,\ldots,\ell_n-1,\ell_{n+1}-2).\]

Via the pairing
 \[\gr^{\mathcal{B}}_n\mathcal{L}\times \bg_n \to \Q\]
 we can view $\Phi_n$ as a linear map 
 \[\bg_n\to\Q[x_1,\ldots,x_{n+1}]\]
 the image of which is precisely the embedding described previously. We may similarly view $\phi_n$ as a linear isomorphism $\bg_n\to\rbg_n$.

 \begin{lem}\label{lem:F_to_double}
In block degree 2, the coefficients of $\phi_{3,e}(x_1,0,x_3)$ are equal to motivic double zeta values:
\[F(a,0,b) = 4\zl(a+1,b+1)\]
for $a,b$ even non-negative integers.
 \end{lem}
 \begin{proof}
  We assume, without loss of generality, that $a\leq b$. From \autoref{lem:eval:z2242prodmot} and \autoref{eval:doublez2242}, we have that
\[\zl(2a+1,2b+3) = \frac{(-1)^{a+b}}{4}\sum_{s=a}^{b}\zl(\{2\}^s,4,\{2\}^{a+b-s}).\]
The block decomposition of 
\[\zl(\{2\}^s,4,\{2\}^{a+b-s})=(-1)^{a+b+1}\il(0;\{1,0\}^s,1,0,0,0,\{1,0\}^{a+b-s};1)\] is $(\ell_1,\ell_2,\ell_3)=(2s+3,1,2a+2b-2s+2)$. Hence
\begin{align*}
\zl(2a+1,2b+3) &= -\frac{1}{4}\sum_{s=a}^b \ibl(2s+3,1,2a+2b-2s+2)\\
    &= -\frac{1}{4}\sum_{s=a}^b F(2s+1,0,2a+2b -2s +1) - F(2s+2,0,2a+2b-2s) \,.
\end{align*}
The dihedral symmetry of $\rbg_2$ implies that
\[F(p,q,r) = F(q,r,p) = -F(r,q,p)\]
and so this sum reduces to
\[\zl(2a+1,2b+3) = \frac{1}{4}F(2b+2,0,2a) \,,\]
from which the claim follows.
 \end{proof}

 \begin{lem}\label{lem:all_F_are_double}
For every tuple $(a,b,c)$ of even, non-negative integers,
\[F(a,b,c)=-4\zl_a(b+1,c+1)\]
where we denote by
\[\zl_a(b+1,c+1) \coloneqq \il(0;\{0\}^a,1,\{0\}^b,1,\{0\}^c;1)\]
the regularized iterated integral modulo products.
 \end{lem}
\begin{proof}
    By viewing $\phi_2$ as a linear isomorphism $\bg_2\to\rbg_2$, we see that we must have that the vector space
    \[\langle F(a,b,c)\mid a,b,c\geq 0,\, a+b+c=2n\rangle_\Q\]
    is dual to $\rbg_{2,2n+2}$, and furthermore that
    \[\langle \langle F(a,b,c)\mid a,b,c\geq 0\text{ and even},\, a+b+c=2n\rangle_\Q\]
    is dual to $P_e(\rbg_{2,2n+2})$.

    Following Brown's conventions \cite{BrownDepthGraded21}, $\gr_\mathcal{D}^1\gmot$ may be identified with the space of translation invariant polynomials $s_{2n+1}(x_0,x_1):=(x_0-x_1)^{2n}$. By the work of the second author \cite{KeilthyBlock21}, $\rbg_1$ is spanned by polynomials
    \[p_{2n+1}(x_0,x_1) := \frac{(2^{2n+1}-1)(x_0+x_1)^{2n} - (x_0-x_1)^{2n}}{2^{2n}}.\]
    Denoting by $f_e(x_0,x_1)$ the projection of a polynomial $f(x_0,x_1)$ onto $\Q[x_0^2,x_1^2]$, it is easy to see that
    \[s_{2n+1,e}(x_0,x_1) = 2p_{2n+1,e}(x_0,x_1)\, .\]
    The depth graded Ihara bracket
    \[\{\cdot,\cdot\}:\gr_\mathcal{D}^1\gmot\wedge \gr_\mathcal{D}^1\gmot\to \gr_\mathcal{D}^2\gmot\]
    is given by
    \begin{align*}\{f,g\}(x_0,x_1,x_2) =&{} f(x_0,x_1)\left(g(x_0,x_2)-g(x_1,x_2)\right)\\
        &+  f(x_1,x_2)\left(g(x_0,x_1)-g(x_0,x_2)\right)\\
        &+  f(x_2,x_0)\left(g(x_1,x_2)-g(x_0,x_1)\right)
    \end{align*}
    which is identical to the block graded Ihara bracket
    \[\{\cdot,\cdot\}:\rbg_1\wedge\rbg_1\to\rbg_2\,.\]
    Furthermore, as all the polynomials involved are of even total degree, this commutes with projection onto polynomials even in each variable (where we have formally extended the Ihara bracket to all polynomials of even total degree). Thus, we obtain a surjective map
    \begin{gather*}
        \gr_\mathcal{D}^2\gmot \to P_e(\rbg_2)\,,\\
        \{s_{2k+1},s_{2\ell+1}\} \mapsto 4\{p_{2k+1,e},p_{2\ell+1,e}\}\,.
    \end{gather*}
    Dualising this we obtain an injective map
    \begin{gather*}
         \langle F(a,b,c)\mid a,b,c\geq 0\text{ and even},\rangle_\Q\to \gr_2^\mathcal{D}\mathcal{L}\,,\\
        F(a,b,c)\mapsto 4\zl_a(b+1,c+1)\,.
    \end{gather*}
    Since this map is injective, and $\{\zl(b+1,c+1)\}_{b+c=2n}$ span $\gr_2^\mathcal{D}\mathcal{L}$ in weight $2n+2$, we must have that if $(a,b,c)$ are even non-negative integers such that
    \[\zl_a(b+1,c+1) = \sum_{2k+2\ell=a+b+c} \eta_{k,\ell}\zl(2k+1,2\ell+1)\]
    then, by \autoref{lem:F_to_double},
    \[F(a,b,c) = \sum_{2k+2\ell=a+b+c} \eta_{k,\ell}F(0,2k,2\ell)\]
    and hence
    \[F(a,b,c) = \sum_{2k+2\ell=a+b+c} -4\eta_{k,\ell}\zl(2k+1,2\ell+1) = - 4\zl_a(b+1,c+1)\,. \qedhere \]
\end{proof}

\propalldbl
\begin{proof}
 By corollary 4.2 of \cite{SchnepsPoisson06}, we have that the dimension of the space of motivic multiple zeta values of weight $2n+2$ modulo products is equal to the dimension of $\gr_2^\mathcal{D}\gmot_{2,2n+2}$, which is given by
\[\Big\lfloor\frac{n-1}{2}\Big\rfloor - \dim S_{2n+2}\,,\]
where $S_{2n+2}$ is the space of cusp forms of weight $2n+2$. It is known that
\[ \dim S_{2n+2} =\begin{cases} 
      \frac{n}{6}-1 & \text{if $n\equiv 0\pmod{6}$} \\
\frac{n}{6} & \text{if $n\equiv 1,2,3,4\pmod{6}$} \\
\frac{n}{6}+1 & \text{if $n\equiv 5\pmod{6}$\,.} \\
   \end{cases}
\]
Checking each case, we see 
\[\Big\lfloor\frac{n-1}{2}\Big\rfloor - \dim S_{2n+2}=\Big\lfloor\frac{n}{3}\Big\rfloor\,.\]

By \autoref{prop:blockdimensioncount},
\[\dim P_eV_{n}\leq\Big\lfloor\frac{n}{3}\Big\rfloor\,.\]
As $\gr^\mathcal{D}_2\mathcal{L}$ is spanned by motivic double zetas, \autoref{lem:all_F_are_double} implies that 
\[\langle F(a,b,c)\mid a,b,c\in 2\mathbb{Z}\rangle_\Q \cong \gr^\mathcal{D}_2\mathcal{L}\,.\]
Hence, the surjection of \autoref{lem:all_F_are_double} is an isomorphism
\[P_e\rbg_{2,2n+2}\cong \gr_2^\mathcal{D}\gmot_{2,2n+2}\]
for every $n>0$.
Following \autoref{thm:block-relations} we have that $P_e\rbg_{2,2n+2}\subset P_eV_n$. But
\[\Big\lfloor\frac{n}{3}\Big\rfloor\ = \dim \gr_2^\mathcal{D}\gmot_{2,2n+2}=\dim P_e\rbg_{2,2n+2}\leq \dim P_eV_n \leq \Big\lfloor\frac{n}{3}\Big\rfloor \]
and hence
\[P_e\rbg_{2,2n+2}= P_eV_n\,.\]
As such, all relations in
\[\langle F(a,b,c)\mid a,b,c\in 2\mathbb{Z}\rangle_\Q \cong \gr^\mathcal{D}_2\mathcal{L}\]
are determined by the relations defining $V_n$. In fact, as \eqref{rel2} has no even part, all weight graded relations among $F(a,b,c)$ with $a,b,c$ even are determined by \eqref{rel1} and \eqref{rel3}.
\end{proof}

 Notably, \autoref{prop:alldouble} tells us that the period polynomial relations among double zetas are a consequence of the dihedral symmetry of the block graded motivic Lie algebra. Using the surjection of \autoref{lem:all_F_are_double} we can, in fact, make this quite explicit.

 \begin{remark}
     Note that this also shows that we can upgrade the results of \autoref{prop:blockdimensioncount} to
     \begin{align*}
\dim \im P_e &=\Big\lfloor\frac{n}{3}\Big\rfloor\,,\\
\dim \ker P_e &=\Big\lfloor\frac{n-1}{2}\Big\rfloor = \dim \gmot_{2,2n+2} \,,
\end{align*}
with equalities in both cases.
 \end{remark}
\subsection{An explicit connection to period polynomials} 
 Recall that the space of even period polynomials $W_{2n}^+$ of degree $2n$ is defines as the subspace of $\Q[x_1,x_2]$ consisting of polynomials that are homogeneous of degree $2n$, even in each variable, and satisfy
 \begin{align*}
     P(x_1,0)&=P(0,x_2)=0\,,\\
     P(x_1,x_2)&+P(x_2,x_1)=0\,,\\
     P(x_1,x_2)+P(x_1,x_1+x_2)&+P(x_1+x_2,x_2)=0\,.
 \end{align*}
 See, for example, Section 8 of \cite{BrownZeta317} for more detail.

 \begin{prop}\label{prop:kernel-period-isomorphism}
 Denote by $e_{2k+1}$ the projection of the image of the Lie algebra generator $\sigma_{2k+1}$ in $\rbg_1$ to $\Q[x_1^2,x_2^2]$. The kernel
 \[K\coloneqq \ker\bigg(\{\cdot,\cdot\}:\bigoplus_{k\geq 1}\Q e_{2k+1}\wedge\bigoplus_{\ell\geq 1}\Q e_{2\ell+1}\to \bigoplus_{n\geq 1}P_eV_n\bigg)\]
 is isomorphic to the space of even period polynomials
 \[K\cong \bigoplus_{n\geq 1}W_{2n}^+\,.\]
 \end{prop}
 \begin{proof}
Define a pair of linear maps
\begin{gather*}
    \pi_1:X^2\Q[X^2] \to \Q[x_1,x_2]\,,\\
    p(X)\mapsto P_e(p(x_1-x_2))
\end{gather*}
and
\begin{gather*}
    \pi_2:\Q[X^2,Y^2]_{>0}\to\Q[x_1,x_2,x_3]\,,\\
    p(X,Y)\mapsto P_e(p(x_1-x_2,x_2-x_3)),
\end{gather*}
where we write $\Q[X^2,Y^2]_{>0}$ for the subspace of polynomials of positive degree. We define a basis for the space of antisymmetric polynomials in $\Q[X^2,Y^2]_{>0}$ given by
\[\left\{ Q_{2k,2\ell}(X,Y) \coloneqq X^{2k}Y^{2\ell} - X^{2\ell}Y^{2k}\right\}\,.\]
The first map $\pi_1$ defines an isomorphism
\[X^2\Q[X^2]\to \bigoplus_{k\geq 1}\Q e_{2k+1}\]
and it is not difficult to show that we have a commutative diagram
\[
\begin{tikzcd}
    X^2\Q[X^2]\wedge Y^2\Q[Y^2]\arrow[d]\arrow[r] & \bigoplus_{k\geq 1}\Q e_{2k+1}\wedge\bigoplus_{\ell\geq 1}\Q e_{2\ell+1} \arrow[d] \\
  \Q[X^2,Y^2]_{>0}\arrow[r] & \Q[x_1,x_2,x_3] \,,
\end{tikzcd}\]
where the left vertical arrow is the map
\[X^{2k}\wedge Y^{2\ell} = Q_{2k,2\ell}(X,Y)+Q_{2k,2\ell}(X,X+Y)+Q_{2k,2\ell}(X+Y,Y)\]
and the right vertical arrow is the map induced by the Ihara bracket. By construction, the image of the right vertical arrow is contained in $\bigoplus_{n\geq 1}P_eV_n$, and has kernel $K$. Note also that
\[\bigoplus_{n\geq 1}W_{2n}^+ \cong  \ker\left(X^2\Q[X^2]\wedge Y^2\Q[Y^2]\to \Q[X^2,Y^2]_{>0}\right)\]
via the identification
\[X^{2k}\wedge Y^{2\ell} \mapsto Q_{2k,2\ell}(X,Y)\,,\]
again by construction. Thus we have a commutative diagram of short exact sequences
\[\begin{tikzcd}
    0\arrow{r} & W_{2n}^+ \arrow{r}\arrow{d}{F} &\bigoplus_{k+l=n} \Q X^{2k}\wedge Y^{2\ell}\arrow{d}{\cong} \arrow{r} & Q_{2n} \arrow{r}\arrow{d}{G}& 0\\
    0\arrow{r}& K_n \arrow{r} &\bigoplus_{k+l=n}\Q e_{2k+1}\wedge \Q e_{2\ell+1} \arrow{r} &\rbg_{2,2n+2} \arrow{r} & 0 \,,
\end{tikzcd}\]
where we denote by 
\[K_n:=\ker\left(\bigoplus_{k+l=n}\Q e_{2k+1}\wedge \Q e_{2\ell+1}\to P_eV_n\right).\]
A short diagram chase shows that $F$ is an injection and is $G$ a surjection. Thus, we must have
\[\dim K_n \geq \dim W_{2n}^+ =\dim S_{2n+2}\,.\]
If $\dim K_n> \dim W_{2n}^+$, then we must have
\[\langle \{e_{2k+1},e_{2\ell+1}\}\mid k+\ell=n\rangle_\Q < \Big\lfloor\frac{n-1}{2}\Big\rfloor - \dim S_{2n+2} = \dim P_e V_n.\]
But 
\[\langle \{e_{2k+1},e_{2\ell+1}\}\mid k+\ell=n\rangle_\Q = P_e\rbg_{2,2n+2}\]
by definition, and by the previous theorem $P_e\rbg_{2,2n+2}=P_eV_n$. Hence $\dim K_n = \dim W_{2n}^+$ for every $n>0$.
 \end{proof}
 
 This suggests that a possible approach to study depth graded motivic multiple zeta values and exploring \autoref{conj:depthzetadim} is to consider the Lie algebra generated by the $\{e_{2k+1}\}_{k\geq 1}$, or equivalently, the projection of $\rbg$ onto $\bigoplus_{n\geq 2}\Q[x_1^2,x_2^2,\ldots,x_n^2]$. Indeed, the results of this section show that this is isomorphic to the depth graded motivic Lie algebra in depths 1 and 2, though this isomorphism cannot extend to depth 4 as the projection of $\rbg_4$ onto $\Q[x_1^2,x_2^2,x_3^2,x_4^2,x_5^2]$ is generated by the $\{e_{2k+1}\}_{k\geq 1}$ and hence we cannot find the ``exceptional'' generators in depth 4 referred to in \autoref{rem:depthgradedfail} required to generate the full depth graded Lie algebra \cite[Section 1.4]{BrownDepthGraded21}

\section{Proofs of the more technical results}\label{sec:proofs}
We now explain the proofs of some of the more technical results used in the previous sections. It is worth noting that determining the statement of \autoref{lem:eval:z2242prodmot} required computing the full evaluation of $\zm(\{2\}^a,4,\{2\}^b)$ described in \autoref{sec:num2242} and \autoref{sec:mot2242}. However, as we only use the evaluation of $\zeta^\lmot(\{2\}^a,4,\{2\}^b)$ in terms of double zetas modulo products, we have elected to give here a simpler direct proof using the motivic formalism.
 \subsection{\texorpdfstring{Evaluation of $\zl(\{2\}^a,4,\{2\}^b)$}{Evaluation of zeta\textasciicircum{}l(\{2\}\textasciicircum{}a, 4, \{2\}\textasciicircum{}b)}}
\begin{lem}\label{lem:eval:z2242prodmot}
	The following evaluation holds in the motivic coalgebra
	\begin{align*}
	  \zeta^\lmot(\{2\}^a, 4, \{2\}^b) {} = {}
	& (-1)^{a+b} \begin{aligned}[t] \bigg\{  & {-} 4 \zeta^\lmot(2a+2, 2b+2) +  4 \zeta^\lmot(2b+1, 2a+3)
	\\ & 
	+ \sum_{\crampedsubstack{i + j = 2a + 2b \\ i, j \geq 0 }} \bigg(  \frac{1}{2^i} \binom{i+1}{2a+1} + \frac{1}{2^j} \binom{j+1}{2b+1} \! \bigg) \zeta^\lmot(i+2, j+2) \bigg\} 
	 \,. \end{aligned}
	\end{align*}
\end{lem}

\begin{proof}
    For \(Z \) a weight \( w \) combination of motivic MZVs, it is sufficient to check, by \autoref{prop:coaction:modprod}, that 
    \[
        \partial_{<w}Z=0
    \]
    for all \(1<2r+1<w \), as this would show that \(Z=\alpha \zm(w) +\text{products} \). If the weight of \( Z \) is even, we have that \(\zeta^\lmot(w)=0 \); this means checking that \( \partial_{<w}Z=0 \) allows us to confirm that \(\pi(Z)=0 \) on the nose, where \(\pi\colon\mathcal{A}\to\mathcal{L}\) is the natural projection.
    
    Explicitly, it amounts to checking for \( 1<2r+1<w \) that
    \[
		(\id \otimes \pi)(\D_{r} Z) - \tau (\id \otimes \pi)(\D_{w-{r}} Z) = 0 \,,
	\]
	where \( \tau(a\otimes b)= b\otimes a\). For the case under consideration we need to check for \( 3 \leq 2r+1 \leq 2a + 2b + 1 \)
	\[
		(\id \otimes \pi)(\D_{2r+1} Z) - \tau (\id \otimes \pi) (\D_{2a + 2b + 3 - 2r} Z) \overset{?}{=} 0 \,.
	\]
	
    It is a straightforward calculation, as explained in \autoref{sec:zs2242dp3:mot}, to show the following.
	\[
		(\id \otimes \pi) \D_{2r+1} \zm(\{2\}^a, 4, \{2\}^b) = \begin{aligned}[t]
			-\delta_{r \leq a} & \zeta_1^\lmot(\{2\}^r) \otimes \zl(\{2\}^{a-r}, 3, \{2\}^b)  \\
			& + \delta_{r \leq b} \zeta_1^\lmot(\{2\}^r) \otimes \zl(\{2\}^a, 3, \{2\}^{b-r}) \,.
		\end{aligned}
	\]
	Recalling the motivic evaluations of \( \zm_1(\{2\}^r) \) and \( \zm(\{2\}^a, 3, \{2\}^b) \) from \cite{BrownMTM12}, we have that
	\begin{align*}
	    \zl_1(\{2\}^r) & {} = 2(-1)^r \zl(2r+1)\,, \\
		\zl(\{2\}^a, 3, \{2\}^b) & {}= \begin{aligned}[t] 2 (-1)^{a + b + 1} \bigg( &  \! \binom{2a  +2b + 2}{2a + 2}   \\ 
		{} & \quad -  (1 - 2^{-2a - 2b - 2}) \binom{2a + 2b + 2}{2b+1} \! \bigg) \zl(2a + 2b + 3) \,. \end{aligned}
	\end{align*}
	Therefore
	\begin{align*}
			& (\id \otimes \pi) \D_{2r+1} \zm(\{2\}^a, 4, \{2\}^b) = \\
			& \begin{aligned}[t]
			 \Bigg\{  4\delta_{r \leq a} (-1)^{a+b} \bigg( \! \binom{2a + 2b + 2 - 2r}{2a - 2r + 2} - {} & (1 - 2^{2r-2a-2b-2}) \binom{2a + 2b + 2 - 2r}{2b + 1} \! \bigg)  \\
			 - 4\delta_{r \leq b} (-1)^{a+b} \bigg( \! \binom{2a + 2b + 2 - 2r}{2a + 2} & {} - (1 - 2^{2r-2a-2b-2}) \binom{2a + 2b + 2 - 2r}{2b -2r + 1} \! \bigg) \Bigg\} \\
			 & \hspace{2em} {} \cdot \zl(2r+1) \otimes \zl(2a + 2b + 3 - 2r) \\
			\end{aligned}
	\end{align*}
	Likewise from \autoref{sec:mot_doubling}
	\[
		(\id\otimes\pi) \D_{2r+1} \zm(p,q) = \begin{aligned}[t] \bigg( \delta_{2r+1 = p} + (-1)^p \binom{2r}{p-1} - (-1)^q \binom{2r}{q-1} \! \bigg) & \\ {} \cdot \zl(2r+1) \otimes \zl(p+q - 2r - 1) & \,. \end{aligned}
	\]
	So for the purpose of checking
	\[
		(\id \otimes \pi)(\D_{2r+1} Z) - \tau (\id \otimes \pi) (\D_{2a + 2b + 3 - 2r} Z) = 0
	\]
	where \( Z \) is the purported evaluation of \( \zl(\{2\}^a, 4, \{2\}^b) \) via double zeta values \( \zl(n_1, n_2) \), we can project \( \zl(2r+1) \otimes \zl(2a + 2b + 3 - 2r) \mapsto 1 \), and reduce to an identity amongst binomial coefficients.
	
	After some straightforward simplification of the deltas and binomial coefficients in the expression for
	\[
		\big( (\id \otimes \pi) \D_{2r+1} - \tau (\id \otimes \pi) \D_{2a + 2b + 3 - 2r} \big) ( \text{LHS \autoref{lem:eval:z2242prodmot}} - \text{RHS \autoref{lem:eval:z2242prodmot}}) \,,
	\]
	for the range \( 3 \leq 2r+1 \leq 2a + 2b + 1 \), and using that \( i, j \) have the same parity in the sum, we find the claimed identity to be equivalent to the following purported identity
	\begin{align*}
		& \sum_{\substack{i + j = 2a + 2b \\ i,j \geq 0}} \begin{aligned}[t] & (-1)^i \bigg( 2^{-i} \binom{i+1}{2a+1} + 2^{-j} \binom{j+1}{2b+1} \! \bigg) \\
		& {} \cdot \bigg( \! \binom{2r}{i+1} - \binom{2r}{j+1} - \binom{2a + 2b + 2 - 2r}{i+1} + \binom{2a + 2b + 2 - 2r}{j+1} \! \bigg) \end{aligned} \\
		& \overset{?}{=} \begin{aligned}[t] & 2^{-(2a + 2b + 1 - 2r)} \binom{2a + 2b + 2 - 2r}{2b + 1} - 2^{-(2a + 2b + 1 - 2r)} \binom{2a + 2b + 2 - 2r}{2a + 1} \\
		& - 2^{-(2r-1)} \binom{2r}{2b + 1} + 2^{-(2r - 1)} \binom{2r}{2a + 1} \,.	
		\end{aligned}	
	\end{align*}
	This is seen to be an exact identity using the results from  \autoref{lem:binomial} below.
\end{proof}

\begin{lem}\label{lem:binomial}
	For \( 3 \leq 2r + 1 \leq 2k+2\ell-3 \), the following identities hold
	\begin{align*}
	\tag{i}
		& \sum_{i=0}^{2k-2} (-2)^{-i} \binom{i+2\ell-1}{2\ell-1} \binom{2r}{i + 2\ell-1} = 2^{-(2r + 1 - 2\ell)} \binom{2r}{2\ell-1} \\
	\tag{ii}
		& \sum_{i=0}^{2\ell-2} (-2)^{-i-2k} \binom{i+2k-1}{2k-1} \binom{2r}{2\ell-i-1} = 
		\sum_{i=0}^{2k-2} (-2)^{-i-2\ell}  \binom{i+2\ell-1}{2\ell-1} \binom{2r}{2k-i-1}  \,,
	\end{align*}
	i.e. the left hand side is symmetric in \( k \leftrightarrow \ell \).
	
	\begin{proof}
		Given the restriction \( 2r + 1 \leq 2k + 2\ell - 3 \), the sum in i) can be truncated to \( i = 2r + 1 - 2\ell \).  It is then reduced to the binomial theorem after simplifying the summand. \medskip
		
		For ii), we show that the left hand side is symmetric in \( k \leftrightarrow \ell \), to obtain the equality.  (We remark here that the symmetry is not obvious, as even the number of non-zero terms differs between the two sides.) To show the symmetry, note that 
		\begin{align*}
			& \sum_{i=0}^{2\ell-2} (-2)^{-i-2k} \binom{i+2k-1}{2k-1} \binom{2r}{2\ell-i-1} {} = {}  \\ 
			& \bigg. -(-2)^{-i-2k} \binom{i+2k-1}{2k-1} \underbrace{\binom{2r}{2\ell-i-1}}_{=1} \bigg\rvert_{i = 2\ell-1} + \sum_{i=0}^{\infty} (-2)^{-i-2k} \binom{i+2k-1}{2k-1} \binom{2r}{2\ell-i-1} \,,
		\end{align*}
		since the second binomial vanishes for \( i > 2\ell-1\).  The first term is equal to the coefficient of \( x^{2\ell-1} \) in
		\[
			-(2 + x)^{-2k}
		\]
		Likewise, the second term is the coefficient of \( x^{2\ell-1} \) in 
		\[
			(1 + x)^{2r} (2+x)^{-2k} \,.
		\]
		Therefore the left hand side of ii) is given by
		\[
			[x^{2\ell-1}] \bigg( \frac{(1+x)^{2r} - 1}{(2+x)^{2k}} \bigg) = [x^{-1}] \bigg( \frac{(1+x)^{2r} - 1}{(2+x)^{2k} x^{2\ell}} \bigg) \,.
		\]
		This is not obviously symmetric in \( k \leftrightarrow \ell \); it is however equal to the following contour integral around 0 (along a sufficiently small circle \( C_{\varepsilon}(0) \)) by the residue theorem
		\[
			= \frac{1}{2 \pi \ii} \oint_{C_{\varepsilon}(0)} \frac{(1+z)^{2r} - 1}{(2+z)^{2k} z^{2\ell}} \mathrm{d}z \,.
		\]
		The only poles of the integrand \[
			f(z) = \frac{(1+z)^{2r} - 1}{(2+z)^{2k} z^{2\ell}}
		\] are at \( z = -2 \), and at \( z = 0 \).  Since \( 2r + 1 \leq 2k + 2\ell - 3 \), we see that \[
			-\frac{1}{z^2}f\Big(\frac{1}{z}\Big) = \frac{z^{2r} - (1 + z)^{2r}}{(1 + 2z)^{2k}} \cdot z^{2k + 2l - 2 - 2r}
		\]
		has no pole at \( z = 0 \), so that our original integrand \( f(z) \) has no pole (and hence no residue) at \( z = \infty \).  We therefore find that the residues at \( z = 0 \) and at \( z = -2 \) must cancel, giving
		\begin{equation}\label{eqn:sumofres}
			 \frac{1}{2 \pi \ii} \oint_{C_{\varepsilon}(0)} \frac{(1+z)^{2r} - 1}{(2+z)^{2k} z^{2\ell}} \mathrm{d}z + \frac{1}{2 \pi \ii}  \oint_{C_{\varepsilon}(-2)} \frac{(1+z)^{2r} - 1}{(2+z)^{2k} z^{2\ell}} \mathrm{d}z = 0 \,.
		\end{equation}
		Now put \( z \mapsto -2-w \), with \( \mathrm{d}z = -\mathrm{d}w \), in the second integral, and we find
		\[
		\frac{1}{2 \pi \ii}  \oint_{C_{\varepsilon}(-2)} \frac{(1+z)^{2r} - 1}{(2+z)^{2k} z^{2\ell}} \mathrm{d}z = \frac{1}{2 \pi \ii}  \oint_{C_{\varepsilon}(0)} \frac{(-1-w)^{2r} - 1}{(-w)^{2k} (-2-w)^{2\ell}} \mathrm{d}w \,.
		\]
		Putting this back into \autoref{eqn:sumofres} shows that
		\[
		 \frac{1}{2 \pi \ii} \oint_{C_{\varepsilon}(0)} \frac{(1+z)^{2r} - 1}{(2+z)^{2k} z^{2\ell}} \mathrm{d}z - \frac{1}{2 \pi \ii}  \oint_{C_{\varepsilon}(0)} \frac{(1+w)^{2r} - 1}{(w)^{2k} (2+w)^{2\ell}} \mathrm{d}w = 0 \,,
		\]
		and so establishes the symmetry in \( k \leftrightarrow \ell \) that we claimed.
	\end{proof}
\end{lem}

\subsection{\texorpdfstring{Computing the dimension of $\im P_e$}{Computing the dimension of im P\_e}}\label{sec:evenetadimension}

\blockdimcount*
\begin{proof}
Suppose $q(x_1,x_2,x_3)\in \ker P_e$.
Then \autoref{rel3} implies
\begin{equation*}
\begin{split}
q(x_1,x_2,x_3)=\sum_{i+j=2n} \begin{aligned}[t] \alpha_{i,j} & (x_1-x_2)^i(x_2-x_3)^j+\beta_{i,j}(x_1+x_2)^i(x_2-x_3)^j\\[-1ex]
& {} + \gamma_{i,j}(x_1-x_2)^i(x_2+x_3)^j+\delta_{i,j}(x_1+x_2)^i(x_2+x_3)^j \,. \end{aligned}
\end{split}
\end{equation*}
Define $q_\star(x_1,x_2,x_3)\coloneqq \frac{1}{4}\left(q(x_1,x_2,x_3)-q(-x_1,x_2,x_3)-q(x_1,x_2,-x_3)+q(-x_1,x_2,-x_3)\right)$;\linebreak[4] this is the part of $q$ that is odd in $x_1\text{ and } x_3$ and even in $x_2$. We can write 
\begin{equation*}
\begin{split}
q_\star(x_1,x_2,x_3)=\sum_{\substack{i+j=2n\\ i,j>0}} \begin{aligned}[t] \rho_{i,j}\big( & (x_1-x_2)^i(x_2-x_3)^j+(-1)^{i+1}(x_1+x_2)^i(x_2-x_3)^j \\
&{} -(x_1-x_2)^i(x_2+x_3)^j+(-1)^i(x_1+x_2)^i(x_2+x_3)^j \big) \end{aligned}
\end{split}
\end{equation*}
where $\rho_{i,j}\coloneqq \alpha_{i,j}+(-1)^{i+1}\beta_{i,j}-\gamma_{i,j}+(-1)^i\delta_{i,j}$. As $q(x_1,x_2,x_3)=-q(x_3,x_2,x_1)$, the same holds for $q_\star(x_1,x_2,x_3)$ and thus $\rho_{i,j}=-\rho_{j,i}$.

Then, as $q_e(x_1,x_2,x_3)=0$, and $q(x_1,x_2,x_3)=q(x_2,x_3,x_1)$, we must have 
\[q(x_1,x_2,x_3)=q_\star(x_1,x_2,x_3)+q_\star(x_2,x_3,x_1)+q_\star(x_3,x_1,x_3)\,.\]
Thus, $q$ is uniquely determined by $q_\star$. We currently have $n-1$ free variables in $q_\star$, so in order for $\dim \ker P_e$ to be equal to $\lfloor\frac{n-1}{2}\rfloor$, we need \eqref{rel2} to impose $\lceil\frac{n-1}{2}\rceil$ independent constraints on the $\rho_{i,j}$.

Writing \eqref{rel2} in terms of $q_\star(x_1,x_2,x_3)$, we find that we must have
$$q_\star(z,0,y)=2q_\star(z,y,y)-2q_\star(y,z,z).$$
Evaluating the coefficient of $y^kz^l$ in this equation we obtain
$$\rho_{l,k}=\sum_{\substack{0<j\leq k\\i+j=2n}}(-2)^j\binom{i}{l}\rho_{i,j} - \sum_{\substack{0<j\leq l\\i+j=2n}}(-2)^j\binom{i}{k}\rho_{i,j}$$
if $k$ is odd, and $0=0$ if $k$ is even, or if $k=l$. As the coefficient of $y^lz^k$ is just the negative of this, this gives us $\lceil\frac{n-1}{2}\rceil$ equations, so it suffices to show that they are independent. As we are solving for rational $\rho_{i,j}$, it is sufficient to show that these equations are independent modulo 2. But modulo 2 we obtain the system of equations
$$\{\rho_{l,k}\equiv 0 \pmod*{2} \} \,,$$
which are clearly independent. Hence, we have $\lfloor\frac{n-1}{2}\rfloor$ free variables in $q_\star$ and $\dim \ker P_e=\lfloor\frac{n-1}{2}\rfloor$.

Similarly, if $q(x_1,x_2,x_3)\in\im P_e$, then
\begin{equation*}
\begin{split}
q(x_1,x_2,x_3)=\sum_{\substack{i+j=2n\\ i,j\geq 0}} \begin{aligned}[t] \eta_{i,j}\big( & (x_1-x_2)^i(x_2-x_3)^j+(-1)^{i}(x_1+x_2)^i(x_2-x_3)^j \\
& {} +(x_1-x_2)^i(x_2+x_3)^j+(-1)^i(x_1+x_2)^i(x_2+x_3)^j\big) \,. \end{aligned}
\end{split}
\end{equation*}
Indeed, the set
\begin{align*}
\mathcal{Q}\coloneqq  \big\{&(x_1-x_2)^i(x_2-x_3)^j+(-1)^{i}(x_1+x_2)^i(x_2-x_3)^j\\
&+(x_1-x_2)^i(x_2+x_3)^j+(-1)^i(x_1+x_2)^i(x_2+x_3)^j\big\}_{i+j=2n}
\end{align*}
 forms a basis for the space of totally even solutions of
$$\diffrbg{q}=0 \,.$$
Then, as $\frac{1}{2}\left(q(0,y,z)-q(0,y,-z)\right)=q(-y,y,z)-q(y,-z,z)$ holds trivially for any totally even polynomial satisfying the symmetry conditions, it is sufficient to compute the dimension of the subspace of skew-symmetric polynomials spanned by $\mathcal{Q}$. This is a simple representation theoretic argument: we consider $\mathbf{Span}(\mathcal{Q})$ as a representation of the symmetric group $\mathcal{S}_3$ via the standard polynomial representation, and compute the dimension of the sign representation within this. In particular, representation theory says that
\begin{equation*}
\begin{split}
\dim \im P_e& {} =\frac{1}{6}\big[\Tr(\id)-3\Tr((1\,3))+2\Tr((1\,2\,3))\big]\\
& {} =\frac{1}{6}\big[2n+1-3\Tr((1\,3))+2\Tr((1\,2\,3))\big]\,.
\end{split}
\end{equation*}
Note that the vector spaces generated by $\{(x_1-x_2)^i(x_2-x_3)^j\}_{i+j=2n} $ is invariant under the action of $\mathcal{S}_3$, and  there is a natural surjection onto $\mathcal{Q}$, so it is sufficient to consider the trace of the action restricted to $\{(x_1-x_2)^i(x_2-x_3)^j\}_{i+j=2n}$ in order to get an upper bound.

Clearly, the trace of $(1\,3)$ is $1$, as the only diagonal entry corresponds to $(x_1-x_2)^n(x_2-x_3)^n\mapsto (x_3-x_2)^n(x_2-x_1)^n$. Now, computing the trace of $(1\,2\,3)$, we find that it is given by
$$\sum_{i=0}^{2n}(-1)^i\binom{2n-i}{i}\,.$$
To compute this, we consider the generating series
\begin{equation*}
\begin{split}
\sum_{n\geq 0}\sum_{i=0}^{n}\binom{n-i}{i}x^iy^{n}&=\sum_{k\geq 0}\sum_{i\geq 0}\binom{k}{i}(xy)^iy^k\\
&=\sum_{k\geq 0}(1+xy)^ky^k\\
&=\frac{1}{1-y-xy^2}\,.
\end{split}
\end{equation*}
Setting $x=-1$, we obtain
\begin{equation*}
\begin{split}
\sum_{n\geq 0}\sum_{i=0}^n (-1)^i\binom{n-i}{i}(-y)^n&=\frac{1}{1+y+y^2}\\
&=\frac{1-y}{1-y^3}\\
&=\sum_{m\geq 0}y^{3m}-y^{3m+1}\,.
\end{split}
\end{equation*}
Thus, 
\[\sum_{i=0}^{2n}(-1)^i\binom{2n-i}{i} =
\begin{cases}
1 &\text{if $2n\equiv 0 \pmod{6}$}\\
-1 &\text{if $2n\equiv 4 \pmod{6}$}\\
0 &\text{if $2n\equiv 2 \pmod{6}$}
\end{cases}
\]
Hence
$$\dim\im P_e\leq \frac{1}{6}\left(2n+1-3 + 2x\right),$$
where $x$ is determined by $2n \pmod*{6}$. A quick consideration of each case shows we obtain $\lfloor\frac{2n}{6}\rfloor=\lfloor\frac{n}{3}\rfloor$.
\end{proof}
\begin{remark}
Recall that the weight $2n+2$ part of $\mathfrak{rbg}_2$ is a subspace of $V_n$ of dimension $\dim \gmot_{2,2n+2}$. In order to describe this subspace in terms of relations, we need to find non-zero linear maps $\{R_i:V_n\to W_i\}_{i\in I}$  such that $\mathfrak{rbg}_{2,2n+2}\subset \cap_{i\in I} \ker R_i$. However, the projection $\mathfrak{rbg}_{2,2n+2}\to \ker P_e$ is an isomorphism, and, as a consequence of \autoref{prop:alldouble}, $P_e(\mathfrak{rbg}_{2,2n+2})=\im P_e$. As such, no such $R$ can impose any additional relations that restrict non-trivially to either the totally even or not-totally-even parts. Equivalently, any such $R$ must induce a map $\ker P_e\to \Gr(k_R,\im P_e)$ to the space of $k_R$-dimensional subspaces of $\im P_e$ for some unique integer $k_R$. From another perspective, if such a description of $\mathfrak{rbg}_{2,2n+2}$ can be found, this would give an alternative proof of \autoref{conj:depthzetadim} in depth 2.
\end{remark}

\section{\texorpdfstring{Applications to multiple \( t \) values}{Applications to multiple t  values}}

From \cite{HoffmanOdd19}, we recall the multiple $t$ value \( t(n_1,\ldots,n_d) \) is defined by restricting the denominators in the series defining an MZV to be odd.  Namely
\[
    t(k_1,\ldots,k_d)\coloneqq\sum_{0<n_1<\cdots<n_d}\frac{1}{(2n_1-1)^{k_1}\cdots (2n_d-1)^{k_d}} \,.
\]

By inserting the factor \( \frac{1}{2} (1 - (-1)^{i_j}) \) into the numerator, one may extend the sum to all denominators, and obtain the following expression \cite[Corollary 4.1]{HoffmanOdd19}
for \( t(n_1,\ldots,n_d) \) in terms of alternating MZVs (with various signs) of the same set of indices
\begin{align}
    t(k_1,\ldots,k_d) &= \frac{1}{2^d} \sum_{0 < n_1 < \cdots < n_d} \frac{(1  -(-1)^{n_1}) \cdots (1 - (-1)^{n_d})}{n_1^{k_1} \cdots n_d^{k_d}} \notag \\
    \label{eqn:tasaltz} &= \frac{1}{2^d} \sum_{\eps_1,\ldots,\eps_d \in \{ \pm 1 \}} \eps_1 \cdots \eps_d \zeta(\eps_1 \diamond k_1, \cdots \eps_d \diamond k_d)
    \,.
\end{align}
Here the operator \( \diamond \) is defined so that \( 1 \diamond x = x \) and \( -1 \diamond x = \overline{x} \), where as usual \( \overline{n_j} \) denotes the argument \( n_j \) is accompanied with sign \( \eps_i = -1 \) in the definition of an alternating MZV (giving character \( (-1)^{n_j} \) in the numerator thereof). \medskip

From Murakami \cite[Theorem 1]{MurakamiMtVs21}, we know that every multiple \( t \) value with all arguments \( \geq 2 \) -- which would \emph{a priori} be a combination of alternating MZVs -- satisfies a Galois descent, and is expressible as a \( \mathbb{Q} \)-linear combination of classical multiple zeta values.  Murakami's Theorem is actually a statement about motivic multiple \( t \) values, but gives the same descent for classical M$t$Vs after applying the period map.  However Murakami's result is purely existential does not give an explicit formula, nor does it put any limits on the depth of the resulting combination.  Using the result \autoref{prop:zevbarevbar:eval} for alternating double zeta values, we will give explicit formulae in terms of depth 2 classical MZVs for any \( t(\mathrm{ev}, \mathrm{ev}) \) in \autoref{prop:tevev}.

\begin{remark}[Galois descent of \( t(\od,\ev) \) and \( t(\ev,\od) \)]
	Observe that the depth-parity theorem in depth 2 for alternating MZVs \cite[Equation 3.5]{PanzerParity16} implies that every multiple \( t(a,b) \) value of odd weight (with \( a, b \neq 1 \)) is a polynomial in single zeta values.  This already gives an explicit formula for the Galois descent of \( t(\text{od},\text{ev}) \) and \( t(\text{ev}, \text{od}) \).  Equivalent formulae were derived in \cite[Theorems 4.1, and 4.2]{XuParametric} using contour integral techniques (compare \cite{Flajolet} for classical MZVs handled in a similar way), namely 
	\begin{align*}
	t(2a+1,2b) & {} = \begin{aligned}[t]
	& t(2a+1)t(2b) - \frac{1}{2} t(2a+2b+1)
	\\
	& - \sum_{s=1}^{a+b} \bigg\{ \binom{2a+2b-2s}{2a} + \binom{2a+2b-2s}{2b-1} \bigg\} \frac{\zeta(2a+2b+1-2s)}{2^{2a+2b+1-2s}} t(2s) \,, \end{aligned} \\[2ex]
	t(2a,2b+1) & {} = \begin{aligned}[t]
	& - \frac{1}{2} t(2a+2b+1)
	\\
	& + \sum_{s=1}^{a+b} \bigg\{ \binom{2a+2b-2s}{2b} + \binom{2a+2b-2s}{2a-1} \bigg\} \frac{\zeta(2a+2b+1-2s)}{2^{2a+2b+1-2s}} t(2s) \,. \end{aligned} 
	\end{align*}
	(One has that \( t(a) = (1 - 2^{-a}) \zeta(a) \), for \( a > 1 \), which can be applied to rewrite the above purely in terms of Riemann zeta values.)
\end{remark}

\begin{remark}[Galois descent of \( t(\od,\od) \)]\label{rem:galoisdepth}
    On the other hand, the remaining case involving \( t(\od,\od) \) is less tractable.  Using the MZV Data Mine \cite{MZVDM}, one can check the following relation \[
        t(3,9) = 
        \begin{aligned}[t]
        & \frac{9}{128} \zeta (1,1,4,6)
        +\frac{1305}{4096}\zeta (3,9)
        -\frac{27}{128} \zeta (2) \zeta (3,7)
        -\frac{27}{256} \zeta (4) \zeta (3,5) \\
        & +\frac{3131}{2048}\zeta (9) \zeta (3)
        -\frac{321}{1024} \zeta (5) \zeta(7)
        -\frac{3}{512} \zeta (3)^4
        -\frac{45}{64} \zeta (2)  \zeta (7) \zeta (3)
        -\frac{63}{256} \zeta (2) \zeta (5)^2
         \\ 
         & +\frac{9}{256} \zeta (4) \zeta (5) \zeta (3)
        +\frac{81}{512} \zeta (6) \zeta (3)^2
        +\frac{353139}{5660672}\zeta (12) \,.
        \end{aligned} 
    \]
    In particular, the (conjecturally) irreducible depth 4 MZV \( \zeta(1,1,4,6) \) (or any equivalent choice) is necessary to obtain an expression for the Galois descent of \( t(3,9) \) to classical MZVs.  This already suggests describing the Galois descent explicitly (with the minimal necessary depth) would be challenging.
    
    We can, conjecturally at least, say that depth 4 MZVs will be sufficient. Indeed, since we may write $t(a,b)$ as a sum of depth 2 alternating MZVs, the alternating analogue of \autoref{lem:depthsubblock} tells us that $t(a,b)$ lies in coradical degree at most 2. Hence, if a Galois descent to classical MZVs exists, it must also like in coradical degree at most 2. When depth 2 MZVs do not span this space in even weight, the homological version of the Broadhurst-Kreimer Conjecture \cite[Conjecture 5]{BrownDepthGraded21} tells us that depth 2 MZVs along with irreducible depth 4 MZVs coming from cusp forms do.
        
    More generally, if a depth $d$ multiple $t$ value has a Galois descent to classical MZVs, the same line of reasoning tells us that we should expect an expression in terms of MZVs of depth at most $2d$.
\end{remark}

By combining the usual expression for \( t(a,b) \) in terms of alternating MZVs \cite[Corollary 4.1]{HoffmanOdd19}, namely
\[
	t(a,b) = \frac{1}{4} \big( \zeta(a,b) + \zeta(a, \overline{b}) + \zeta(\overline{a},b) + \zeta(\overline{a},\overline{b}) \big)
\]
with the distribution relation \cite[Proposition 2.13]{GoncharovMPL01}
\[
	\zeta(a,b) + \zeta(a,\overline{b}) + \zeta(\overline{a},b) + \zeta(\overline{a},\overline{b}) = \frac{1}{2^{a+b-2}} \zeta(a,b) \,,
\]
we can write
\begin{equation}\label{eqn:tab:evenbars}
	t(a,b) = \frac{1}{2} \zeta(\overline{a},\overline{b}) + \frac{1}{2} \zeta(a,b) - \frac{1}{2^{a+b}} \zeta(a,b) \,.
\end{equation}
(More generally, see the alternative expression given by Hoffman, using a sum which inserts only an even number of bars into the argument string \cite[Corollary 4.2]{HoffmanOdd19}.)

Let us now note the following result from \autoref{sec:gendouble}, which gives an explicit form for the Galois descent of \( \zeta(\overline{2\ell}, \overline{2k}) \) in terms of classical double MZVs.

\begin{restatable*}[Galois descent of \( \zeta(\overline{2\ell},\overline{2k}) \)]{prop}{propgaldescent} \label{prop:zevbarevbar:eval}
	The alternating double zeta value \( \zeta(\overline{2\ell},\overline{2k}) \) enjoys a Galois descent to classical depth 2 MZVs as follows
	\begin{equation}
		\label{eqn:zaltevev}
		\begin{aligned}[c]
		\zeta(\overline{2\ell},\overline{2k}) = 
		\begin{aligned}[t]
			& \sum_{i=2}^{2k + 2\ell - 2} 2^{-i} \bigg\{ \binom{i-1}{2k-1} \zeta(2k+2\ell-i,i) + \binom{i-1}{2\ell-1} \zeta(i,2k+2\ell-i) \bigg\} \\
			& -\zeta(2\ell,2k) + \sum_{r=2}^{2k+2\ell-2} (-2)^{-r} \binom{r-1}{2k-1}  \zeta(r) \zeta(2k+2\ell-r) \\
			& - 2^{-2k-2\ell} \bigg\{ 2 \binom{2k+2\ell-2}{2k-1} + \binom{2k+2\ell-1}{2k-1} \bigg\} \zeta(2k + 2\ell) \,.
		\end{aligned}
		\end{aligned}
	\end{equation}
\end{restatable*}

\begin{proof}[Proof sketch]
    We recall the notation \( \zeta_\ell(k_1,\ldots,k_d) \) is defined by inserting \( \ell \) leading 0's at the start of the integral representation of \( \zeta(k_1,\ldots,k_d) \) (c.f. \autoref{eqn:zasint} or \autoref{eqn:altmzv:int} for alternating MZVs).  Now simultaneously solve the following equations: the dihedral symmetry \autoref{eqn:dih:1and2lbar}
    \begin{equation*}
        \begin{aligned}
 \zeta_{2k-1}(1, \overline{2\ell}) - {} & \zeta(\overline{2\ell}, \overline{2k}) = \binom{2k+2\ell-1}{2k-1} \zeta(\overline{2k+2\ell}) \\ & - \sum_{\substack{r = 1}}^{2k+2\ell-2} \bigg( (-1)^r \binom{r-1}{2k-1} + \binom{r-1}{2\ell-1}  \! \bigg) \zeta(\overline{r}) \zeta(2k+2\ell-r) \,, \end{aligned}
\end{equation*}
    and the generalised doubling identity \cite[\S4]{MZVDM}, \cite[Section 14.2.5]{ZhaoBook}
    \begin{align*}
	& \zeta(\overline{s},\overline{t}) + (-1)^t \zeta_{t-1}(1,\overline{s}) = \\
	& \sum_{i=1}^s \binom{s+t-i-1}{t-1} 2^{1+i-s-t} \zeta(i, s+t-i) + \sum_{i=1}^t \binom{s+t-i-1}{s-1} 2^{1+i-s-t} \zeta(s+t-i, i) \\
	& - \zeta(s,t) + (-1)^t \zeta_{t-1}(s,1) - \sum_{i=1}^t \binom{s+t-i-1}{s-1} \zeta(\overline{s+t-i})\zeta(i)
	- \binom{s+t-1}{s} \zeta(s+t) 
\end{align*}
(here slightly rewritten, see \autoref{sec:gendouble}) in the case \( t = 2k, s = 2\ell \).
\end{proof}

Now substituting this Galois descent into \autoref{eqn:tab:evenbars}, we immediate have the following proposition.

\begin{prop}\label{prop:tevev}
	The multiple \( t \) value \( t(2\ell,2k) \) is expressed through classical double zeta values as follows
		\begin{equation}
		\begin{aligned}[c]
		t(2\ell,2k) = 
		\begin{aligned}[t]
		& \sum_{i=2}^{2k + 2\ell - 2} 2^{-i-1} \bigg\{ \binom{i-1}{2k-1} \zeta(2k+2\ell-i,i) + \binom{i-1}{2\ell-1} \zeta(i,2k+2\ell-i) \bigg\} \\
		& -2^{-2k-2\ell} \zeta(2\ell,2k) - \sum_{r=2}^{2k+2\ell-2} (-2)^{-r-1} \binom{r-1}{2k-1}  \zeta(r) \zeta(2k+2\ell-r) \\
		& - 2^{-2k-2\ell-1} \bigg\{ 2 \binom{2k+2\ell-2}{2k-1} + \binom{2k+2\ell-1}{2k-1} \bigg\} \zeta(2k + 2\ell) \,.
		\end{aligned}
		\end{aligned}
		\end{equation}
\end{prop}

\bibliographystyle{habbrv}
     \bibliography{bibliography}

\appendix
\section{\texorpdfstring{Analytic evaluation of \( \zeta(\{2\}^a, 4, \{2\}^b) \) via double zeta values}{Analytic evaluation of zeta(\{2\}\textasciicircum{}a, 4, \{2\}\textasciicircum{}b)}}
\label{sec:num2242}

The goal of this section is to given an explicit evaluation for \( \zeta(\{2\}^a, 4, \{2\}^b) \) in terms of double zeta values on the analytic (numerical) level.  In \autoref{sec:mot2242} we will then lift this to the corresponding identity among motivic MZVs. \medskip

For the numerical evaluation, we need to assemble a number of ingredients.  In particular, we need to use the stuffle antipode (\cite{HoffmanQuasi20}, \cite{LiIdentities13} or \cite{GlanoisThesis16,GlanoisBasis16}) to convert \( \zeta(\{2\}^a, 4, \{2\}^b) \) to a corresponding multiple zeta star value.  Then we can apply Zhao's generalised 2-1 Theorem \cite{ZhaoIdentity16} (in the block decomposition form \cite{CharltonCyclic20} for convenience) to rewrite the zeta star value as an alternating zeta-half value.  By application of the Parity Theorem (\cite{PanzerParity16}, or \cite{GoncharovGalois01}), we reduce this to an explicit combination of depth 2 alternating MZVs.  It becomes convenient to write these (combinations of) alternating double zeta values as certain shuffle-regularised alternating double zetas \( \zeta_z^{\shuffle,0}(r,\overline{s}) \) with a number of initial zeros; this presentation then manifests a dihedral symmetry modulo products and lower depth \cite{GlanoisThesis16},\cite{GlanoisBasis16}, which we can describe explicitly.  Finally (perhaps surprisingly) by combining this dihedral symmetry with a generalised doubling identity \cite{MZVDM,ZhaoBook}, one can explicitly evaluate these alternating double zeta values in terms of classical double zeta values (as opposed to higher depth MZVs which would certainly suffice by the generalised 2-1 Theorem). \medskip

\paragraph{\bf Alternating and interpolated MZVs:} Let us recall again the notions of alternating MZVs, and of multiple zeta star values and multiple zeta-half values, which will be useful imminently.  Given a tuple \( (k_1,k_2,\ldots,k_d) \) of positive integers, and a tuple \( (\eps_1,\eps_2,\ldots,\eps_d) \in \{ \pm 1 \}^d \), with \( (k_d,\eps_d) \neq (1,1) \), we define the \emph{alternating MZV} (or Euler sum) with signs \( \eps_1,\ldots,\eps_d \) as follows,
\[
	\zeta\bigg(\begin{matrix} \eps_1,\eps_2,\ldots,\eps_d \\ k_1,k_2,\ldots,k_d \end{matrix}\bigg) \coloneqq \sum_{0 < n_1 < n_2 < \cdots < n_d} \frac{\eps_1^{n_1} \eps_2^{n_2} \cdots \eps_d^{n_d}}{n_1^{k_1} n_2^{k_2} \cdots n_d^{k_d}} \,.
\]
One then streamlines the notation by suppressing the \( \eps_i \)'s and writing \( \overline{k_i} \) if \( \eps_i = -1 \), and just \( k_i \) if \( \eps_i = 1 \) otherwise.  For example
\[
	\zeta(k_1,\overline{k_2},\overline{k_3})  \coloneqq 	\zeta\bigg(\begin{matrix} 1,& \!\!-1,& \!\!-1 \\ k_1,& \!\!k_2,& \!\!k_3 \end{matrix}\bigg) \coloneqq \sum_{0 < n_1 < n_2 < n_3} \frac{(-1)^{n_2} (-1)^{n_3}}{n_1^{k_1} n_2^{k_2} n_3^{k_3}} \,.
\]
An alternating MZV can be written as an iterated integral in the following way
\begin{equation}\label{eqn:altmzv:int}
	\zeta\bigg(\begin{matrix} \eps_1,\eps_2,\ldots,\eps_d \\ k_1,k_2,\ldots,k_d \end{matrix}\bigg) = (-1)^d I^\mot(0; \eta_1, \{0\}^{k_1-1}, \eta_2, \{0\}^{k_2-1}, \ldots, \eta_d, \{0\}^{k_d-1} ;1) \,,
\end{equation}
where \( \eta_i = \eps_i \eps_{i+1} \cdots \eps_d \). 

Next, we have the interpolated multiple zeta values \( \zeta^r(k_1,\ldots,k_d) \) introduced by Yamamoto \cite{YamamotoInterp},
\[
	\zeta^r(k_1,\ldots,k_d) \coloneqq \sum_{\text{$\circ_i = \text{``$+$'' or ``$,$''}$}} r^{\#\{ i \mid \circ_i = \text{``+''}\}} \zeta(k_1 \, \circ_1 \, k_2 \, \circ_2 \, \cdots \, \circ_{r{-}1} \, k_d) \,.
\]
For example \(
	\zeta^r(a,b,c) = \zeta(a,b,c) + r \zeta(a+b,c) + r \zeta(a, b+c) + r^2 \zeta(a+b+c) \).
In the case \( r = 0 \), only the term with all \( \circ_i = \text{``$,$''} \) survives, and so \( \zeta^0(k_1,\ldots,k_d) = \zeta(k_1,\ldots,k_d)  \).  When \( r = 1 \), then we have \( \zeta^1(k_1,\ldots,k_d) = \zeta^\star(k_1,\ldots,k_d) \), where the \emph{multiple zeta star value} (MZSV) is originally defined as
\[
	\zeta^\star(k_1,\ldots,k_d) \coloneqq \sum_{0 < n_1 \leq n_2 \leq \cdots \leq n_d} \frac{1}{n_1^{k_1} n_2^{k_2} \cdots n_r^{k_d}} \,,
\]
and arises by replacing the strict inequalities between \( n_i, n_{i+1} \) with a non-strict inequality.  For \( r = \half \), we then obtain a new variant `mid-way' between \( \zeta \) and \( \zeta^\star \), called the \emph{multiple zeta-half value}.

This formalism can be extended to allow for alternating interpolated MZVs, by replacing \( \text{``$+$''} \), above with \( \text{``$\oplus$''} \), where \( a \oplus b \) denotes addition of the absolute values and multiplication of the bars viewed as signs (i.e. if \( k \in \mathbb{Z}_{>0} \), we have\( \abs{k} = \abs{\smash{\overline{k}}} = k \) with \( \sgn(\overline{k}) = -1 \) and \( \sgn(k) = 1 \)).  In particular, for \( \alpha,\beta \in \mathbb{Z}_{>0} \), we have \( \alpha \oplus \beta = \alpha + \beta \), \( \alpha \oplus \overline{\beta} = \overline{\alpha} \oplus \beta = \overline{\alpha + \beta} \) and \( \overline{\alpha} \oplus \overline{\beta} = \alpha + \beta \).  Then for example we have the following interpolated alternating MZV
\begin{align*}
	\zeta^r(a, \overline{b},\overline{c}) & {} = \zeta(a,\overline{b},\overline{c})
	 + r \zeta(a\oplus \overline{b},\overline{c})
	 + r \zeta(a, \overline{b}\oplus\overline{c}) 
	 + r^2 \zeta(\overline{a}\oplus \overline{b} \oplus \overline{c}) \\
	 & {} = \zeta(a,\overline{b},\overline{c})
	 + r \zeta(\overline{a+b},\overline{c})
	 + r \zeta(a, b+c) 
	 + r^2 \zeta(\overline{a+b+c}) \,.
\end{align*}
The case \( r = \half \) of alternating interpolated MZVs is a convenient way of formulating Zhao's generalised 2-1 Theorem, as we will see below.

\subsection{Stuffle antipode}

We define \( G^\star \), the generating series of \( \zeta^\star(\{2\}^a, 4 ,\{2\}^b) \), \( G \), a related generating series for \( \zeta(\{2\}^a, 4, \{2\}^b) \), and \( S^\star \) the generating series of \( \zeta^\star(\{2\}^n) \) as follows.
\begin{align*}
	G^\star(x,y) & \coloneqq \sum_{a,b=0}^\infty \zeta^\star(\{2\}^a, 4, \{2\}^b) x^{2a} y^{2b} \,, \\
	G(x,y) & \coloneqq \sum_{a,b=0}^\infty (-1)^{a+b} \zeta(\{2\}^a, 4, \{2\}^b) x^{2a} y^{2b} \,, \\
	S^\star(x) & \coloneqq \sum_{n=0}^\infty \zeta^\star(\{2\}^n) x^{2n} \,.
\end{align*}
Then from \cite[Equation 2.4]{LiIdentities13} (in the special case \( a_2 = z_4 \), \( a_1 = a_3 = z_2 \)) we have that 
\begin{equation}\label{eqn:gs:zs2242asz2242}
	G^\star(x,y) = G(y,x) \ast S^\star(x) \ast S^\star(y) \,.
\end{equation}
This is an identity in the stuffle algebra; in particular, it automatically lifts to a motivic identity since the stuffle product is known to be motivic (see \cite{RacinetThesis} or \cite{SouderesDblSh}).
Moreover, it is well-known (or readily verifiable by factoring the generating series as a product, see for example \cite[Equation 36]{BBBkfold}, and \cite[Equation 44]{zaghoff}) that
\[
	\sum_{n=0}^\infty \zeta^\star(\{2\}^n) x^{2n} = \frac{\pi x}{\sin(\pi x)} \,, \text{ and }
	\sum_{n=0}^\infty \zeta(\{2\}^n) x^{2n} = \frac{\sin(\ii \pi x)}{\ii \pi x} \,.
\]
By solving \autoref{eqn:gs:zs2242asz2242} for \( G(y,x) \), and extracting the coefficient of \( x^{2a} y^{2b} \), we obtain following explicit formula for \( \zeta(\{2\}^a, 4, \{2\}^b) \) in terms of similar \( \zeta^\star \) values,
\begin{equation}
\label{eqn:z2242aszs2242}
	\zeta(\{2\}^a,4,\{2\}^b) = \sum_{n=0}^a \sum_{m=0}^b (-1)^{m+n} \zeta^\star(\{2\}^m, 4, \{2\}^n) \zeta(\{2\}^{a-n}) \zeta(\{2\}^{b-m}) \,.
\end{equation}

A similar identity holds for \( 4 \) replaced by any value \( k \); these identities gives the precise version of the stuffle antipode result
\[
	\zeta^\lmot(k_1,\ldots,k_d) = (-1)^{d+1} \zeta^{\star,\lmot}(k_d,\ldots,k_1)
\]
considered in \cite[Lemma 4.2.2]{GlanoisThesis16}.  Moreover, since the stuffle algebra identity and the evaluations of \( \zeta(\{2\}^n) \), \( \zeta^\star(\{2\}^n) \) are motivic (\cite[Lemma 3.4]{BrownMTM12}, \cite[Lemma 4.4.3]{GlanoisThesis16}), \autoref{eqn:z2242aszs2242} lifts automatically to a motivic version.

\subsection{Generalised 2-1 Theorem} We now recall the generalised 2-1 Theorem, established by Zhao \cite{ZhaoIdentity16}, which evaluates each \( \zeta^\star \) value in terms of a certain alternating \( \zeta^\half \) value.  It is more convenient -- and indeed has a closer connection with the goal -- to write the generalised 2-1 Theorem in the block decomposition form given in \cite[Lemma 3.1]{CharltonCyclic20}.

Let \( \vec{s} = (s_1,\ldots,s_k) \) be a sequence of MZV arguments, and let \( \vec{B} = (\ell_1,\ldots,\ell_n) \) be the corresponding block decomposition (see \autoref{sec:blockdecomp}).  Write
\[
	\widetilde{\,x\,} = \begin{cases}
		x & \text{if \( x \) odd}, \\
 		\overline{x} & \text{if \( x \) even}.
	\end{cases}
\]  (Recall: \( \overline{x} \) denotes that the argument \( x \) in an alternating MZV has sign \( -1 \).)  Then
\[
	\zeta^\star(\vec{s}) = \varepsilon(\vec{s}) \cdot 2^n \zeta^\half(\widetilde{\ell_1 - 2}, \widetilde{\ell_2}, \ldots, \widetilde{\ell_n}) \,,
\]
where \( \varepsilon(\vec{s}) = 1 \) if \( s_1 = 1 \) and  \( \varepsilon(\vec{s}) = -1 \) if \( s_1 \geq 2 \), and if \( \ell_1-2 = 0 \), one should neglect this argument.  This follows by combining Zhao's generalised 2-1 Theorem \cite{ZhaoIdentity16} which involves a certain recursively constructed sequence of indices \( \vec{s}^{(i)} \), with the description of the final such index string \( \vec{s}^{(k)} \), given in \cite[Lemma 3.1]{CharltonCyclic20}, in terms of the block decomposition.  This final string supplies the \( \zeta^\half \) arguments in Zhao's formulation of the 2-1 Theorem. \medskip

In our case, we want to apply this to \( \zeta^\star(\{2\}^a, 4, \{2\}^b) \).  The block decomposition of \( (\{2\}^a, 4, \{2\}^b) \) is given by \( (2a+3,1,2b+2) \), and we therefore have
\[
	\zeta^\star(\{2\}^a, 4, \{2\}^b) = -2^3 \zeta^\half(2a + 1, 1, \overline{2b+2}) \,.
\]
Then expanding out, by definition of the interpolated \( \zeta^\half \), we have
\begin{equation}
\label{eqn:zs2242aszhalfdp3}
\begin{aligned}
	& \zeta^\star(\{2\}^a, 4, \{2\}^b) = \begin{aligned}[t] -2 & \zeta(\overline{2a + 2b + 4}) - 4 \zeta(2a+1, \overline{2b + 3}) \\ 
	& - 4 \zeta(2a + 2, \overline{2b + 2}) - 8 \zeta(2a + 1, 1, \overline{2b+2}) \,.
	\end{aligned}
\end{aligned}
\end{equation}
This reduces our task of evaluating \( \zeta^\star(\{2\}^a, 4, \{2\}^b) \) to understanding certain explicit depth 3 alternating MZVs.

\subsection{The parity theorem in depth 3}

The parity theorem for MZVs states roughly that an MZV of weight \( w \) and depth \( d \) can be reduced to a combination of lower depth MZVs and products, when \( w \not\equiv d \pmod*{2} \).  In particular, an MZV of depth 3 and even weight is reducible.  This claim actually also holds for alternating MZVs, via the parity theorem for multiple polylogarithms \cite{PanzerParity16}, as \( -1 \) is its own multiplicative inverse. \medskip

An explicit version of the depth 3 parity theorem is given for the multiple polylogarithm functions \( \Li_{n_1,n_2,n_3}(z_1,z_2,z_3) \) in \cite[Equation 4.3]{PanzerParity16}.  By specialising to \( z_i = \pm 1 \), we recover the claimed reduction of depth 3 alternating MZVs, for any choice of signs \( z_i \) (encoded with a `bar' over the corresponding argument \( \overline{n_i} \), if \( z_i = -1 \)), as follows.  Namely if \( \alpha + \beta + \gamma \) is even and \( \gamma \neq 1 \) (although \( \gamma = \bar{1} \) is okay), then with $\zeta(0) = \zeta(\overline{0}) = -\frac{1}{2}$ by convention, and stuffle-regularisation if necessary (see \cite{IKZ06} for the notion of regularisation, and \autoref{rem:parityreg} below for the behaviour in this case) when $\zeta(1)$ appears, we have
\begin{align*}
\zeta(\alpha,\beta,\gamma) &= \begin{aligned}[t] 
& \frac{1}{2} \zeta(\alpha) \big( \zeta(\beta,\gamma) - (-1)^{\abs{\beta}+\abs{\gamma}} \zeta(\beta,\gamma) \big) - \zeta(\beta,\alpha) \zeta(\gamma) \delta_{\text{$\abs{\gamma}$ even}} \\
&- \frac{1}{2} \zeta(\alpha \oplus \beta, \gamma) + \frac{1}{2} \zeta(\beta \oplus \gamma, \alpha) \\
&+ \sum_{\substack{2s + \nu + \mu = \abs{\smash{\alpha}} \\ s, \mu, \nu \geq 0}} (-1)^{\abs{\beta} + \abs{\gamma} + \mu+ \nu} \zeta(\sgn(\alpha\beta\gamma) \diamond 2s)\binom{-\abs{\beta}}{\mu} \binom{-\abs{\gamma}}{\nu} \zeta(\beta\oplus \mu, \gamma \oplus \nu) \\
&+ \sum_{\substack{2s + \nu + \mu = \abs{\smash{\beta}} \\ s, \mu, \nu \geq 0}} (-1)^{\gamma + \mu} \zeta(\sgn(\alpha\beta\gamma) \diamond 2s)\binom{-\abs{\gamma}}{\mu} \binom{-\abs{\alpha}}{\nu} \zeta(\gamma \oplus \mu) \zeta( \alpha \oplus \nu) \\
&+ \sum_{\substack{2s + \nu + \mu = \abs{\smash{\gamma}} \\ s, \mu,\nu \geq 0}} \zeta(\sgn(\alpha\beta\gamma) \diamond 2s) \binom{-\abs{\beta}}{\mu} \binom{-\abs{\alpha}}{\nu} \zeta(\beta \oplus \mu, \alpha \oplus \nu) \,.
\end{aligned}
\end{align*}
(To avoid abuse of notation, we define \( 1 \diamond x \coloneqq x \), and \( -1 \diamond x \coloneqq \overline{x} \), to give the corresponding decoration for signed arguments.
) \medskip

Now specialise to \( \alpha = 2a + 1, \beta = 1, \gamma = \overline{2b+2} \).  We can simplifying various binomial coefficients and powers of \( -1 \), using \( \binom{-k}{\ell} = (-1)^\ell \binom{\ell + k-1}{k-1} \), and expand out the second summation into its two non-trivial terms \( (s,\mu,\nu) = (0, 1, 0), (0, 0, 1) \).  After doing so, and inserting the result into \autoref{eqn:zs2242aszhalfdp3} we note some simplifications.  Firstly the term
\( -4\zeta(2a+2,\overline{2b+2}) \) in the \( \zeta^\half \) cancels with one from the depth 3 reduction; secondly the term \( -4\zeta(2a+1,\overline{2b+3}) \) combines with one from the depth 3 reduction to produce
\begin{align*}
	& -4\zeta(2a+1,\overline{2b+3}) - 4\zeta(\overline{2b+3},2a+1) \\
	& = -4\zeta(2a+1)\zeta(\overline{2b+3}) + 4 \zeta(\overline{2a+2b+4})
\end{align*}
Overall this produces the following evaluation for \( \zeta(\{2\}^a, 4, \{2\}^b) \), as the first main stepping stone, with stuffle-regularisation applied where necessary
\begin{equation}\label{eqn:zs2242full1}
\begin{aligned}[c]
 \zeta^\star(&\{2\}^a, 4, \{2\}^b) = \\
 & \begin{aligned}[t]
& 2 \zeta(\overline{2a + 2b + 4}) + 8 \zeta(1, 2a+1) \zeta(\overline{2b+2}) - 8 \zeta(2a + 1)\zeta(1, \overline{2b + 2}) \\
& + 4 (2b + 1) \zeta(\overline{2b +3}) \zeta(2a + 1) - 4 (2a + 1) \zeta(\overline{2b + 2}) \zeta(2a + 2) \\
& + 8 \sum_{\substack{2s + \nu + \mu = 2a+1 \\ s, \mu, \nu \geq 0}} \zeta(\overline{2s}) \binom{\nu+(2b+1)}{\nu} \zeta(1 + \mu, \overline{2b + 2 + \nu}) \\
& - 8 \sum_{\substack{2s + \nu + \mu = 2b+2 \\ s, \mu, \nu \geq 0}} \zeta(\overline{2s}) \binom{\nu + (2a)}{\nu} \zeta(1 + \mu, 2a + 1 + \nu) \,.
\end{aligned}		
\end{aligned}
\end{equation}

\begin{remark}[Independence of regularisation]\label{rem:parityreg}
Let us note here that the shuffle-regularised and stuffle-regularised versions of this formula agree (and indeed also for the depth 3 reduction, with \( c \neq 1 \)), and are independent of the regularisation parameter; we may therefore switch to the shuffle-regularisation at \( T = 0 \) for later convenience.  This is expected since we are reducing a convergent triple zeta value.  

More precisely, this is because terms with a single trailing 1 are equal under either regularisation, and in the case \( a = 0 \), the single term  \( 8\zeta(\overline{2b+1})\zeta(1,1) \) with two trailing 1's, which arises arising from \( (s,\nu,\mu) = (b+1, 0, 0) \) in the last sum, cancels with the corresponding term on the second line.  Otherwise when \( a = 1 \), the regularisation parameter \( T \) in terms arising from \( \nu = 0 \) in the last sum (with \( \zeta^{\ast,T} \) explicitly denoting the stuffle-regularised version with \( \zeta^{\ast,T} (1) = T \)),
\[
	\sum_{\substack{2s + \mu = 2b + 2 \\ s \geq 0,\mu > 0}} \zeta(\overline{2s}) \zeta^{\ast,T}(1 + \mu, 1)
	= \sum_{\substack{2s + \mu = 2b + 2 \\ s \geq 0,\mu > 0}} \zeta(\overline{2s}) \big( T \zeta(1 + \mu) - \zeta(1, 1+\mu) - \zeta(2+\mu) \big)
\]
can be seen to cancel with that arising from the terms
\begin{align*}
	& 8 \zeta^{\ast,T}(1) \zeta(1,\overline{2b+2}) + 4 (2b+1) \zeta(\overline{2b+3}) \zeta^{\ast,T}(1)  \\
	&= 8 T \zeta(1, \overline{2b+2}) + 4 (2b+1) \zeta(\overline{2b+3}) T
\end{align*}
In fact this cancellation is equivalent to the following reduction which follows from the depth-parity theorem in depth 2 (see \cite[Equation 3.5]{PanzerParity16}):
\begin{equation}
\label{eqn:red:zeta1xbar}
	\zeta(1, \overline{2b+2}) = -\sum_{\substack{2s+k = 2b + 3 \\ s \geq 0, k \geq 2}} \zeta(\overline{2s}) \zeta(k) + \frac{2b+1}{2} \zeta(\overline{2b+3}) \,.
\end{equation}
\end{remark}

\subsection{Shuffle-regularisation and dihedral symmetries}	Now let us take advantage of the shuffle-regularisation in earnest.  Define \( \zeta_\ell(k_1,\ldots,k_d) \) is by inserting \( \ell \) leading 0's at the start of the iterated integral representation of \( \zeta(k_1,\ldots,k_d) \) given in \autoref{eqn:altmzv:int} (see also \autoref{eqn:zasint} and \autoref{eqn:altmzv:mot}), and write \( \zeta_\ell^{\shuffle,T=0}(k_1,\ldots,k_d) \) for the shuffle regularisation thereof (see \autoref{rem:regularisation} above), with \( \zeta^{\shuffle,T=0}(1) = 0 \).  Then we have the regularisation formula (see for example \cite[\S{}5.1 R2]{BrownDecomposition12}, and the obvious generalisation to alternating MZVs in \cite[Equation 2.28]{GlanoisThesis16})
\[
	\zeta_\ell^{\shuffle,T=0}(k_1,\ldots,k_d) = (-1)^\ell \! \sum_{i_1 + \cdots + i_d = \ell} \! \binom{k_1 + i_1 - 1}{i_1} \cdots \binom{k_r + i_r - 1}{i_r} \zeta(k_1 \oplus i_1, \ldots, k_r \oplus i_r) \,.
\]
In particular, using this we can write
\begin{align*}
	 & \sum_{\substack{\nu + \mu = 2a+1 - 2s \\ \mu, \nu \geq 0}} \binom{\nu+(2b+1)}{\nu} \zeta(1 + \mu, \overline{2b + 2 + \nu}) = - \zeta_{2a+1-2s}^{\shuffle,T=0}(1, \overline{2b+2}) \\
	& \sum_{\substack{\nu + \mu = 2b+2 - 2s \\ \mu, \nu \geq 0}} \binom{\nu + (2a)}{\nu} \zeta(1 + \mu, 2a + 1 + \nu) = \zeta_{2b+2-2s}^{\shuffle,T=0}(1, 2a+1) \,.
\end{align*}
Substituting these and \autoref{eqn:red:zeta1xbar} into \autoref{eqn:zs2242full1} gives us the following
\begin{equation*}
\begin{aligned}[c]
\zeta^\star(\{2\}^a, 4, \{2\}^b) = {}
 \begin{aligned}[t]
& 2 \zeta(\overline{2a + 2b + 4}) 
+ 8 \zeta(1, 2a+1) \zeta(\overline{2b+2}) 
 -4 (2a+1) \zeta(\overline{2b+2}) \zeta(2a+2) \\
&  + 8 \zeta(2a+1) \sum_{k=1}^{b+1} \zeta(2k+1) \zeta(\overline{2b+2-2k}) \\[-1ex]
&  - 8 \sum_{s = 0}^{a} \zeta(\overline{2s}) \zeta^{\shuffle,T=0}_{2a+1-2s}(1, \overline{2b+2}) - 8 \sum_{s = 0}^{b+1} \zeta(\overline{2s}) \zeta^{\shuffle,T=0}_{2b+2-2s}(1, 2a+1) \,.
\end{aligned}		
\end{aligned}
\end{equation*}

Let us finally note that the term \( 8\zeta(1,2a+1) \zeta(\overline{2b+2}) \) cancels with the \( s = b+1 \) term of the last sum.  In particular we obtain the second stepping stone in our quest to evaluate \( \zeta^{\star}(\{2\}^a, 4, \{2\}^b) \).  Namely
\begin{equation}\label{eqn:zs2242full2}
\begin{aligned}[c]
\zeta^\star(\{2\}^a, 4, \{2\}^b) =  {}
& \begin{aligned}[t]
& 2 \zeta(\overline{2a + 2b + 4}) 
 -4 (2a+1) \zeta(\overline{2b+2}) \zeta(2a+2) 
 \\
 & + 8 \zeta(2a+1) \sum_{k=1}^{b+1} \zeta(2k+1) \zeta(\overline{2b+2-2k}) \\[-1ex]
&  - 8 \sum_{s = 0}^{a} \zeta(\overline{2s}) \zeta^{\shuffle,T=0}_{2a+1-2s}(1, \overline{2b+2}) - 8 \sum_{s = 0}^{b} \zeta(\overline{2s}) \zeta^{\shuffle,T=0}_{2b+2-2s}(1, 2a+1) \,.
\end{aligned}		
\end{aligned}
\end{equation}

Henceforth we shall always use the shuffle regularisation with \( \zeta^{\shuffle,T=0}(1) = 0 \), and will therefore drop the extraneous \( \bullet^{\shuffle,T=0} \) from our notation.  This regularisation is consistent with the regularisation normally used in the motivic framework (c.f. \autoref{rem:regularisation}).  Moreover, we will check in \autoref{sec:zs2242dp3:mot} that this reduction is indeed motivic.\medskip

Now we recall from \cite[Corollary 4.2.6]{GlanoisThesis16} that depth \( p \) alternating zeta star values satisfy a dihedral symmetry of order \( p + 1 \), modulo products and lower depth.  (More precisely therein, this symmetry is phrased in terms of so-called multiple zeta star-star values, which incorporate the lower depth terms, making the symmetry hold already modulo products.)  In particular, in our case we claim that
\begin{align*}
 	\zeta_{2k-1}(1,\overline{2\ell}) &\equiv \zeta(\overline{2\ell},\overline{2k}) \pmod{\text{products}} \\
	\zeta_{2k}(1,2\ell+1) &\equiv \zeta(2\ell+1,2k+1) \pmod{\text{products}} \,.
\end{align*}
(The depth 1 terms in these cases are reducible, as the weight is even.)

We will not actually use the implicit form of the dihedral symmetry established by Glanois which immediately produces the above; instead, guided by the Glanois' proof, we will establish an exact version in this depth 2 case.  However, let us point out some technical issue which apparently occurs when considering the octagon relation, in an attempt to derive a so-called \emph{hybrid relation} (Theorem 4.2.3 in \cite{GlanoisThesis16}), a key part of the proof of the dihedral symmetry.

\begin{remark}[Regularisation in the octagon relation]
	The octagon relation for level \( N = 2 \) multiple zeta values (i.e. alternating MZVs) is obtained by integrating a word in \( e_0, e_1, e_{-1} \) around the following closed loop.
		
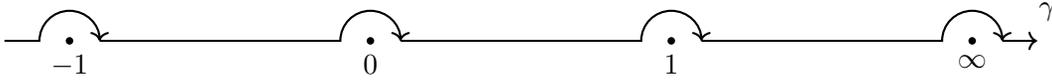
\begin{center}
		\begin{tikzpicture}
			[point/.style={circle,fill=black,inner sep=0pt,minimum size=1mm}]
			\node (start) at (-7,0)  {};
			\node[point] (xm1) at (-6,0)  [label=below:$-1$]{};
			\node[point] (x0) at (-2,0)  [label=below:$0$]{};
			\node[point] (x1) at (2,0) [label=below:$1$] {};
			\node[point] (xinf) at (6,0) [label=below:$\infty$] {};
			\node (end) at (7,0) [label=above:$\gamma$] {};
			\draw [->,thick] (start) -- ($(xm1) - (4mm,0)$) arc [start angle = 180, end angle = 0, radius = 4mm];
			\draw [->,thick] ($(xm1) + (4mm,0)$) -- ($(x0) - (4mm,0)$) arc [start angle = 180, end angle = 0, radius = 4mm];
			\draw [->,thick] ($(x0) + (4mm,0)$) -- ($(x1) - (4mm,0)$) arc [start angle = 180, end angle = 0, radius = 4mm];
			\draw [->,thick] ($(x1) + (4mm,0)$) -- ($(xinf) - (4mm,0)$) arc [start angle = 180, end angle = 0, radius = 4mm];
			\draw [->,thick] ($(xinf) + (4mm,0)$) -- (end);
		\end{tikzpicture}
	\end{center} 
	However, one must take into account the tangential base points, and how they are transformed under the M\"obius transformation which cyclically maps the segments \( (0,1) \mapsto (-1,0) \mapsto (\infty,-1) \mapsto (1,\infty) \).  More precisely, the M\"obius transformation
	\[		
		f(z) = \frac{z-1}{z+1}
	\]
	transforms the segments as indicated, and therefore the integral \( \int_{(-1,0)} \) is related to the integral \( \int_{(0,1)} \) via a suitable pullback.  However, note that the straight line path
	\begin{align*}
		\mathrm{dch} \colon [0,1] &\to [0,1] \\
			t &\mapsto t
	\end{align*}
	with tangential base points \( \overrightarrow{0}_{1} \) and \( \overrightarrow{1}_{-1} \) 
	is transformed into the path 
	\begin{align*}
		(f^\ast \mathrm{dch})(t) \colon [0,1] &\to [-1,0] \\
			t &\mapsto \frac{t-1}{t+1}
	\end{align*}
	with tangential base points \( \overrightarrow{-1}_{2} \) and \( \overrightarrow{0}_{-\frac{1}{2}} \).  In particular, the semicircular integrals evaluate in the following way
	\[
		\int_{\overrightarrow{0}_{-\frac{1}{2}}}^{\overrightarrow{0}_1} \frac{\mathrm{d}t}{t} = \log(2) - \ii \pi \,.
	\]
	So the octagon relation actually takes the form
	\begin{align*}
		e^{\big(-\frac{\mathbb{L}^\mot}{2} + \log^\mot(2)\big) e_{-1}} &
		\Phi^\mot(e_0, e_{-1}, e_1)
		e^{\big(-\frac{\mathbb{L}^\mot}{2} + \log^\mot(2)\big) e_{0}}
		\Phi^\mot(e_0, e_{1}, e_{-1}) \\
		&
		\cdot  e^{\big(-\frac{\mathbb{L}^\mot}{2} + \log^\mot(2)\big) e_{1}}
		\Phi^\mot(e_\infty, e_{1}, e_{-1})
		e^{\big(-\frac{\mathbb{L}^\mot}{2} + \log^\mot(2)\big) e_{\infty}}
		\Phi^\mot(e_\infty, e_{-1}, e_1) = 1 \,,
	\end{align*}
	where \( e_\infty \) is defined such that \( e_0 + e_1 + e_\infty + e_{-1} = 0 \).
	
	This change should  not render Glanois' hybrid identity invalid, as the derivation of the hybrid identity mainly requires the octagon relation modulo products, and these additional terms largely cancel out. 
\end{remark}

We now turn to the derivation of the exact identities which verify our earlier claim that 
\begin{align*}
\zeta_{2k-1}(1,\overline{2\ell}) &\equiv \zeta(\overline{2\ell},\overline{2k}) \pmod{\text{products}} \\
\zeta_{2k}(1,2\ell+1) &\equiv \zeta(2\ell+1,2k+1) \pmod{\text{products}} \,.
\end{align*}
We treat the first, as the second is exactly analogous; we will nevertheless give the full identity in each case.  Firstly, apply shuffle regularisation to
\[
	\zeta_{z-1}(\overline{\alpha}, \beta) + \zeta_{z-1}(\beta, \overline{\alpha}) \,,
\]
to obtain
\[
	= (-1)^{z-1} \sum_{i+j = z-1} \binom{i+\alpha-1}{i} \binom{j+\beta-1}{j} \big\{ \zeta(\overline{i+\alpha}, j+\beta) + \zeta(j+\beta, \overline{i+\alpha}) \big\} \,.
\]
(Note, we have combined the two original sums by switching \( i \leftrightarrow j \) in the second sum.)  By the stuffle product (switching to stuffle regularisation is okay, as there is at most a single trailing 1), we have
\[
= (-1)^{z-1} \sum_{i+j = z-1} \binom{i+\alpha-1}{i} \binom{j+\beta-1}{j} \big\{ \zeta(j+\beta)\zeta(\overline{i+\alpha})  - \zeta(\overline{\alpha+\beta+z-1}) \big\} \,.
\]
On the other hand, apply the shuffle antipode \cite[Equation 29]{GoncharovGalois01} \[
	\begin{aligned}[t]
	& (-1)^N I(a; x_N, \ldots, x_1; b) + I(a; x_1,\ldots,x_N; b) \\
	& + \sum_{i=1}^{N-1} (-1)^{N-i} I(a; x_1, \ldots, x_i; b) I(a; x_{N}, \ldots, x_{i+1}; b) = 0
	\end{aligned}
\]
which effectively reverses the differential forms in an iterated integral \( I(a, x_1,\ldots,x_N; b) \), modulo explicit products terms, to
\[
	\zeta_{z-1}(\overline{\alpha}, \beta) = I(0; \{0\}^{z-1}, -1, \{0\}^{\alpha-1}, 1, \{0\}^{\beta-1}; 1)
\]
and we find
\begin{align*}
	& \zeta_{z-1}(\overline{\alpha}, \beta) + (-1)^{z-1+\alpha+\beta} \zeta_{\beta-1}(\overline{\alpha}, \overline{z}) \\
	& = \begin{aligned}[t] 
	& I(0; \{0\}^{z-1}, -1, \{0\}^{\alpha-1}, 1, \{0\}^{\beta-1}; 1) \\
	& \quad + (-1)^{z-1+\alpha+\beta} I(0; \{0\}^{\beta-1}, 1, \{0\}^{\alpha-1}, -1, \{0\}^{z-1} ; 1) \end{aligned}  \\
	&= -\sum_{i=0}^{\alpha-1} (-1)^{z-1 + \alpha + i} \binom{z+i-1}{i} \binom{\alpha+\beta-2-i}{\beta-1} \zeta(\overline{z+i}) \zeta(\alpha+\beta-1-i) \,.
\end{align*}
Finally, take the difference of these two identities, and set \( \alpha = 2\ell, \beta = 1, z = 2k \), we then obtain the dihedral symmetry we claimed
\begin{align*}
	& \zeta_{2k-1}(1, \overline{2\ell}) - \zeta(\overline{2\ell}, \overline{2k}) = \begin{aligned}[t]
	 -\sum_{i=0}^{2k-1} &  \binom{i+2\ell-1}{i} \big\{ \zeta(\overline{i+2\ell}) \zeta(2k-i) - \zeta(\overline{2k+2\ell}) \big\} \\
	& - \sum_{i=0}^{2\ell-1} (-1)^i \binom{2k+i-1}{i} \zeta(\overline{2k+i}) \zeta(2\ell-i) \end{aligned}
\end{align*}
A slightly more concise version of this is obtained by extending the sums to negative indices -- where the binomial coefficient vanish -- in order to combine them into one, and explicitly summing the coefficient of \( \zeta(\overline{2k+2\ell}) \).  This puts the identity in a form closer to that which one could directly check/derive with the motivic derivations, namely
	\begin{equation}
	\label{eqn:dih:1and2lbar}
\begin{aligned}
 \zeta_{2k-1}(1, \overline{2\ell}) - {} & \zeta(\overline{2\ell}, \overline{2k}) = \binom{2k+2\ell-1}{2k-1} \zeta(\overline{2k+2\ell}) \\ & - \sum_{\substack{r = 1}}^{2k+2\ell-2} \bigg( (-1)^r \binom{r-1}{2k-1} + \binom{r-1}{2\ell-1}  \! \bigg) \zeta(\overline{r}) \zeta(2k+2\ell-r) \,. \end{aligned}
\end{equation}

In an analogous way, we find the explicit form of the dihedral symmetry in the other case to be
\begin{equation}
\label{eqn:dih:1and2lp1}
\begin{aligned}
 \zeta_{2k}(1, 2\ell+1) - {} & \zeta(2\ell+1,2k+1) = {} 
- \zeta(2) \delta_{k=\ell=0} - \binom{2k+2\ell+1}{2\ell+1} \zeta(2k+2\ell+2) \\
& + \sum_{\substack{r = 1}}^{2k+2\ell} \bigg( (-1)^r \binom{r-1}{2k} + \binom{r-1}{2\ell} \! \bigg) \zeta(r) \zeta(2k+2\ell+2-r) \,. \end{aligned}
\end{equation}
(Here the term \( \delta_{k=\ell=0} \) accounts for the difference in shuffle- and stuffle-regularisation in the case \( \zeta(1,1) \).)

Both of these identities are easily verified to be motivic, either by direct calculation via \( \D_{2r+1} \), or by noting that the ingredients -- namely the shuffle and stuffle products, and the regularisation \( \zeta^{\shuffle,0} \) -- are themselves motivic in nature.

\subsection{Generalised doubling identity} \label{sec:gendouble}
The final ingredient we require for our evaluation is one of the so-called \emph{generalised doubling identities}, as described in \cite[\S4]{MZVDM}, and \cite[Section 14.2.5]{ZhaoBook} (be aware these references use the opposite MZV convention).

In depth 2, the relevant relation is already given explicitly by Zhao, and states (with either shuffle, or stuffle regularisation) that\clearpage
\begin{align*}
	& \zeta(s,t) + \zeta(\overline{s},\overline{t}) = \\
	& \sum_{i=1}^s \binom{s+t-i-1}{t-1} 2^{1+i-s-t} \zeta(i, s+t-i) 
	+ \sum_{i=1}^t \binom{s+t-i-1}{s-1} 2^{1+i-s-t} \zeta(s+t-i,i) \\
	& - \sum_{i=1}^t \binom{s+t-i-1}{s-1} \big\{ \zeta(s+t-i,i) + \zeta(\overline{s+t-i},i) \big\} - \binom{s+t-1}{s} 2^{1-s-t} \zeta(s+t)
\end{align*}
We then flip \( \zeta(\overline{a},b) \) to \( \zeta(b,\overline{a}) \) using the stuffle product, rewrite the double zeta sums that lack powers of 2 using the shuffle regularisation as before, and simplify the resulting coefficient of \( \zeta(s+t) \).  (The power of 2 does indeed just disappear!)  This gives the equivalent identity
\begin{align*}
	& \zeta(\overline{s},\overline{t}) + (-1)^t \zeta_{t-1}(1,\overline{s}) = \\
	& \sum_{i=1}^s \binom{s+t-i-1}{t-1} 2^{1+i-s-t} \zeta(i, s+t-i) + \sum_{i=1}^t \binom{s+t-i-1}{s-1} 2^{1+i-s-t} \zeta(s+t-i, i) \\
	& - \zeta(s,t) + (-1)^t \zeta_{t-1}(s,1) - \sum_{i=1}^t \binom{s+t-i-1}{s-1} \zeta(\overline{s+t-i})\zeta(i)
	- \binom{s+t-1}{s} \zeta(s+t) \,.
\end{align*}
Finally we note that upon substituting \( t = 2k, s = 2\ell \), we can solve this identity simultaneously with \autoref{eqn:dih:1and2lbar} to obtain expressions for both \( \zeta(\overline{2\ell},\overline{2k}) \) and \( \zeta_{2k-1}(1,\overline{2\ell}) \) individually as combinations of classical depth 2 MZVs and products.  In particular, we have established the following proposition, (after substituting an expression for \( (-1)^t \zeta_{t-1}(s,1) = \zeta_{2k-1}(2\ell,1) \equiv \zeta(2k,2\ell) \pmod*{\text{products}} \) using the shuffle antipode, or via a further dihedral symmetry, and simplifying).

\propgaldescent

Moreover, by a direct calculation, since \( \D_{2r+1} \) is a tensor product of single-zeta values in this case, we see \autoref{prop:zevbarevbar:eval} (and the generalised doubling identity itself) lifts to the motivic level. This is checked in detail in \autoref{sec:mot_doubling}.

\begin{remark}
	It is clear from the generalised 2-1 Theorem that \( \zeta(\overline{2\ell},\overline{2k}) \) descends to a combination of classical MZVs; we can in fact easily give an explicit formula
	\[
		\zeta(\overline{2\ell},\overline{2k}) = \frac{1}{4} \zeta^\star(1,\{2\}^{\ell-1}, 3, \{2\}^{k-1}) - \frac{1}{2} \zeta(2k+2\ell) \,.
	\]
	However, it is certainly not clear from this expression that depth 2 classical MZVs suffice, and so this would not help us in evaluating \( \zeta(\{2\}^a, 4, \{2\}^b) \) in any useful manner.  However, we do obtain an evaluation for \( \zeta^\star(1,\{2\}^{\ell-1}, 3, \{2\}^{k-1}) \) by substituting \autoref{prop:zevbarevbar:eval} into the above.
	
	Moreover, since \( \zeta(1,1,4,6) \) is -- according to the Data Mine \cite{MZVDM} -- a combination of depth 2 alternating MZVs and products
	\[
	\zeta (1,1,4,6) = \begin{aligned}[t] 
	& \frac{64}{9} \zeta (\overline{3},\overline{9})
	+\frac{371}{144}  \zeta (3,9)
	+3 \zeta (2) \zeta (3,7)
	+\frac{3}{2} \zeta (4) \zeta (3,5) 
	-\frac{3131}{144} \zeta (9) \zeta (3)
	\\
	&+\frac{107}{24}  \zeta (5) \zeta (7)
	+10  \zeta (2) \zeta (7) \zeta (3)
	+\frac{7}{2} \zeta (2) \zeta (5)^2
	-\frac{1}{2} \zeta (4) \zeta (5) \zeta (3)
	\\
	&-\frac{9}{4} \zeta (6) \zeta (3)^2
	+\frac{\zeta (3)^4}{12}
	-\frac{117713 }{132672}\zeta (12) \,,
	\end{aligned}
	\]
	but apparently irreducible as a classical MZV, one cannot in general expect the Galois descent to always respect the depth.  However, as pointed out in \autoref{rem:galoisdepth}, one has -- assuming the homological version of the Broadhurst-Kreimer Conjecture \cite{BroadhurstKConj96} (see also \autoref{conj:depthzetadim} above) -- that the depth of an alternating MZV after Galois descent should be at most twice the original; here the Galois descent of \( \zeta(\overline{3},\overline{9}) \) involving classical MZVs up to depth 4 corroborates this.
\end{remark}

By substituting \autoref{prop:zevbarevbar:eval} into \autoref{eqn:dih:1and2lbar}, and this result into \autoref{eqn:zs2242full2}, we establish that \( \zeta^\star(\{2\}^a, 4, \{2\}^b) \) and (via \autoref{eqn:z2242aszs2242}) that \( \zeta(\{2\}^a, 4, \{2\}^b) \) are both expressible in terms of only classical double zeta values.

\begin{thm}[Non-explicit form] 
	Both \( \zeta^\star(\{2\}^a, 4, \{2\}^b) \) and \( \zeta(\{2\}^a, 4, \{2\}^b) \) are expressible in terms of classical double zeta values.
\end{thm}

\subsection{Generating series} In order to extract an explicit useable identity for \( \zeta(\{2\}^a, 4, \{2\}^b) \) we will convert everything to a generating series identity as a route to simplifying the resulting combinations.

Let us introduce the following generating series, whose names originate from Zagier's evaluation of \( \zeta(\{2\}^a, 3, \{2\}^b) \) \cite{zaghoff}, and some related generating series of even zeta values.  The generating series of odd MZVs, and alternating odd MZV' s are given by
\begin{align*}
A(z) &= \sum_{r=1}^\infty \zeta(2r+1) z^{2r}  =  \psi(1) - \frac{1}{2} \psi(1 - z) - \frac{1}{2} \psi(1+z)\,, \\
B(z) &= \sum_{r=1}^\infty (-\zeta(\overline{2r+1})) z^{2r} = \sum_{r=1}^\infty (1 - 2^{-2r}) \zeta(2r+1) z^{2r} = A(z) - A(\tfrac{z}{2}) \,,
\end{align*}
where \( \psi(x) = \frac{\mathrm{d}}{\mathrm{d}x} \log \Gamma(x) \) is the digamma function, the logarithmic derivative of the Gamma function.
(We keep with the choice established by Zagier taking negatives in the series for alternating MZVs.)  The generating series for even versions, using the convention that \( \zeta(0) = \zeta(\overline{0}) = -\frac{1}{2} \), are given by
\begin{alignat*}{4}
	E(z) &= \sum_{r=0}^\infty \zeta(2r) z^{2r-1} = -\frac{\pi}{2} \cot(\pi z) \,, &\quad
	F(z) &= \sum_{r=0}^\infty (-\zeta(\overline{2r})) z^{2r-1} = \frac{\pi}{2} \csc(\pi z) \,.
	\\
	\widetilde{E}(z) &= \sum_{r=1}^\infty \zeta(2r) z^{2r-1} = \frac{1}{2x}-\frac{\pi}{2} \cot(\pi z) \,, &\quad
	\widetilde{F}(z) &= \sum_{r=0}^\infty (-\zeta(\overline{2r})) z^{2r-1} = -\frac{1}{2x} + \frac{\pi}{2} \csc(\pi z) \,.
\end{alignat*}
The \( \widetilde{\bullet} \) versions which are missing the polar term will be convenient later.
Let us introduce the following double zeta generating series
\begin{alignat*}{4}
	D_\ev(x,y) &= \sum_{a,b=1}^\infty \zeta(2a, 2b) x^{2a-1} y^{2b-1} \,, &\quad&
	D_\od(x,y) &= \sum_{a,b=1}^\infty \zeta(2a + 1, 2b + 1) x^{2a} y^{2b} \,, \\
	D(x,y) &= \mathrlap{
		\sum_{\substack{a,b=2 \\ a \equiv b \pmod*{2}}}^\infty \zeta(a,b) x^{a-1} y^{b-1} = D_\od(x,y) + D_\ev(x,y)\,.
	}
\end{alignat*}
As an intermediate step, let us also introduce the following generating series to capture the shuffle-regularised zetas appearing \autoref{eqn:zs2242full2}, and in the dihedral symmetries \autoref{eqn:dih:1and2lbar} and \autoref{eqn:dih:1and2lp1}, as well as the alternating zeta values as part of the Galois descent result.
\begin{alignat*}{4}
	K_\alt(x,y) &= \sum_{a,b=1}^\infty \zeta(\overline{2a},\overline{2b}) x^{2a-1} y^{2b-1} \,, & \quad &
	K_\ev(x,y) &= \sum_{a,b=1}^\infty \zeta_{2b-1}(1,\overline{2a}) x^{2a-1} y^{2b-1} \,, \\
	K_\od(x,y) &= \sum_{a,b=0}^\infty \zeta_{2b}(1,2a + 1) x^{2a} y^{2b}\,.
\end{alignat*}
Note that we sum from \( a,b = 0 \) in \( K_\od \), but will restrict this to start from \( a,b = 1 \) in \( D_\od \), on account of the well known reductions of \( \zeta(1, 2b+1) \), and \( \zeta(2a+1, 1) \). \medskip

\paragraph{\bf Generating series for \autoref{eqn:zs2242full2}.} To obtain the generating series \( G^\star(x,y) \), we sum the left hand side of \autoref{eqn:zs2242full2} weighted by \( x^{2a} y^{2b} \) over all \( a, b \geq 0 \).

We find then that the generating series of the first term on the right hand side is
\begin{align*}
	\sum_{a,b=0}^\infty \zeta(\overline{2a+2b+4}) x^{2a} y^{2b} &= \sum_{r=0}^\infty \sum_{s=0}^r \zeta(\overline{2r+4}) x^{2r-2s} y^{2s} \\
	&= \sum_{r=0}^\infty \zeta(\overline{2r+4}) \sum_{s=0}^r x^{2r-2s} y^{2s} \\
	&= \sum_{r=0}^\infty \zeta(\overline{2r+4}) \cdot \frac{x^{2r+2} - y^{2r+2}}{x^2 - y^2} \\
	&= - \frac{y \widetilde{F}(x) - x \widetilde{F}(y)}{x y(x^2 - y^2)} \,.
\end{align*}

Likewise, the second leads to
\begin{align*}
	\sum_{a,b=0}^\infty (2a+1) \zeta(2a+2) \zeta(\overline{2b+2}) x^{2a} y^{2b} 
	&= \sum_{a=0}^\infty (2a+1) \zeta(2a+2) x^{2a} \cdot \sum_{b=0}^\infty \zeta(\overline{2b+2}) y^{2b} \\
	&= \frac{1}{y} \frac{\mathrm{d} \widetilde{E}(x) }{\mathrm{d}x} \cdot  \widetilde{F}(y)\,.
\end{align*}
The third and fourth terms are readily summed to give
\begin{align*}
	& \sum_{a,b=0}^\infty \zeta(2a+1) \cdot \sum_{k=1}^{b+1} \zeta(2k+1) \zeta(\overline{2b+2-2k}) x^{2a} y^{2b} 
	= -\frac{1}{y} A(x) A(y) F(y)  \\
	& \sum_{a,b=0}^\infty \sum_{s=0}^a \zeta(\overline{2s}) \zeta_{2a+1-2s}(1, \overline{2b+2}) x^{2a} y^{2b} = - \frac{1}{y} F(x) K_\ev(y,x)\,.
\end{align*}
The last term requires a little care, as the summand \( \zeta_{0}(1,2a+1) \) does not appear, and so must be discounted via \( K_\od(x,0) \), the constant-in-\( y \) term of \( K_\od(x,y) \).  That is 
\begin{align*}
	& \sum_{a,b=0}^\infty \sum_{s=0}^b \zeta(\overline{2s}) \zeta_{2b+2-2s}(1, 2a+1) x^{2a} y^{2b} = -\frac{1}{y} F(y) ( K_\od(x,y) - K_\od(x,0) ) \,.
\end{align*}

This gives us
\begin{align*}
	G^\star(x,y) = 
	\begin{aligned}[t]
	& \frac{8}{y} A(x) A(y) F(y) + \frac{8}{y} F(x) K_\ev(y,x) + \frac{8}{y} F(y) (K_\od(x,y) - K_\od(x,0)) \\
& - 2 \cdot \frac{y \widetilde{F}(x) - x \widetilde{F}(y)}{x y(x^2 - y^2)}	+ \frac{4}{y} \frac{\mathrm{d} \widetilde{E}(x) }{\mathrm{d}x} \cdot  \widetilde{F}(y)\,.
	\end{aligned}
\end{align*}

\paragraph{\bf Generating series for dihedral identities:}

The generating series of the form given by summing \(
	\sum_{k,\ell=1}^\infty (\bullet) y^{2k-1} x^{2\ell-1}
\)
(note the order of the variables), over the left hand side of \autoref{eqn:dih:1and2lbar} is simply just
\(
	K_\ev(x,y) - K_\alt(x,y) 
\).
Then
\begin{align*}
	\sum_{k,\ell=1}^\infty \binom{2k+2\ell-1}{2k-1} \zeta(\overline{2k+2\ell}) y^{2k-1} x^{2\ell-1} &=  \sum_{r=2}^\infty \zeta(\overline{2r}) \sum_{k=1}^{r-1} \binom{2r-1}{2k-1} y^{2k-1} x^{2r-2k-1} \\
	&= \sum_{r=2}^\infty \zeta(\overline{2r}) \Big\{ {-}y^{2r} + \frac{y}{2} (y-x)^{2r-1} + \frac{y}{2} (y+x)^{2r-1} \Big\} \\
	&= -\frac{1}{2x} \Big( \widetilde{F}(y-x) + \widetilde{F}(y+x) - 2 \widetilde{F}(y)  \Big)\,.
\end{align*}
Similarly
\begin{align*}
	& \sum_{k,\ell=1}^\infty \sum_{r=1}^{2k+2\ell-2} (-1)^r \binom{r-1}{2k-1} \zeta(\overline{r})\zeta(2k+2\ell-r) \cdot y^{2k-1} x^{2\ell-1} \\
	& {} = \sum_{k,\ell=1}^\infty \begin{aligned}[t] 
	\bigg(
	{-}\sum_{r=0}^{k+\ell-2} & \binom{2r}{2k-1}  \zeta(\overline{2r+1}) \zeta(2k+2l-2r-1) \\
	& + \sum_{r=1}^{k+\ell-1} \binom{2r+1}{2k-1} \zeta(\overline{2r}) \zeta(2k+2\ell-2r) \bigg) y^{2k-1} x^{2\ell-1} \end{aligned} \\
	& = \sum_{r=0}^\infty \sum_{s=1}^\infty  \sum_{k=1}^r \bigg( {-}\binom{2r}{2k-1} \zeta(\overline{2r+1}) \zeta(2s-1) + \binom{2r-1}{2k-1} \zeta(\overline{2r}) \zeta(2s)  \bigg)y^{2k-1} x^{2s+2r-2k-1}\,.
\end{align*}
The sum over \( k \) can be evaluated explicitly, and (taking care with the \( r = 0 \) terms) one obtains 
\begin{align*}
	& = \begin{aligned}[t]
	\sum_{r,s=1}^\infty & \frac{1}{2} \big( (y-x)^{2r} - (y+x)^{2r} \big) \zeta(\overline{2r+1}) \zeta(2s-1) x^{2s-2} \\[-1ex]
	& + \sum_{r,s=1}^\infty \frac{1}{2}  \big( (y-x)^{2r-1} + (y+x)^{2r-1} \big) \zeta(\overline{2r}) \zeta(2s) x^{2s-1}
	\end{aligned} \\
	&= -\frac{1}{2} \big( B(y-x) - B(y+x) \big) A(x) - \frac{1}{2} \big( \widetilde{F}(y-x) + \widetilde{F}(y+x) \big) \widetilde{E}(x) \,.
\end{align*}
Likewise, one finds
\begin{align*}
& \sum_{k,\ell=1}^\infty \sum_{r=1}^{2k+2\ell-2}  \binom{r-1}{2\ell-1} \zeta(\overline{r})\zeta(2k+2\ell-r) \cdot y^{2k-1} x^{2\ell-1} \\
& = \frac{1}{2} \big( B(y-x) - B(y+x) \big) A(y)  + \frac{1}{2} \big( \widetilde{F}(y-x) - \widetilde{F}(y+x)  \big) \widetilde{E}(y) \,,
\end{align*}
which essentially amounts to switching \( x \leftrightarrow y \), and switching the sign between the two terms.

Overall, one obtains
\begin{align*}
	& K_\ev(x,y) - K_\alt(x,y) = \\
	& \frac{1}{2} (A(x) - A(y))(B(x-y) - B(x+y)) + \frac{1}{2} \widetilde{E}(y) (\widetilde{F}(x-y) + \widetilde{F}(x+y)) \\
	& - \frac{1}{2} \widetilde{E}(x) (\widetilde{F}(x-y) - \widetilde{F}(x+y)) - \frac{\widetilde{F}(x+y) - \widetilde{F}(x-y) - 2 \widetilde{F}(y)}{2x}\,.
\end{align*}

In exactly the same way, one finds for the second dihedral identity \autoref{eqn:dih:1and2lp1} -- taking care with the terms \( \zeta(1,2b+1) \) and \( \zeta(2a+1,1) \) missing from \( D_\od(x,y) \) -- that
\begin{align*}
	& K_\od(x,y) - D_\od(x,y) = \\
	& 2\zeta(2) + \frac{1}{2} (A(x) - A(y))(A(x) - A(x-y) + A(y) - A(x+y)) \\
	& + \frac{1}{2} ( \widetilde{E}(x) - \widetilde{E}(y) ) (\widetilde{E}(x) + \widetilde{E}(y) - \widetilde{E}(x-y)) + \frac{1}{2} \big( \widetilde{E}(x) + \widetilde{E}(y) \big) \widetilde{E}(x+y) \\
	& - \frac{\widetilde{E}(x-y) + \widetilde{E}(x+y)}{2x} - \frac{\widetilde{E}(y)}{y}\,.
\end{align*}

One may observe from this, that
\[
	K(x,0) = \frac{1}{4} \zeta(2) - \frac{3}{8x^2} - \frac{1}{2} A(x)^2 + \Big( \frac{\pi^2}{4} + E(x)^2 \Big) \Big( \frac{1}{2} + \frac{\sin(2\pi x)}{2 \pi x} \Big) \,.
\]

\paragraph{\bf Generating series for \autoref{eqn:zaltevev}:} Now we compute the generating series for the identity from \autoref{prop:zevbarevbar:eval}.  Again, taking \( \sum_{k,\ell=1}^\infty (\bullet)y^{2k-1} x^{2\ell-1}  \), note the variable order, we find the left hand side to be just
\(
	K_\alt(x,y) 
\).
The binomial times double zeta terms can be summed as follows.
\begin{align*}
	& \sum_{k,\ell=1}^\infty \sum_{i=2}^{2k+2\ell-2} 2^{-i} \binom{i-1}{2k-1} \zeta(2k+2\ell-i,i) \cdot  y^{2k-1} x^{2\ell-1} \\
	&= \sum_{r,s=1}^\infty \begin{aligned}[t] 
	\Big\{ 
	\zeta(2r,2s) \cdot & 2^{-2s} \sum_{k=1}^s \binom{2s-1}{2k-1} x^{2r + 2s - 2k-1} y^{2k-1} \\[-1ex]
	& + \zeta(2r + 1, 2s + 1) \cdot 2^{-2s-1} \sum_{k=1}^s \binom{2s}{2k-1} x^{2r + 2s + 1 - 2k} y^{2k-1}
	\Big\} \end{aligned} \\
	&= \frac{1}{4} \sum_{r,s=1}^\infty \Big\{ \begin{aligned}[t]
		\zeta(2r,2s) & x^{2r-1} \Big\{ \Big( \frac{x+y}{2} \Big)^{2s-1} -  \Big( \frac{x-y}{2} \Big)^{2s-1} \Big\} \\
		& + \zeta(2r+1,2s+1) x^{2r} \Big\{ \Big( \frac{x+y}{2} \Big)^{2s} -  \Big( \frac{x-y}{2} \Big)^{2s} \Big\}
		\Big\} \end{aligned} \\
	&= \frac{1}{4} \Big( D\Big(x, \frac{x+y}{2}\Big) - D\Big(x, \frac{x-y}{2}\Big) \Big)\,.
\end{align*}
Similarly with the other terms.  Overall, we find the generating series identity
\begin{align*}
	K_\alt(x,y) = \begin{aligned}[t]
		& \frac{1}{4} \Big( D\Big(x, \frac{x+y}{2}\Big) - D\Big(x, \frac{x-y}{2}\Big) + D\Big(\frac{x+y}{2},y\Big) - D\Big(-\frac{x-y}{2},y\Big) \Big) - D_\ev(x,y) \\
		& + \frac{1}{4} \Big( \widetilde{E}\Big(\frac{x+y}{2}\Big) - \widetilde{E}\Big(\frac{x-y}{2}\Big) \Big) \widetilde{E}(x) - \frac{1}{4} A(x) \Big( A\Big( \frac{x+y}{2} \Big) - A\Big( \frac{x-y}{2} \Big) \Big) \\
		& + \frac{3x - y}{4x(x-y)} \widetilde{E}\Big(\frac{x-y}{2}\Big) - \frac{3x + y}{4x(x+y)} \widetilde{E}\Big( \frac{x+y}{2} \Big) + \frac{1}{2x} \widetilde{E}\Big(\frac{y}{2} \Big)
 	\end{aligned}
\end{align*}

\subsection{\texorpdfstring{Explicit evaluations for \( \zeta^\star(\{2\}^a, 4, \{2\}^b) \)}{Explicit evaluations for zeta\textasciicircum{}*(\{2\}\textasciicircum{}a, 4, \{2\}\textasciicircum{}b)}}
\label{sec:eval:zs2242}

After substituting the previous generating series into the expression for \( G^\star(x,y) \) given via \autoref{eqn:zs2242full2}, performing a non-trivial amount of trigonometric manipulation, we find
{ 
\begin{align*}
	& G^\star(x,y) = \\[1ex]
	& \frac{2 F(x)}{y} \Big( {-}D\Big(\frac{x-y}{2}, x\Big) - D\Big(y, -\frac{x-y}{2}\Big) + D\Big(y, \frac{x+y}{2}\Big) + D\Big(\frac{x+y}{2}, x\Big) \Big) \\
	& - \frac{8 F(x)}{y} D_\ev(y,x) + \frac{8 F(y)}{y} D_\od(x,y) \\
	& + \frac{2 F(x)}{y} \Big\{ - A(y) \cdot \Big( A\Big( \frac{x+y}{2} \Big) - A\Big( \frac{x-y}{2} \Big) \Big) +
	 2 \big( A(x) - A(y) \big) \cdot \big( B(x+y) - B(x-y) \big) \Big\}  \\
	 & + \frac{4 F(y)}{y} \Big\{ A(x) \cdot \big(A(x) - 2 A(y) \big) + \big(A(x)  - A(y) \big) \cdot \big( A(x) + A(y) - A(x-y) - A(x+y) \big) \Big\} \\
	 & - 3 \zeta(2) \cdot \frac{1 - y \widetilde{E}(y)}{y^2} \cdot \sec\Big(\frac{\pi(x-y)}{2} \Big) \sec\Big( \frac{\pi(x+y)}{2}  \Big) - \frac{2 \big( x F(x) - y F(y) \big)}{y^2(x^2-y^2)} \\
	 & +  \frac{4 F(x)}{y} \bigg\{ \widetilde{E}(x) \Big( \frac{1}{y} + \widetilde{E}(x+y) - \widetilde{F}(x-y) \Big) - \frac{\widetilde{E}(\tfrac{1}{2} (x+y))}{x+y} + \frac{\widetilde{E}(\tfrac{1}{2} (x-y))}{x-y} \bigg\} \\
	 & + \frac{4 F(y)}{y} \bigg\{ \begin{aligned}[t] 2\zeta(2) + \widetilde{E}(y)^2 - \frac{2 \widetilde{E}(y)}{y} & + \frac{x \widetilde{E}(x) - y \widetilde{E}(y)}{x^2 - y^2} \\ & \hspace{2em} - \frac{(x+y) \widetilde{E}(x-y) - (x-y) \widetilde{E}(x+y)}{2 x y} \bigg\}\,. \end{aligned}
\end{align*}
}
If one so desires, the following explicit formula for the individual coefficient \( \zeta^\star(\{2\}^a, 4, \{2\}^b) \) can be extracted from the above.

\begin{thm}[Evaluation of \( \zeta^\star(\{2\}^a, 4, \{2\}^b) \) via double zeta values\protect\footnotemark]\label{thm:eval:zs2242}
\footnotetext{A computer readable version as plain text files in \texttt{Mathematica} syntax and in \texttt{pari/gp} syntax is attached to the arXiv submission.}
Write as shorthand \( \zeta(\overline{n}) = -(1-2^{1-n})\zeta(n) \), and employ the conventions that \( \zeta(0) = \zeta(\overline{0}) = -\frac{1}{2} \) and \( \zeta(1) = 0 \) (however no further regularisation is necessary).  Denote by \( E_n \) the \( n \)-th Euler number, given as the coefficients of 
	\(
	\operatorname{sech}(t) = \sum_{n=0}^\infty \frac{E_n}{n!} t^n \,.
	\)
	Then for any \( a, b \in \mathbb{Z}_{\geq0} \), the following evaluation holds, where we assume all summation variables start from \( 0 \),\clearpage
{\small  
\begin{align*}
	& \zeta^\star(\{2\}^a, 4, \{2\}^b) = \\[1em]
	& 
	\sum_{\crampedsubstack{p+q = a}}8  \zeta(\overline{2q}) \zeta(2b+2, 2p+2) 
	 - \sum_{\crampedsubstack{r + s = b}} 8 \delta_{a > 0} \zeta(\overline{2s}) \zeta(2a+1, 2r+3) 
	 \\	 &
	  - \sum_{\crampedsubstack{2u + i + j = 2a + 2b}}  2\bigg( \frac{1}{2^i} \binom{i+1}{2b+1} + \frac{1}{2^j} \binom{j+1}{2a+1-2u} \!  \bigg) \zeta(\overline{2u})  \zeta(i+2,j+2) 
	  \\[1.7em]\linebreak
	 & 
	 + \sum_{\crampedsubstack{p + q = a \\ r + s = b-1}} \bigg( \frac{1}{2^{2q+2s}} \zeta(2q + 2s + 3) 
	  - 8 \zeta(\overline{2s+2q+3}) \bigg) \binom{2 + 2q + 2s}{1 + 2s} \zeta(2r + 3) \zeta(\overline{2p}) 
	  \\[-1ex]	 & 
	 +  \sum_{\crampedsubstack{u + v + w = a-1}} 8\binom{2w + 2b + 2}{2b+1} \zeta(2u+3) \zeta(\overline{2v}) \zeta(\overline{2b+2w+3})
	 \\ &
	 - \sum_{\crampedsubstack{p + q = a-2}} 8 \zeta(2p+3) \zeta(2q+3) \zeta(\overline{2b+2}) 
	  +  \sum_{\crampedsubstack{r + s = b}} 8\zeta(2a+1) \zeta(2r+3) \zeta(\overline{2s})
	 \\	 &
	 + \sum_{\crampedsubstack{u + v + w = b-1}} 4 \delta_{a=0} \zeta(2u+3) \zeta(2v+3) \zeta(\overline{2w}) 
	 \\	 & 
	 +  \sum_{\crampedsubstack{p + q = a-1 \\ r + s = b+1}} 8\binom{2q + 2s}{2s} \zeta(\overline{2r}) \zeta(2p+3) \zeta(2q + 2s + 1) 
	 \\[-0.5ex]	 &
	 -  \sum_{\crampedsubstack{u + v + w = b}} \binom{2a + 2v}{2v} 8\zeta(\overline{2u}) \zeta(2w+3) \zeta(2a + 2v + 1) 
	 \\[1.7em]
	 & 
	 -\sum_{\crampedsubstack{i+j = 2a \\ r + s = 2b+2}}  3 \zeta(2) \frac{(-1)^{r} E_{i+r} E_{j+s}}{i!j!r!s!} \Big(\frac{\ii \pi}{2}\Big)^{2a+2b+2}
	 \\[-0.5ex]	 &
	 + \sum_{\crampedsubstack{i + j = 2a \\ r + s + 2t = 2b}} 3 \zeta(2) \frac{(-1)^r E_{i+r} E_{j+s}}{i!j!r!s!} \Big(\frac{\ii \pi}{2}\Big)^{2a+2b-2t} \zeta(2t+2) 
	 \\[-0.5ex]	 & 
	 + 2 \zeta(\overline{2a+2b+4}) +\sum_{\substack{p+q = a+1}} \frac{4}{2^{2p+2b}} \binom{2p + 2b}{2b+1} \zeta(2p+2b+2) \zeta(\overline{2q}) 
	 \\	 & 
	 +\sum_{\crampedsubstack{u + v + w = a}}  8 \binom{2u + 2b + 1}{2b+1} \zeta(\overline{2w}) \zeta(2v+2) \zeta(\overline{2u+2b+2}) 
	 \\	 & 
	 +  \sum_{\crampedsubstack{r + s = b+1}} 4\bigg( \! \binom{2a+2r+1}{2a+1} - \binom{2a+2r+1}{2r+1} \! \bigg) \zeta(2a+2r+2) \zeta(\overline{2s}) 
	 \\	 & 
	 -  \sum_{\crampedsubstack{r + s = b+1}} 4\zeta(2a+2r+2) \zeta(\overline{2s})
	 - \sum_{\crampedsubstack{u + v + w = b}}4 \delta_{a=0}  \zeta(2u+2) \zeta(2v+2) \zeta(\overline{2w}) 
	 \\	 & 
	 + \sum_{\crampedsubstack{r + s = b+1}} 8 \delta_{a=0} \zeta(2r+2) \zeta(\overline{2s})
	 - 8 \delta_{a=0} \zeta(2) \zeta(\overline{2b+2}) \,.
\end{align*} 
}
\end{thm}
\clearpage

\subsection{\texorpdfstring{Explicit evaluations for \( \zeta(\{2\}^a, 4, \{2\}^b) \)}{Explicit evaluations for zeta(\{2\}\textasciicircum{}a, 4, \{2\}\textasciicircum{}b)}}
\label{sec:eval:z2242}
By substituting the expression for \( G^\star(x,y) \) into \autoref{eqn:gs:zs2242asz2242}, and finding \( G(x,y) \) via
\[
	G(x,y) = \sum_{a,b=0}^\infty (-1)^{a+b} \zeta(\{2\}^a, 4, \{2\}^b) x^{2a} y^{2b} = G^\star(y,x) \frac{\sin(\pi x)}{\pi x} \frac{\sin(\pi y)}{\pi y} \,,
\]
we obtain the following explicit expression for the generating series,

{
	\begin{align*}
	& G(x,y) = \\[1ex]
	& \frac{\sin(\pi x)}{\pi x^2 y} \Big( {-}D\Big(-\frac{x-y}{2}, x\Big) - D\Big(x, \frac{x-y}{2}\Big) + D\Big(x, \frac{x+y}{2}\Big) + D\Big(\frac{x+y}{2}, y\Big) \Big) \\
	& - \frac{4\sin(\pi x)}{\pi x^2 y} D_\ev(x,y) + \frac{\sin(\pi y)}{\pi x^2 y} D_\od(y,x) \\
	& - \frac{\sin(\pi x)}{\pi x^2 y} \Big\{ A(x) \cdot \Big( A\Big( \frac{x+y}{2} \Big) - A\Big( \frac{x-y}{2} \Big) \Big) +
	2 \big( A(x) - A(y) \big) \cdot \big( B(x+y) - B(x-y) \big) \Big\}  \\
	& - \frac{\sin(\pi y)}{\pi x^2 y} \Big\{ A(y) \cdot \big(2A(x) - A(y) \big) + \big(A(x)  - A(y) \big) \cdot \big( A(x) + A(y) - A(x-y) - A(x+y) \big) \Big\} \\
	& - 3 \zeta(2) \cdot \frac{1 - x \widetilde{E}(x)}{x^2} \cdot \sec\Big(\frac{\pi(x-y)}{2} \Big) \sec\Big( \frac{\pi(x+y)}{2}  \Big) \frac{\sin(\pi x)}{\pi x}\frac{\sin(\pi y)}{\pi y} \\
	 & \quad - \frac{1}{x^2(x^2-y^2)} \Big( \frac{\sin(\pi x)}{\pi x} - \frac{\sin(\pi y)}{\pi y} \Big)\\
	& +  \frac{2\sin(\pi x)}{\pi x^2 y}  \bigg\{ \widetilde{E}(y) \Big( \frac{1}{x} + \widetilde{E}(x+y) + \widetilde{F}(x-y) \Big) - \frac{\widetilde{E}(\tfrac{1}{2} (x+y))}{x+y} + \frac{\widetilde{E}(\tfrac{1}{2} (x-y))}{x-y} \bigg\} \\
	& + \frac{2\sin(\pi y)}{\pi x^2 y}  \bigg\{ \begin{aligned}[t] 
 2\zeta(2) + \widetilde{E}(x)^2 - \frac{2 \widetilde{E}(x)}{x} & + \frac{x \widetilde{E}(x) - y \widetilde{E}(y)}{x^2 - y^2} \\ 
 & \hspace{2em} + \frac{(x+y) \widetilde{E}(x-y) - (x-y) \widetilde{E}(x+y)}{2 x y} \bigg\} \,. \end{aligned}
	\end{align*}
}

If one desires, the following explicit formula for the individual coefficient \( \zeta(\{2\}^a, 4, \{2\}^b) \) can be extracted from the above.  

\begin{thm}[Evaluation of \( \zeta(\{2\}^a, 4, \{2\}^b) \) via double zeta values\protect\footnotemark]\label{thm:eval:z2242}
\footnotetext{A computer readable version as plain text files in \texttt{Mathematica} syntax and in \texttt{pari/gp} syntax is attached to the arXiv submission.}
Write as shorthand \( \zeta(\overline{n}) = -(1-2^{1-n})\zeta(n) \), and employ the conventions that \( \zeta(0) = \zeta(\overline{0}) = -\frac{1}{2} \) and \( \zeta(1) = 0 \) (however no further regularisation is necessary). Denote by \( E_n \) the \( n \)-th Euler number, given as the coefficients of 
	\(
	\operatorname{sech}(t) = \sum_{n=0}^\infty \frac{E_n}{n!} t^n \,.
	\)
	Then for any \( a, b \in \mathbb{Z}_{\geq0} \), the following evaluation holds, where we assume all summation variables start from \( 0 \),\newpage
{\small \allowdisplaybreaks
\begin{align*}
&  \zeta(\{2\}^a, 4, \{2\}^b) = 
\\[1em]
& \begin{aligned}[t]  (-1)^{a+b}  \Bigg\{ & -\sum_{\crampedsubstack{p + q = a}}  4 \zeta(2p+2, 2b+2) \frac{(\ii\pi)^{2q}}{(2q+1)!}
+ \sum_{\crampedsubstack{r + s = b-1}} 4 \zeta(2r+3, 2a+3) \frac{(\ii\pi)^{2s}}{(2s+1)!} 
\\
& 
+ \sum_{\crampedsubstack{2u + i + j = 2a + 2b}} \bigg( \frac{1}{2^i} \binom{i+1}{2a-2u+1} + \frac{1}{2^j} \binom{j+1}{2b+1} \! \bigg) \zeta(i+2, j+2) \frac{(\ii\pi)^{2u}}{(2u+1)!} 
\\[1.7em]
&
+ \sum_{\crampedsubstack{u + v + w = a-1}} 4 \binom{2w+2b+2}{2b+1} \bigg( \zeta(\overline{2b+2w+3}) - \frac{ \zeta(2b+2w+3) }{2^{2b+2w+3}}\bigg) \zeta(2v+3) \frac{(\ii\pi)^{2u}}{(2u+1)!} 
\\ 
&
- \sum_{\crampedsubstack{p + q = a \\ r + s = b}} 4 \binom{2p + 2s}{2s-1} \zeta(2r+3) \zeta(\overline{2s+2p+1}) \frac{(\ii\pi)^{2q}}{(2q+1)!} 
\\[-1ex]
&
- \sum_{\crampedsubstack{u + v + w = b-1}} 4 \binom{2a + 2v + 2}{2v} \zeta(2w+3) \zeta(2a + 2v + 3) \frac{(\ii\pi)^{2u}}{(2u+1)!}
\\
&
 + \sum_{\crampedsubstack{p + q = a \\ r + s = b}} 4 \binom{2q + 2r}{2r} \zeta(2p+3) \zeta(2q + 2r + 1) \frac{(\ii\pi)^{2s}}{(2s+1)!}
\\[-1ex]
&
 - \sum_{\crampedsubstack{p + q = a-1}} 2 \zeta(2p+3) \zeta(2q+3) \frac{(\ii\pi)^{2b}}{(2b+1)!} 
 - \sum_{\crampedsubstack{r + s = b-1}} 4 \zeta(2a+3) \zeta(2r+3) \frac{(\ii\pi)^{2s}}{(2s+1)!} 
 \\[1.7em]
 &
 - \sum_{\crampedsubstack{i + j + 2k = 2a + 2 \\ p + q + 2r = 2b}} 3 \zeta(2) \frac{(-1)^p E_{i+p} E_{j+q}}{i!j!p!q!} \Big( \frac{\ii\pi}{2} \Big)^{2a+2b+2} \frac{2^{2k+2r}}{(2k+1)! (2r+1)! }
 \\[-0.5ex]
 &
 + \sum_{\crampedsubstack{i + j + 2k + 2\ell = 2a  \\ p + q + 2r = 2b}} 3 \zeta(2) \frac{(-1)^p E_{i+p} E_{j+q}}{i!j!p!q!} \Big( \frac{\ii\pi}{2} \Big)^{2a+2b-2\ell} \frac{2^{2k+2r}}{(2k+1)! (2r+1)! } \zeta(2\ell+2) 
 \\
 &
 + \frac{(\ii\pi)^{2a+2b+4}}{(2a+2b+5)!} 
 + 2 \zeta(2b+2) \frac{(\ii\pi)^{2a+2}}{(2a+3)!}
 - 4 \zeta(2a+4) \frac{(\ii\pi)^{2b}}{(2b+1)!} 
 \\
 &
 - \sum_{\crampedsubstack{p + q = a+1}} \frac{1}{2^{2p+2b-1}} \binom{2p+2b}{2b+1} \zeta(2p+2b+2) \frac{(\ii\pi)^{2q}}{(2q+1)!}
 \\
 &
 - \sum_{\crampedsubstack{p + q = a \\ r + s = b}} 4 \binom{2p+2r+1}{2p+1} \zeta(\overline{2p+2r+2}) \zeta(2s+2) \frac{(\ii\pi)^{2q}}{(2q+1)!} \\[-0.5ex]
 & + \sum_{\crampedsubstack{r + s = b}} 2 \zeta(2s+2r+4) \bigg( \! \binom{2a + 2r + 3}{2a+3} - \binom{2a+2r+3}{2r+1} \! \bigg) \frac{(\ii\pi)^{2s}}{(2s+1)!}  
 \\[-0.5ex]
 &
 + \sum_{\crampedsubstack{r + s = b}} 2 \zeta(2a+2r+4) \frac{(\ii\pi)^{2s}}{(2s+1)!}
 + \sum_{\crampedsubstack{p + q = a}} 2 \zeta(2p+2) \zeta(2q+2) \frac{(\ii\pi)^{2b}}{(2b+1)!} \,\, \Bigg\} \,.
\end{aligned}
\end{align*}
}
\end{thm}

In particular, we obtain the following corollary on the reduction of \( \zeta(\{2\}^a, 4, \{2\}^b) \) modulo products.  In essence it extracts those double zeta terms above, which are not multiplied by a power of \( \pi \).
\begin{cor}\label{cor:eval:z2242prod}
	Modulo decomposables (i.e. products of MZVs), the following evaluation holds
	\begin{align*}
	 &  \zeta(\{2\}^a, 4, \{2\}^b) = \\
	& (-1)^{a+b} \bigg\{ {-} 4 \zeta(2a+2, 2b+2) +  4 \zeta(2b+1, 2a+3)
	\\ & 
	+ \sum_{\crampedsubstack{i + j = 2a + 2b \\ i, j \geq 0}} \bigg(  \frac{1}{2^i} \binom{i+1}{2a+1} + \frac{1}{2^j} \binom{j+1}{2b+1} \! \bigg) \zeta(i+2, j+2) \bigg\} 
	\pmod{\mathrm{products}} \,.
	\end{align*}
\end{cor}

\section{\texorpdfstring{Motivic evaluation of \( \zm(\{2\}^a, 4, \{2\}^b) \)\except{toc}{\\} via motivic double zeta values}{Motivic evaluation of zeta\textasciicircum{}m(\{2\}\textasciicircum{}a, 4, \{2\}\textasciicircum{}b) via motivic double zeta values}}
\label{sec:mot2242}

In order to verify that the evaluation of \( \zeta(\{2\}^a, 4, \{2\}^b) \) in \autoref{sec:eval:z2242} (or at least, the evaluation in \autoref{cor:eval:z2242prod}) is motivic we only need to show that the various ingredients used in \autoref{sec:num2242} are motivic.

More precisely, we need to show that \autoref{prop:zevbarevbar:eval} and \autoref{eqn:zs2242full2} are motivic.  All other identities used in the derivation of \autoref{cor:eval:z2242prod}, \autoref{thm:eval:z2242} and \autoref{thm:eval:zs2242} were obtained from the shuffle or stuffle product, and so are automatically motivic (the shuffle product is motivic by definition, for the stuffle-product see \cite{RacinetThesis,SouderesDblSh}). \medskip

\paragraph{\bf Framework of alternating motivic MZVs:} We note here that the motivic MZV framework of \autoref{sec:motmzv} generalises readily to the case of alternating motivic MZVs.  For the technical details of this, we refer to \cite{GlanoisThesis16,GlanoisBasis16}; the most important points are the comodule of alternating motivic MZVs \( \mathcal{H}^{(2)} \) is obtained by extending \autoref{def:Hn} to allow \( a_i \in \{0, \pm 1 \} \) (although functoriality in a useful form only applies when all \( a_i \in \{ 0, 1 \}\)).  Then for a tuple \( (k_1,\ldots,k_d) \) of positive integers, and \( (\eps_1,\ldots,\eps_d) \in \{\pm1\}^d \) of signs, and \( \ell \geq 0 \), we define the motivic alternating MZV by
\begin{equation}\label{eqn:altmzv:mot}
	\zeta^\mot_\ell\bigg( \begin{matrix} \eps_1,\eps_2,\ldots,\eps_d \\ k_1,k_2,\ldots,k_d \end{matrix} \bigg) \coloneqq (-1)^d I^\mot(0; \{0\}^{\ell}, \eta_1, \{0\}^{k_1-1}, \eta_2, \{0\}^{k_2-1}, \ldots, \eta_d, \{0\}^{k_d-1} ;1) \,,
\end{equation}
where \( \eta_i = \eps_i \eps_{i+1} \cdots \eps_d \), mimicking the integral representation of alternating MZVs in \autoref{eqn:altmzv:int}  One can again streamline the notation by dropping the \( \eps_i \)'s and writing \( \overline{k_i} \) if \( \eps_i = -1 \), and just \( k_i \) if \( \eps_i = 1 \).  Then \( \mathcal{A}^{(2)} \coloneqq \mathcal{H}^{(2)} / ( \zeta^\mot(2) ) \) and \( \mathcal{L}^{(2)} = \mathcal{A}^{(2)}_{>0} / \mathcal{A}^{(2)}_{>0}\mathcal{A}^{(2)}_{>0} \) define the obvious extensions of the Hopf algebra and the Lie coalgebra of irreducibles.  The coaction \( \Delta \colon \mathcal{H}^{(2)} \to \mathcal{A}^{(2)} \otimes \mathcal{H}^{(2)} \) is defined by the same formula as in \autoref{eqn:coaction}, and the  infinitesimal derivations \( \D_r \colon \mathcal{H}^{(2)} \to \mathcal{L}^{(2)}_r \otimes \mathcal{H}^{(2)} \), with \( \mathcal{L}^{(2)}_r \) the weight \( r \) component of \( \mathcal{L}^{(2)} \), are given by the same formula as in \autoref{eqn:derivation}.

For alternating motivic MZVs, \( \D_1 \) plays a non-trivial role, as the weight 1 alternating MZV \( \zeta^\mot(\overline{1}) = \log^\mot(2) \) is non-zero.  The analogue of Brown's characterisation of \( \ker \D_{<N} \) in the alternating case is given by Glanois as follows.

\begin{thm}[Glanois, Corollary 2.4.5 \cite{GlanoisThesis16}] Let \( N \geq 1 \), and denote by \( \D_{<N} \linebreak =  \bigoplus_{1 \leq 2r+1 < N} \D_{2r+1} \).  Then in weight \( N \), the kernel of \( D_{<N} \) on alternating motivic MZVs is one dimensional:
	\[
		\ker \D_{<N} \cap \mathcal{H}^{(2)}_N = \begin{cases}
			\mathbb{Q} \zeta^\mot(\overline{1}) = \mathbb{Q}  \log^\mot(2) & \text{if $N = 1$} \\ 
			\mathbb{Q}  \zeta^\mot(N) & \text{if \( N > 1 \)\,.}
		\end{cases}
	\]
\end{thm}

Since the identities we wish to lift involve alternating MZV terms in a non-trivial way, we necessarily have to use Glanois' criterion to verify the motivic lift, even if as it happens \( \D_1 = 0 \) in each case.

\subsection{Motivic version of \autoref{prop:zevbarevbar:eval}}
\label{sec:mot_doubling} 
	
	We prove the following proposition which claims that \autoref{prop:zevbarevbar:eval} lifts to a motivic version.
	
	\begin{prop}[Motivic Galois descent of \( \zm(\overline{2\ell},\overline{2k}) \)]\label{prop:zevbarevbar:mot}
		The alternating motivic double zeta value \( \zm(\overline{2\ell},\overline{2k}) \) enjoys a Galois descent to classical depth 2 motivic MZVs as follows
		\begin{equation}
		\label{eqn:zaltevevmot}
		\begin{aligned}[c]
		\zm(\overline{2\ell},\overline{2k}) = 
		\begin{aligned}[t]
		& \sum_{i=2}^{2k + 2\ell - 2} 2^{-i} \bigg\{ \binom{i-1}{2k-1} \zm(2k+2\ell-i,i) + \binom{i-1}{2\ell-1} \zm(i,2k+2\ell-i) \bigg\} \\
		& -\zm(2\ell,2k) + \sum_{r=2}^{2k+2\ell-2} (-2)^{-r} \binom{r-1}{2k-1}  \zm(r) \zm(2k+2\ell-r) \\
		& - 2^{-2k-2\ell} \bigg\{ 2 \binom{2k+2\ell-2}{2k-1} + \binom{2k+2\ell-1}{2k-1} \bigg\} \zm(2k + 2\ell) \,.
		\end{aligned}
		\end{aligned}
		\end{equation}
	\end{prop}

\begin{proof}
We compute \( \D_{2r+1} \) of both sides, and verify they agree for \(  1 \leq r \leq k+\ell-2 \).  The case \( r = 0 \) does not play a role, since \( \D_1 \) is known to be exactly zero by the Galois descent property established in \cite{GlanoisThesis16}; alternatively one can directly compute it and see there is no contribution since the sequences \( (0, 1, -1), (0, -1, 1), (-1,1,0), (1,-1,0) \) which give rise to \( \log^\lmot(2) \) are not present in the integral representation of any term.  In the case \( r = k + \ell-1 \), \( \D_{2k+2\ell-1} \) is quickly checked to vanish as only \( \zm(1) = 0 \) appears in the right hand tensor factor. \medskip

\paragraph{\bf Computation of \( \D_{2r+1} \zm(\overline{2a},\overline{2b}) \) and \( \D_{2r+1} \zm(2a, 2b) \)}  We see that only the following subsequences can contribute to the motivic coaction.  This is because any subsequence must start or end one of the three non-zero entries; one then checks whether the length \( 2r+1 \) subsequences which start/end at these points actually contribute.
\begin{center}
	\begin{tikzpicture}
	\matrix[name=M1, matrix of nodes, inner sep=1pt, column sep=0pt]{
		\node () [] {$\zm(\overline{2a},\overline{2b}) = I^\mot($};
		& \node (s1) [] {${\phantom{\mathllap{\fbox{$,0$}}}} 0$}; 
		& \node () [] {$;$}; 
		& \node (m1) [] {${\phantom{\mathllap{\fbox{$,0$}}}}  {}1$};  			& \node () [] {$,$};
		& \node (zm1) [] {${\phantom{\mathllap{\fbox{$,0$}}}} \{0\}^{2a-1} , $};
		& \node (m2) [] {${\phantom{\mathllap{\fbox{$,0$}}}}  {-}1$};  & \node () [] {$,$}; 
		& \node (zm2) [] {${\phantom{\mathllap{\fbox{$,0$}}}} \{0\}^{2b-1} , $};
		& \node (t1) [] {${\phantom{\mathllap{\fbox{$,0$}}}}  1$};  
		& \node () [] {$) \,.$}; 
		\\
	};
	\draw[thick,black] (zm1.north) to[bracket=5pt] node[above=-1mm] {$$} (t1.north);
	\draw[thick,black] (m1.south) to[ubracket=5pt] node[below=-1mm] {$$} (zm2.south);
	\draw[black,thin,solid] ($(m1.north west)+(0,0)$)  rectangle ($(m1.south east)+(0,0)$);
	\draw[black,thin,solid] ($(m2.north west)+(0,0)$)  rectangle ($(m2.south east)+(0,0)$);
	\draw[black,thin,solid] ($(t1.north west)+(0,0)$)  rectangle ($(t1.south east)+(0,0)$);
	\end{tikzpicture}
	\end{center}
	We find
	\begin{align*}
		&\D_{2r+1} \zm(\overline{2a}, \overline{2b}) \\
		& = -\delta_{a \leq r} \zeta_{2r + 1 - 2a}^\lmot(\overline{2a}) \otimes \zeta(2a + 2b - 2r - 1) + \delta_{b\leq r} \zeta_{2r + 1 - 2b}^\lmot(\overline{2b}) \otimes \zm(2a + 2b - 2r - 1) \\
		&= \bigg( \! \binom{2r}{2a-1} - \binom{2r}{2b-1} \! \bigg) \zeta^\lmot(\overline{2r+1}) \otimes \zeta(2a + 2b - 2r - 1)
	\end{align*}
	(The binomial factors should \emph{a prior} retain the delta factors, but they can be removed as the binomials vanish already for the complementary condition.)
	The corresponding result holds for \( \zm(2a, 2b) \) by removing all bars from the above result
	\begin{align*}
	&\D_{2r+1} \zm({2a}, {2b}) \\
	&= \bigg( \! \binom{2r}{2a-1} - \binom{2r}{2b-1} \! \bigg) \zeta^\lmot({2r+1}) \otimes \zeta(2a + 2b - 2r - 1)\,.
	\end{align*}

	\paragraph{\bf Computation of \( \D_{2r+1} \zm(2a+1, 2b+1) \)}  We see that only the following subsequences can contribute to the motivic coaction.  This is because any subsequence must involve one of the three non-zero entries; one then checks whether the length \( 2r+1 \) subsequences which start/end at these points actually contribute.
	\vspace{-0.5em}
	\begin{center}
		\begin{tikzpicture}
		\matrix[name=M1, matrix of nodes, inner sep=1pt, column sep=0pt]{
			\node () [] {$\zm({2a+1},{2b+1}) = I^\mot($};
			& \node (s1) [] {${\phantom{\mathllap{\fbox{$,0$}}}} 0$}; 
			& \node () [] {$;$}; 
			& \node (m1) [] {${\phantom{\mathllap{\fbox{$,0$}}}}  {}1$};  			& \node () [] {$,$};
			& \node (zm1) [] {${\phantom{\mathllap{\fbox{$,0$}}}} \{0\}^{2a} , $};
			& \node (m2) [] {${\phantom{\mathllap{\fbox{$,0$}}}}  1$};  & \node () [] {$,$}; 
			& \node (zm2) [] {${\phantom{\mathllap{\fbox{$,0$}}}} \{0\}^{2b} , $};
			& \node (t1) [] {${\phantom{\mathllap{\fbox{$,0$}}}}  1$};  
			& \node () [] {$) \,.$}; 
			\\
		};
		\draw[thick,black] (zm1.south) to[ubracket=12pt] node[below=-7mm] {$$} (t1.south);
		\draw[thick,black] (m1.south) to[ubracket=5pt] node[below=-1mm] {$$} (zm2.south);
		\draw[thick,black] (s1.north) to[bracket=5pt] node[above=-1mm] {$$} (m2.north);
		\draw[black,thin,solid] ($(m1.north west)+(0,0)$)  rectangle ($(m1.south east)+(0,0)$);
		\draw[black,thin,solid] ($(m2.north west)+(0,0)$)  rectangle ($(m2.south east)+(0,0)$);
		\draw[black,thin,solid] ($(t1.north west)+(0,0)$)  rectangle ($(t1.south east)+(0,0)$);
		\end{tikzpicture}
	\end{center}
	\vspace{-1em}
	We find
	\begin{align*}
	&\D_{2r+1} \zm({2a+1}, {2b+1}) \\
	& = \begin{aligned}[t]
	& \delta_{a=r} \zeta^\lmot(2r+1) \otimes \zm(2a + 2b + 1 - 2r) \\
	& + \Big( {-} \delta_{a\leq r} \zeta_{2r-2a}^\lmot(2a+1) + \delta_{b \leq r} \zeta_{2r-2b}^\lmot(2b+1) \Big) \otimes \zm(2a+2b+1-2r) \end{aligned} \\
	&= \bigg( \delta_{a=r} - \binom{2r}{2a} + \binom{2r}{2b} \! \bigg) \zeta^\lmot({2r+1}) \otimes \zeta(2a + 2b + 1 - 2r) \,.
	\end{align*}

\medskip
\paragraph{\bf Computation of \( \D_{2r+1} \zm(p, q) \), \( p + q \) even}

We note that the two cases above can be combined to give the following, for \( p + q \) even
\[
	\D_{2r+1} \zm(p,q) = \bigg( \delta_{2r+1 = p} + (-1)^p \binom{2r}{p-1} - (-1)^q \binom{2r}{q-1}\!  \bigg) \zeta^\lmot(2r+1) \otimes \zeta(p+q - 2r - 1) \,.
\]

\medskip
\paragraph{\bf Verification of \autoref{prop:zevbarevbar:mot}}

The claim that \( \D_{2r+1} \) of both sides agree is equivalent to the following putative identity amongst binomial coefficients, when \( 1 \leq r \leq k+\ell-1 \), which arises after projecting \( \zeta^\lmot(2r+1) \otimes \zm(2k+2\ell-2r-1) \mapsto 1 \).

\begin{align*}
	& 0 \overset{?}{=} (1-2^{-2r}) \bigg( \! \binom{2r}{2\ell-1} - \binom{2r}{2k-1} \! \bigg) \\
	& + \sum_{i=2}^{2k+2\ell-2} 2^{-i} \binom{i-1}{2k-1} \bigg( \delta_{2k+2\ell-i = 2r+1} + (-1)^i \binom{2r}{2k+2\ell-i-1} - (-1)^{i} \binom{2r}{i-1} \! \bigg) \\
	& + \sum_{i=2}^{2k+2\ell-2} 2^{-i} \binom{i-1}{2\ell-1} \bigg( \delta_{i = 2r+1} + (-1)^i \binom{2r}{i-1} - (-1)^{i} \binom{2r}{2k+2\ell-i-1}\!  \bigg) \\
	& - \bigg(\!  \binom{2r}{2l-1} - \binom{2r}{2k-1} \! \bigg) \\
	& + (-2)^{-(2r+1)} \binom{2r}{2k-1} + (-2)^{-(2k+2\ell-2r-1)} \binom{2k+2\ell-2r-2}{2k-1}\,.
\end{align*}

After some simplification of the right hand side, and reindexing the sums, we find that the claim is equivalent to the following
\begin{align*}
	0 \overset{?}{=} & -2^{-2r-1} \bigg( \!  \binom{2r}{2\ell-1} - \binom{2r}{2k-1} \!  \bigg) \\
	& + (-2)^{-2k} \sum_{i=0}^{2\ell-2} (-2)^{-i} \binom{i+2k-1}{2k-1} \bigg( \! \binom{2r}{2\ell-i-1} - \binom{2r}{2k+i-1} \! \bigg) \\
	& + (-2)^{-2\ell} \sum_{i=0}^{2k-2} (-2)^{-i} \binom{i+2k-1}{2\ell-1} \bigg( \! \binom{2r}{2\ell+i-1} - \binom{2r}{2k-i-1} \!  \bigg)\,.
\end{align*}

This is verified to be exactly 0 from the \autoref{lem:binomial} of \autoref{sec:proofs}.  With that, we have finished the proof of \autoref{prop:zevbarevbar:mot}.
\end{proof}

\subsection{\texorpdfstring{Motivic version of \protect\autoref{eqn:zs2242full2}}{Motivic version of Equation (\ref{eqn:zs2242full2})}}
\label{sec:zs2242dp3:mot}

We prove the following proposition which claims that \autoref{eqn:zs2242full2} lifts to a motivic version.

\begin{prop}\label{prop:zsasdp3:mot}
	The following identity holds amongst motivic multiple zeta (star) values.
	\begin{equation}\label{eqn:zs2242full2mot}
	\begin{aligned}[c]
	\phantom{(\{2\}^a, 4, \{2\}^b)}
	\mathllap{\zeta^{\mot,\star}(\{2\}^a, 4, \{2\}^b)} =  {}
	& \begin{aligned}[t]
	& 2 \zm(\overline{2a + 2b + 4}) 
	-4 (2a+1) \zm(\overline{2b+2}) \zm(2a+2) 
	\\
	& + 8 \zm(2a+1) \sum_{k=1}^{b+1} \zm(2k+1) \zm(\overline{2b+2-2k}) \\[-1ex]
	&  - 8 \sum_{s = 0}^{a} \zm(\overline{2s}) \zeta^{\mot}_{2a+1-2s}(1, \overline{2b+2}) - 8 \sum_{s = 0}^{b} \zm(\overline{2s}) \zeta^{\mot}_{2b+2-2s}(1, 2a+1) \,. \hspace{-1em}
	\end{aligned}		
	\end{aligned}
	\end{equation}	
\end{prop}

\begin{proof}
We compute \( \D_{2r+1} \) of both sides, and will show that they agree.  The analytic version of this identity, which is given in \autoref{eqn:zs2242full2} then fixes the remaining coefficient of \( \zm(2a+2b+4) \) (which here is expressed as a sum of two terms involving products of even zetas).  Notice that the only place \( \D_1 \) could contribute is from \( \zm_{2a+1}(1, \overline{2b+2}) \), but we will see momentarily that \( \D_1 = 0 \), hence we can take \( r > 0 \). \medskip

\paragraph{\bf Computation of \( \D_{2r+1} \zm_{2a+1}(1, \overline{2b+2}) \)}  We see that only the following subsequences can contribute to the motivic coaction.  This is because any subsequence must start or end at one of the three non-zero entries; one then checks whether the length \( 2r+1 \) subsequences which start/end at these points actually contribute.
\vspace{-0.5em}
\begin{center}
		\begin{tikzpicture}
\matrix[name=M1, matrix of nodes, inner sep=1pt, column sep=0pt]{
	\node () [] {$\zeta_{2a+1}(1,\overline{2b+2}) = I^\mot($};
	& \node (s1) [] {${\phantom{\mathllap{\fbox{$,0$}}}} 0$}; 
	& \node () [] {$;$}; 
	& \node (zm1) [] {${\phantom{\mathllap{\fbox{$,0$}}}} \{0\}^{2a+1} , $};
	& \node (m1) [] {${\phantom{\mathllap{\fbox{$,0$}}}}  {-}1$};  			& \node () [] {$,$};
	& \node (m2) [] {${\phantom{\mathllap{\fbox{$,0$}}}}  {-}1$};  & \node () [] {$,$}; 
	& \node (zm2) [] {${\phantom{\mathllap{\fbox{$,0$}}}} \{0\}^{2b+1} , $};
	& \node (t1) [] {${\phantom{\mathllap{\fbox{$,0$}}}}  1$};  
	& \node () [] {$) \,.$}; 
	\\
};
\draw[thick,black] (m1.north) to[bracket=5pt] node[above=-1mm] {$$} (zm2.north);
\draw[thick,black] (m2.south) to[ubracket=5pt] node[below=-1mm] {$$} (zm1.south);
\draw[black,thin,solid] ($(m1.north west)+(0,0)$)  rectangle ($(m1.south east)+(0,0)$);
\draw[black,thin,solid] ($(m2.north west)+(0,0)$)  rectangle ($(m2.south east)+(0,0)$);
\draw[black,thin,solid] ($(t1.north west)+(0,0)$)  rectangle ($(t1.south east)+(0,0)$);
\end{tikzpicture}
\end{center}
\vspace{-1em}
Hence we have
\begin{align*}
	& \D_{2r+1} \zm_{2a+1}(1,\overline{2b+2}) \\
	& {} = \delta_{r\leq a} \zeta^\lmot_{2r}(1) \otimes \zm_{2a+1-2r}(\overline{2b+2}) - \delta_{r \leq b} \zeta^\lmot_{2r}(1) \otimes \zm_{2a+1}(\overline{2b+2-2r}) \\
	& {} = \begin{aligned}[t]
	\zeta^\lmot(2r+1) \otimes \bigg( \! {-} \delta_{r\leq a} \binom{2a + 2b + 2 - 2r}{2b + 1} + \delta_{r \leq b} \binom{2a + 2b + 2 - 2r}{2a} \!\! \bigg) \\
	{} \cdot \zm(\overline{2a + 2b+ 3 - 2r}) \,. \hspace{-3em}
	\end{aligned}
\end{align*}
And in particular \( \D_1 = 0 \). \smallskip

\paragraph{\bf Computation of \( \D_{2r+1} \zm_{2b+2}(1, 2a+1) \)}  Similarly, only the following subsequences can contribute to the motivic coaction.
\vspace{-0.5em}
\begin{center}
	\begin{tikzpicture}
	\matrix[name=M1, matrix of nodes, inner sep=1pt, column sep=0pt]{
		\node () [] {$\zeta_{2b+2}(1,2a+1) = I^\mot($};
		& \node (s1) [] {${\phantom{\mathllap{\fbox{$,0$}}}} 0$}; 
		& \node () [] {$;$}; 
		& \node (zm1) [] {${\phantom{\mathllap{\fbox{$,0$}}}} \{0\}^{2b+2} , $};
		& \node (m1) [] {${\phantom{\mathllap{\fbox{$,0$}}}}  1$};  			& \node () [] {$,$};
		& \node (m2) [] {${\phantom{\mathllap{\fbox{$,0$}}}}  1$};  & \node () [] {$,$}; 
		& \node (zm2) [] {${\phantom{\mathllap{\fbox{$,0$}}}} \{0\}^{2a} , $};
		& \node (t1) [] {${\phantom{\mathllap{\fbox{$,0$}}}}  1$};  
		& \node () [] {$) \,.$}; 
		\\
	};
	\draw[thick,black] (m1.north) to[bracket=5pt] node[above=-1mm] {$$} (zm2.north);
	\draw[thick,black] (m2.south) to[ubracket=5pt] node[below=-1mm] {$$} (zm1.south);
	\draw[black,thin,solid] ($(m1.north west)+(0,0)$)  rectangle ($(m1.south east)+(0,0)$);
	\draw[black,thin,solid] ($(m2.north west)+(0,0)$)  rectangle ($(m2.south east)+(0,0)$);
	\draw[black,thin,solid] ($(t1.north west)+(0,0)$)  rectangle ($(t1.south east)+(0,0)$);
	\end{tikzpicture}
\end{center}\vspace{-1em}
Hence we have
\begin{align*}
& \D_{2r+1} \zm_{2b+2}(1,2a+1) \\
& {} = \delta_{r\leq b+1} \zeta^\lmot_{2r}(1) \otimes \zm_{2b+2-2r}(2a+1) - \delta_{r \leq a-1} \zeta^\lmot_{2r}(1) \otimes \zm_{2b+2}(2a+1-2r) \\
& {} = \begin{aligned}[t] \zeta^\lmot(2r+1) \otimes \bigg( \!{-} \delta_{r\leq b+1} \binom{2a+2b+2-2r}{2a + 1} \!+\! \delta_{r \leq a-1} \binom{2a+2b+2 - 2r}{2b + 1} \!\! \bigg) \\ 
{} \cdot \zm(2a + 2b+ 3 - 2r) \,.\hspace{-3em} \end{aligned}
\end{align*}

\medskip
\paragraph{\bf Computation of \( \D_{2r+1} \) of right hand side of \autoref{eqn:zs2242full2mot}}
With the above two computations of the motivic coaction on the double zeta values in \autoref{eqn:zs2242full2mot}, we can readily compute the rest of the coaction using the derivation property of \( \D_{2r+1} \), namely \( \D_{2r+1} X Y = (1 \otimes Y) \D_{2r+1}X + (1 \otimes X) \D_{2r+1} Y \), as well as the fact that \( \D_{2r+1} \zm(N) = \delta_{N=2r+1} \zeta^\lmot(N) \otimes 1 \). Note also the first two terms on the right hand side of \autoref{eqn:zs2242full2mot} are products of even zetas, and so do not contribute.  So we find
\begin{align*}
	 & \D_{2r+1} (\text{RHS \autoref{eqn:zs2242full2mot}}) = {} \\
	 & 8 \delta_{r=a} \zeta^\lmot(2r+1) \otimes \sum_{k=1}^{b+1} \zm(2k+1) \zm(\overline{2b+2-2k}) \\
	 & + 8 \delta_{r \leq b+1} \zeta^\lmot(2r+1) \otimes \zm(2a+1) \zm(\overline{2b + 2 - 2r}) \\
	 &  - 8 \sum_{s = 0}^{a}  \begin{aligned}[t] 
	 \zeta^\lmot(2r+1) \otimes \bigg( {-} 
	 \delta_{r\leq a-s} \binom{2a + 2b + 2 - 2r - 2s}{2b+1}
	  + & \delta_{r \leq b} \binom{2a + 2b + 2 - 2r - 2s}{2b + 1 - 2r} \! \bigg) \\ 
	 & \cdot \zm(\overline{2a - 2s + 2b + 3 - 2r}) \zm(\overline{2s}) \end{aligned} \\
	 & - 8 \sum_{s=0}^b \begin{aligned}[t]
	 	\zeta^\lmot(2r+1) \otimes \bigg( \delta_{r\leq b-s+1} \binom{2a + 2b + 2 - 2r - 2s}{2a} & - \delta_{r \leq a-1} \binom{2a + 2b + 2 - 2r - 2s}{2a - 2r} \! \bigg) \\ 
	 	& \cdot \zm(2a + 2b - 2s + 3 - 2r) \zm(\overline{2s})\,.
	 	\end{aligned}
\end{align*}

\medskip
\paragraph{\bf Computation of \( \D_{2r+1} \) of left hand side of \autoref{eqn:zs2242full2mot}}

We compute the derivation \( \D_{2r+1} \zeta^{\mot,\star}(\{2\}^a, 4, \{2\}^b) \) by first applying the stuffle antipode to obtain an expression involving only \( \zm(\{2\}^j, 4, \{2\}^i) \), which has a simpler coaction.

We see that only the following subsequences can contribute to the motivic coaction of \( \D_{2r+1} \zm(\{2\}^a, 4, \{2\}^b) \); all other subsequences will start and end at letters of the same parity.
\vspace{-1em}
\begin{center}
	\begin{tikzpicture}
	\matrix[name=M1, matrix of nodes, inner sep=1pt, column sep=0pt]{
		\node () [] {$\zm(\{2\}^a,4,\{2\}^b) = (-1)^{a+b+1} I^\mot($};
		& \node (s1) [] {${\phantom{\mathllap{\fbox{$,0$}}}} 0$}; 
		& \node () [] {$;$}; 
		& \node () [] {${\phantom{\mathllap{\fbox{$,0$}}}} \{\!$};
		& \node (zm1) [] {${\phantom{\mathllap{\fbox{$,0$}}}} 1\!$};
		& \node () [] {${\phantom{\mathllap{\fbox{$,0$}}}} ,0\}^{a} ,$};	
		& \node () [] {${\phantom{\mathllap{\fbox{$,0$}}}}  1, 0$};  & \node () [] {$,$}; 
		& \node (m1) [] {${\phantom{\mathllap{\fbox{$,0$}}}} 0$};   & \node () [] {$,$}; 
		& \node () [] {${\phantom{\mathllap{\fbox{$,0$}}}}  0$};  & \node () [] {$,$}; 
		& \node () [] {${\phantom{\mathllap{\fbox{$,0$}}}} \{\!$};
		& \node (zm2) [] {${\phantom{\mathllap{\fbox{$,0$}}}} 1\!$};
		& \node () [] {${\phantom{\mathllap{\fbox{$,0$}}}} ,0\}^{b},$};	
		& \node (t1) [] {${\phantom{\mathllap{\fbox{$,0$}}}}  1$};  
		& \node () [] {$) \,.$}; 
		\\
	};
	\draw[thick,black] (m1.north) to[bracket=5pt] node[above=-1mm] {$$} (zm2.north);
	\draw[thick,black] (m1.south) to[ubracket=5pt] node[below=-1mm] {$$} (zm1.south);
	\draw[black,thin,solid] ($(m1.north west)+(0,0)$)  rectangle ($(m1.south east)+(0,0)$);
	\end{tikzpicture}
\end{center}
\vspace{-1em}
Hence we have
\begin{align*}
	& \D_{2r+1} \zm(\{2\}^a, 4, \{2\}^b) = \\
	& - \delta_{r \leq a} \zeta^\lmot_1(\{2\}^r) \otimes \zm(\{2\}^{a-r}, 3, \{2\}^b)  + \delta_{r \leq b} \zeta^\lmot_1(\{2\}^r) \otimes \zm(\{2\}^a, 3, \{2\}^{b-r}) \,.
\end{align*}
Now with the stuffle antipode formula extracted from \autoref{eqn:gs:zs2242asz2242}, we compute
\begin{align*}
	& D_{2r+1 }\zeta^{\mot,\star}(\{2\}^a, 4, \{2\}^b) \\
	& = \sum_{i=0}^a \sum_{j=0}^b (-1)^{i+j} \D_{2r+1} \zm(\{2\}^j, 4, \{2\}^i) \Big( 1 \otimes \zeta^{\mot,\star}(\{2\}^{a-i}) \zeta^{\mot,\star}(\{2\}^{b-j}) \Big) \\
	&= \zeta^\lmot_1(\{2\}^r) \otimes \sum_{i=0}^a \sum_{j=0}^b (-1)^{i+j} \begin{aligned}[t] \big( - \delta_{r \leq j}\zm(\{2\}^{j-r}, 3, \{2\}^i)  + \delta_{r \leq i} \zm(\{2\}^j, 3, \{2\}^{i-r}) \big)  \\ {} \cdot \zeta^{\mot,\star}(\{2\}^{a-i}) \zeta^{\mot,\star}(\{2\}^{b-j})\,.
	\end{aligned}
\end{align*}
Here, we can either apply the motivic evaluation of \( \zm(\{2\}^\alpha, 3, \{2\}^\beta) \) established by Brown \cite{BrownMTM12}.  Alternatively we can apply the stuffle antipode again to rewrite the result instead involving \( \zeta^{\mot,\star}(\{2\}^\alpha, 3, \{2\}^\beta) \) and appeal to the motivic evaluation thereof, for a more direct formula.  (Glanois \cite{GlanoisThesis16} claims that the motivic evaluation of \( \zeta^{\mot,\star}(\{2\}^\alpha, 3, \{2\}^\beta) \) requires knowing exactly certain conjectural identities amongst so-called \( \zeta^{\star\star} \) values, however it seems that the stuffle antipode formula allows one to automatically transfer the \( \zeta^{\mot}(\{2\}^\alpha, 3, \{2\}^\beta) \) evaluation to a corresponding \( \zeta^{\mot,\star}(\{2\}^\beta, 3, \{2\}^\alpha) \) evaluation.)

After (separately) shifting \( i \) and \( j \) by \( r \), which gives the factor \( (-1)^r \) below (and taking care with the signs; use the correspondence \( j \leftrightarrow b \), \( i \rightarrow a \)), we find
\begin{align*}
& \D_{2r+1 }\zeta^{\mot,\star}(\{2\}^a, 4, \{2\}^b) \\
&= (-1)^{r} \zeta^\lmot_1(\{2\}^r) \otimes \big( \delta_{r\leq a} \zeta^{\mot,\star}(\{2\}^{a-r}, 3, \{2\}^{b}) - \delta_{r \leq b} \zeta^{\mot,\star}(\{2\}^a, 3,\{2\}^{b-r}) \big) \,.
\end{align*}
We note that this is essentially the same expression as one obtains with Glanois' setup involving the motivic coaction on \( \zeta^{\star} \) values, after applying the dihedral symmetries to simplify terms in the coalgebra on the left hand side.  One only needs to apply the result that \( (-1)^r  \zeta^\lmot_1(\{2\}^r) = 2 \zeta^\lmot(2r+1) = -\zeta^{\lmot,\star}_1(\{2\}^r) \), to obtain exactly the same formula.

Now apply the following motivic evaluations
\begin{align*}
	\zeta^\lmot_1(\{2\}^r) & = 2 (-1)^r \zeta^\lmot(2r+1)  \\
	\zeta^{\mot,\star}(\{2\}^a, 3, \{2\}^b) & = - 2 \! \sum_{s=1}^{a+b+1} \! \bigg[ \! \binom{2s}{2a} {-} \delta_{s=a} - (1 - 2^{-2s}) \binom{2s}{2b+1} \! \bigg] \zeta^{\star,\mot}(\{2\}^{a+b+1-s}) \zm(2s+1)
\end{align*}
along with
\(
	\zeta^{\star,\mot}(\{2\}^n) = -2 \zm(\overline{2n}) \,.
\)
We find
\begin{align*}
	& \D_{2r+1} \zeta^{\mot,\star}(\{2\}^a, 4, \{2\}^b) = \\
	& \begin{aligned}[t]
	&	 8 \zeta^\lmot(2r+1) \otimes \sum_{s=1}^{a+b+1-r} \begin{aligned}[t] \bigg[ \binom{2s}{2a-2r} - & \delta_{s=a-r} - (1-2^{-2s}) \binom{2s}{2b+1} \bigg] \\ & \cdot \zm(\overline{2a + 2b + 2 - 2s - 2r}) \zm(2s+1) \end{aligned} \\
	&	- 8 \zeta^\lmot(2r+1) \otimes \sum_{s=1}^{a+b+1-r} \begin{aligned}[t] \bigg[ \binom{2s}{2a} - & \delta_{s=a} - (1-2^{-2s}) \binom{2s}{2b-2r+1} \bigg] \\ & \cdot \zm(\overline{2a + 2b + 2 - 2s - 2r}) \zm(2s+1) \,. \end{aligned} \\
	\end{aligned}
\end{align*}

\medskip
\paragraph{\bf Comparison of left and right hand side of \autoref{eqn:zs2242full2mot}}
Firstly, make the change of variables \( s \mapsto a + b + 1 - s - r \) in the sums for \( \D_{2r+1} ( \text{RHS \autoref{eqn:zs2242full2mot}} ) \); after considering the cases in each resulting delta term -- and dropping terms \( \zm(1) = 0 \) by regularisation-- we find
\begin{align*}
& \D_{2r+1} (\text{RHS \autoref{eqn:zs2242full2mot}}) = {} \\
& 8 \delta_{r=a} \zeta^\lmot(2r+1) \otimes \sum_{k=1}^{b+1} \zm(2k+1) \zm(\overline{2b+2-2k}) \\
& + 8 \delta_{r \leq b+1} \zeta^\lmot(2r+1) \otimes \zm(2a+1) \zm(\overline{2b + 2 - 2r}) \\
&  - 8 \sum_{s = \max(1, b-r+1)}^{a+b+1-r}  \begin{aligned}[t] 
\zeta^\lmot(2r+1) \otimes \bigg( {-} \binom{2s}{2b+1} + & \binom{2s}{2b+1-2r} \bigg) \\ 
& \cdot \zm(\overline{2s+1}) \zm(\overline{2a + 2b + 2 - 2r - 2s)}) \end{aligned} \\
& - 8 \sum_{s=\max(1, a-r+1)}^{a+b+1-r} \begin{aligned}[t]
\zeta^\lmot(2r+1) \otimes \bigg( \binom{2s}{2a} + & \delta_{r=a} - \binom{2s}{2a-2r} \bigg) \\ 
& \cdot \zm(2s + 1) \zm(\overline{2a + 2b + 2 - 2r - 2s)}) \,.
\end{aligned}
\end{align*}
Note here that the two terms involving \( \delta_{r=a} \) cancel.  Then the sums over  \( s \) may be extended to start at \( s = 1 \).  The first sum needs no correction term, as the numerators of each binomial are strictly greater than the denominators in this case, however the term \( \binom{2s}{2a-2r} \) in the second sum needs to be corrected when \( s = a-r \) for \( a-r \geq 1 \).  We obtain
\begin{align*}
& \D_{2r+1} (\text{RHS \autoref{eqn:zs2242full2mot}}) = {} \\
& 8 \delta_{r \leq b+1} \zeta^\lmot(2r+1) \otimes \zm(2a+1) \zm(\overline{2b + 2 - 2r}) \\
& - 8 \delta_{r \leq a-1} \zeta^\lmot(2r+1) \otimes \zm(2a + 1 - 2r) \zm(\overline{2b+2}) \\
&  - 8 \sum_{s = 1}^{a+b+1-r}  \begin{aligned}[t] 
\zeta^\lmot(2r+1) \otimes \bigg( {-} \binom{2s}{2b+1} + & \binom{2s}{2b+1-2r} \! \bigg) \\ 
& \cdot \zm(\overline{2s+1}) \zm(\overline{2a + 2b + 2 - 2r - 2s)}) \end{aligned} \\
& - 8 \sum_{s=1}^{a+b+1-r} \begin{aligned}[t]
\zeta^\lmot(2r+1) \otimes \bigg( \! \binom{2s}{2a} & - \binom{2s}{2a-2r} \! \bigg) \\ 
& \cdot \zm(2s + 1) \zm(\overline{2a + 2b + 2 - 2r - 2s)}) 
\end{aligned}
\end{align*}
Finally write \( \zm(\overline{2s+1}) = -(1-2^{2s})\zm(2s+1) \).  It is now straightforward to check that \( \D_{2r+1} ( \text{LHS \autoref{eqn:zs2242full2mot}} ) = \D_{2r+1} ( \text{RHS \autoref{eqn:zs2242full2mot}} ) \); the two terms outside the sum for \( \D_{2r+1} ( \text{RHS \autoref{eqn:zs2242full2mot}} ) \) above correspond to the deltas terms in the expression for \( \D_{2r+1}( \text{LHS \autoref{eqn:zs2242full2mot}} ) \).
\medskip

This completes the proof of \autoref{prop:zsasdp3:mot}, and shows the reduction of \( \zeta^{\star}(\{2\}^a, 4, \{2\}^b) \) to depth 3 alternating MZVs is motivic.
\end{proof}

\subsection{\texorpdfstring{Motivic evaluation of \( \zeta^\lmot(\{2\}^a, 4, \{2\}^b) \)}{Motivic evaluation of zeta\textasciicircum{}l(\{2\}\textasciicircum{}a, 4, \{2\}\textasciicircum{}b)}}

Now that we have verified all of the ingredients for the evaluations of \( \zeta(\{2\}^a, 4, \{2\}^b) \) and \( \zeta^\star(\{2\}^a, 4, \{2\}^b) \) are motivic, we may conclude that the identities in \autoref{thm:eval:z2242} and \autoref{thm:eval:zs2242} hold for \( \zeta^{(\star)} \) replaced by their motivic counterparts, and \( \ii\pi \) replaced by \( \frac{1}{2} \mathbb{L}^\mot = (\ii\pi)^\mot \).

More importantly, the evaluation of \( \zeta(\{2\}^a, 4, \{2\}^b) \) modulo products from \autoref{cor:eval:z2242prod} is also motivic, and we obtain the result of \autoref{lem:eval:z2242prodmot} as an immediate corollary.

\begin{cor}\label{cor:eval:z2242prodmot}
	The following evaluation holds in the motivic coalgebra
	\begin{align*}
	&  \zeta^\lmot(\{2\}^a, 4, \{2\}^b) = \\
	& (-1)^{a+b} \begin{aligned}[t] \bigg\{  & {-} 4 \zeta^\lmot(2a+2, 2b+2) +  4 \zeta^\lmot(2b+1, 2a+3)
	\\[-1ex] & 
	+ \sum_{\crampedsubstack{i + j = 2a + 2b \\ i, j \geq 0 }} \bigg(  \frac{1}{2^i} \binom{i+1}{2a+1} + \frac{1}{2^j} \binom{j+1}{2b+1} \! \bigg) \zeta^\lmot(i+2, j+2) \bigg\} 
	 \,. \end{aligned}
	\end{align*}
\end{cor}

\end{document}